\documentclass[a4paper]{amsart}
\usepackage[utf8]{inputenc}
\usepackage{amssymb,amsfonts, stmaryrd, amsmath, amsthm, enumerate}

\usepackage{fullpage}

\usepackage{soul}

\usepackage{tensor}

\usepackage[all,cmtip]{xy}

\usepackage{wrapfig,tikz, tikz-cd}
\usetikzlibrary{arrows}

\numberwithin{equation}{section}

\newtheorem{theorem}{Theorem}[section]
\newtheorem{lemma}[theorem]{Lemma}

\newtheorem{corollary}[theorem]{Corollary}

\theoremstyle{definition}
\newtheorem{example}[theorem]{Example}
\newtheorem{remark}[theorem]{Remark}
\newtheorem{definition}[theorem]{Definition}

\newtheorem*{maintheorem}{Main Theorem}

\newcommand{\A}{\mathbb{A}}
\newcommand{\R}{\mathbb{R}}

\newcommand{\Z}{\mathbb{Z}}
\newcommand{\N}{\mathbb{N}}
\newcommand{\C}{\mathbb{C}}
\newcommand{\HH}{\mathbb{H}}

\newcommand{\Sp}{\mathbb{S}} 
\newcommand{\ideal}[1]{\langle #1 \rangle} 
\newcommand{\norm}[1]{\lVert #1 \rVert}

\DeclareMathOperator{\Hom}{Hom}
\DeclareMathOperator{\Ker}{Ker}

\DeclareMathOperator{\Ima}{Im}
\DeclareMathOperator{\Aut}{Aut}

\DeclareMathOperator{\diag}{diag}

\begin{document}
\title{Semilinear clannish algebras}

\author{Raphael Bennett-Tennenhaus}
\address{Fakult\"at f\"ur Mathematik, Universit\"at Bielefeld, 33501 Bielefeld, Germany}
\email{raphaelbennetttennenhaus@gmail.com}

\author{William Crawley-Boevey 
}
\address{Fakult\"at f\"ur Mathematik, Universit\"at Bielefeld, 33501 Bielefeld, Germany}
\email{wcrawley@math.uni-bielefeld.de}


\subjclass[2020]{Primary 16G10; Secondary 15A04, 16S36.}




\keywords{String algebra, Clannish algebra, Dedekind-like ring, Dieudonn\'e module}


\thanks{The authors have been supported by the Alexander von Humboldt Foundation in the framework of an Alexander von Humboldt Professorship endowed by the German Federal Ministry of Education and Research.}

\begin{abstract}
We define a class of associative algebras generalizing `clannish algebras', as introduced by the second author, but also incorporating semilinear structure, like a skew polynomial ring. Clannish algebras generalize the well known `string algebras' introduced by Butler and Ringel. Our main result is the classification of finite-dimensional indecomposable modules for these new algebras. 
\end{abstract}

\maketitle

\section{Introduction}
We define a class of associative algebras, which we call `semilinear clannish algebras', generalizing the  `clannish algebras' introduced by the second author in \cite{Cra1989}, but whose modules may also incorporate semilinear structure. Recall that for an automorphism $\sigma$ of a division ring $K$, a map $\theta:V\to W$ between left $K$-modules is \emph{$\sigma$-semilinear} if $\theta(\lambda v+ \lambda' v') = \sigma(\lambda) \theta(v)+\sigma(\lambda')\theta(v')$ for all $v,v'\in V$ and $\lambda,\lambda'\in K$. Clannish algebras generalize the well known `string algebras' introduced by Butler and Ringel \cite{ButRin1987}. In unpublished work \cite{Rin-lec}, Ringel considered representations of the corresponding semilinear generalization of string algebras.

Our main result is the classification of the finite-dimensional indecomposable modules for semilinear clannish algebras, under suitable hypotheses. As a special case we recover results Ringel claimed for semilinear string algebras. Recall that the finite-dimensional indecomposable modules for string algebras are classified into two types, strings and bands, indexed by certain `words', and in addition, the band modules depend on the choice of an indecomposable module for a Laurent polynomial ring $K[x, x^{-1}]$. For clannish and semilinear clannish algebras there is also a classification into strings and bands, but each of these classes divides into two subclasses, asymmetric and symmetric, and there are several replacements for $K[x,x^{-1}]$. 

Let $K$ be a division ring, let $Q$ be a finite quiver and let $\boldsymbol{\sigma}$ be a collection of automorphisms $\sigma_a$ of $K$ indexed by the arrows $a$ in $Q$. The \emph{semilinear path algebra} $K_{\boldsymbol{\sigma}} Q$ of $Q$ over $K$ is the left $K$-module with basis the paths in $Q$, including a trivial path $e_i$ for each vertex $i$, with multiplication twisted by the rule that $a\lambda = \sigma_a(\lambda) a$ for $a$ an arrow and $\lambda\in K$. Its modules correspond to semilinear representations of $Q$.

To define a semilinear clannish algebra, we fix a set $\Sp$ of loops in $Q$, which we call \emph{special loops}; other arrows are called \emph{ordinary arrows}. For each $s\in \Sp$ we fix a monic quadratic element $q_s(x) = x^2-\beta_s x + \gamma_s$ in the skew polynomial ring $K[x;\sigma_s]$. Let $Z$ be a set of paths in $Q$ of length at least 2, which will be `zero-relations'. We assume that no element of $Z$ starts or ends with a special loop, or has one special loop occurring twice consecutively. Let $R = K_{\boldsymbol{\sigma}} Q/I$ where $I$ is the ideal generated by $Z$ and elements of the form $s^2 - \beta_s s + \gamma_s e_i$ for each special loop $s\in \Sp$, say at vertex $i$. We say that $R$ is a \emph{semilinear clannish algebra} provided that the following conditions are satisfied.
\begin{itemize}
\item[(1)]
At most two arrows have tail at any vertex of $Q$.
\item[($1'$)]
At most two arrows have head at any vertex of $Q$.
\item[(2)]
For any ordinary arrow $a$, there is at most one arrow $c$ with $ca$ a path not in $Z$.
\item[($2'$)]
For any ordinary arrow $a$, there is at most one arrow $c$ with $ac$ a path not in $Z$.
\end{itemize}

When $K$ is commutative, so a field, and all automorphisms in $\boldsymbol{\sigma}$ are trivial, this recovers the notion of a clannish algebra \cite{Cra1989}; when there are no special loops we call it a \emph{semilinear string algebra} (Ringel \cite{Rin-lec} also assumed that $K$ is commutative); when all of these restrictions hold, one recovers the notion of a string algebra (but without the finiteness conditions (3) and (3*) of \cite[p. 157]{ButRin1987}, so for example corresponding to string algebras as defined in \cite{Cra2018} given by a finite quiver). 

We say that a semilinear clannish algebra is
\begin{itemize}
\item[(i)]\emph{normally-bound} if $q_s(x)$ is normal in $K[x;\sigma_s]$ for all $s\in\Sp$. 
Recall that an element of a ring is \emph{normal} if the left and right ideals that it generates are equal. By Lemma~\ref{lemma:characterising_central_skew_quadratics}, it is equivalent that $\sigma_s(\beta_s)=\beta_s$ and $\sigma_s(\gamma_s)=\gamma_s$, and that $\beta_{s} \lambda= \sigma_{s}(\lambda) \beta_{s}$ and $\gamma_{s}\lambda =  \sigma^{2}_{s}(\lambda)\gamma_{s}$ for all $\lambda\in K$;

\item[(ii)]\emph{of non-singular type} if $q_s(x)$ is non-singular for all $s\in \Sp$.
Here we say that a polynomial $p(x) \in K[x;\sigma]$ is \emph{non-singular} if it has a non-zero constant term. Thus the condition is that $\gamma_s\neq 0$;

\item[(iii)]\emph{of semisimple type} if $q_s(x)$ is semisimple for all $s\in \Sp$. Here we say that a polynomial $p(x) \in K[x;\sigma]$ is \emph{semisimple} if the factor ring
$K[x;\sigma]/\ideal{p(x)}$ is a semisimple artinian ring. By Lemma~\ref{lemma:characterisingsemisimple-quadratics}, if $q_s(x)$ is normal, it is equivalent that it is not of the form $(x-\eta)^2$ with $\eta\in K$ and $\sigma(\lambda)\eta=\eta\lambda$ for all $\lambda\in K$.
\end{itemize}
See Remark~\ref{remark:condsonpolys} for a discussion of the necessity of these conditions. In Section~\ref{s:words} we define the notion of a `word', an equivalence relation on words, and sets of `strings' and `bands' which are unions of equivalence classes of words. For each string or band $w$, in Sections \ref{s:asymstring}--\ref{s:symband} we define a ring $R_w$ equipped with a ring homomorphism $K\to R_w$ and an $R$-$R_w$-bimodule $M(C_w)$, finitely generated and free as a right $R_w$-module.

\begin{maintheorem}
Let $R$ be a semilinear clannish algebra which is normally  bound, of non-singular type and of semisimple type. As $w$ runs through representatives of the equivalence classes of strings and bands and as $V$ runs through a complete set of non-isomorphic indecomposable $R_w$-modules, finite-dimensional over $K$, the modules $M(C_w)\otimes_{R_w} V$ run through a complete set of non-isomorphic indecomposable $R$-modules, finite-dimensional over $K$.
\end{maintheorem}

Under the stated conditions, if $w$ is a string, then $R_w$ is a semisimple artinian ring, and if $w$ is a band, then $R_w$ is a hereditary noetherian prime ring. 

The classifications in \cite{Cra1988ff} and \cite{Cra1989} are proved using the so-called `functorial filtration method', which goes back to Gelfand and Ponomarev \cite{GelPon1968}, essentially for modules for the string algebra $K[x,y]/\ideal{xy}$ (for a field $K$), and was adapted to the string algebra $K\langle x,y\rangle/\ideal{x^2,y^2}$ by Ringel \cite{Rin1975}. The method involves certain functorially-defined subspace filtrations on a module built from linear relations. Compatibility conditions are then checked between the filtrations and a list of indecomposables. These conditions form \cite[Lemma, p. 22]{Rin1975}, part (iii) of which is a `mapping property'. This property was verified in \cite[FF4]{Cra1989} using certain `splitting lemmas' written in terms of relations.

We adapt the method to our context by considering relations which are \emph{semilinear}, the prototypical example being the graph of a semilinear map. Additionally, certain subquotients of the functorial filtrations discussed above are realised as factors of the functor $\Hom_R(M(C_{w}),-)$. By writing our splitting lemmas in terms of this Hom-functor, we simplify the verification of the mapping property (see Lemma \ref{lemma:existsgamma}). 

To verify part (i) of \cite[Lemma, p. 22]{Rin1975}, the aforementioned subquotients are evaluated on each member of the given list of indecomposables. For clannish algebras this used certain `orientation results', see for example \cite[\S 4.1]{Cra1988ff}. Here we both generalise and simplify these orientation results.

%
We now give some examples of semilinear clannish algebras; see \S\ref{s:examples} for more details. The special case of string algebras, and especially so-called `gentle algebras', has attracted much interest, see for example \cite{AAG,BC,GLFS,LP,PPP,Schroll}.
%
%
%
If $K$ is a perfect field of characteristic $p>0$, and $\sigma$ is its Frobenius automorphism, then there is a semilinear string algebra whose modules are given by a $K$-vector space equipped with a $\sigma$-semilinear endomorphism $F$ and a $\sigma^{-1}$-semilinear endomorphism $V$, satisfying $FV=VF=0$. These are exactly Dieudonn\'e modules annihilated by $p$.
%
%
De Oliveira, Futorny, Klimchuk, Kovalenko and Sergeichuk \cite{Oliveira2013} have studied semilinear representations of a quiver of finite or extended Dynkin type $\A$ over $\C$, such that the automorphism of $\C$ associated to an arrow is either the identity or complex conjugation. This is a special case of a semilinear string algebra.

The clannish algebra $K\langle a,e\rangle /\ideal{a^2,e^2-e}$ is studied in \cite{Cra1988ff}; the classification of its finite-dimensional modules is exactly the classification of an idempotent matrix and a square-zero matrix up to simultaneous similarity. But the quadratic polynomial for the special loop $e$ is $q_e(x)=x^2-x$, which has constant term zero, so, like \cite{Cra1988ff,Cra1989}, our theory does not apply to this algebra. 
%
However, as explained in \cite{Cra1988ff}, if the field $K$ has more than two elements, and $\mu \in K\setminus\{0,1\}$, then replacing the generator $e$ by $t = e-\mu 1$, we obtain the relation $(t+\mu)(t+\mu-1) = 0$, giving a presentation of the algebra to which our theory does apply; see also Remark~\ref{remark:condsonpolys}(ii).
%
%
So-called `skewed-gentle algebras', see \cite{GdlP}, are a special class of clannish algebras, analogous to gentle algebras. 
Another example of a clannish algebra is $\R\langle a,t \rangle/\ideal{a^2,t^2+1}$, the free product over $\R$ of the ring of dual numbers $\R[a]/\ideal{a^2}$ and the field of complex numbers $\C$. Here $t$ is a special loop with $q_t(x)=x^2+1$, which is irreducible over $\R$, and hence this algebra is not covered by the classification in \cite{Cra1989}. 
It was mentioned there that the results may remain true for irreducible quadratics, provided that the splitting lemma held true. In this paper we prove that splitting does hold, so this example is covered by our present theory.

Another example of a semilinear clannish algebra is the ring $R$ whose modules are given by a $\C$-vector space equipped with conjugate-linear endomorphisms $a$ and $t$, satisfying $a^2=0$ and $t^2=-1$. This is a semilinear clannish algebra with $K=\C$, but because of the conjugate-linear endomorphisms, $\C$ is not central. However, it is an algebra in the usual sense over $\R$, and can be presented as
\[
R = \R \langle i,a,t\rangle/\ideal{i^2+1,a^2,t^2+1,ai+ia,ti+it}.
\]
In fact $R\cong M_2(\R[x,y]/\ideal{x^2+y^2})$, so our classification implicitly gives a classification of finite-dimensional indecomposable modules for $\R[x,y]/\ideal{x^2+y^2}$.

One difference with \cite{Cra1989}, is that that paper does not contain an explicit list of indecomposable modules for clannish algebras; rather, it explains how to convert the classification problem into a certain type of matrix problem called a `clan', and then gives a classification of representations of clans. In this paper, on the other hand, we give a classification of the indecomposable modules for semilinear clannish algebras directly, along the lines followed in \cite{Cra1988ff} for an idempotent and a square zero matrix, avoiding matrix problems. This is perhaps more convenient in applications. Going in the reverse direction, one can easily define the notion of a `semilinear clan' and obtain the classification of its indecomposable representations from those of a suitable semilinear clannish algebra, using the discussion in \cite[\S2.5]{Cra1989}.

\section{Preliminaries}
\subsection{Semilinear path algebras}
\label{s:basicdefns}
Let $K$ be a division ring and let $\sigma$ be an automorphism of $K$. If $V$ is a left $K$-module, we write ${}_\sigma V$ for its restriction via $\sigma$. We call it a \emph{twist} of $V$. Clearly a $\sigma$-semilinear map $V\to W$ is the same thing as a $K$-module homomorphism $V\to {}_\sigma W$. Note that a collection of elements of $V$ is a $K$-basis of $V$ if and only if it is a $K$-basis of ${}_{\sigma} V$, so $V$ and its twists have the same dimension.

Recall that a \emph{$K$-ring} is a ring $R$ equipped with a ring homomorphism $K\to R$; we identify $K$ with its image in $R$. Let $Q = (Q_0,Q_1,h,t)$ be a finite quiver, where $Q_0$ and $Q_1$ are the sets of vertices and arrows, and $h,t:Q_1\to Q_0$ give the head and tail of each arrow. Let $\boldsymbol{\sigma} = (\sigma_a)_{a\in Q_1}$ be a collection of automorphisms of~$K$.  The \emph{semilinear path algebra} $K_{\boldsymbol{\sigma}} Q$ is the $K$-ring generated by elements $e_i$ ($i\in Q_0$) and arrows $a\in Q_1$ subject to the relations
\[
e_i e_j = \begin{cases} e_i & (i=j) \\ 0 & (i\neq j),\end{cases}
\quad
\sum_{k\in Q_0} e_k = 1,
\quad 
e_{h(a)} a = a,
\quad
a e_{t(a)} = a,
\quad
e_i \lambda = \lambda e_i,
\quad
a \lambda = \sigma_a(\lambda) a
\]
for $i,j\in Q_0$, $a\in Q_1$ and $\lambda\in K$. Equivalently, $K_{\boldsymbol{\sigma}} Q$ is the tensor ring over the ring $S = K^{Q_0}$ of the $S$-$S$-bimodule 
\[
\bigoplus_{a\in Q_1} {}_{\pi_{h(a)}} K_{\sigma_a\pi_{t(a)}}
\]
where $\pi_i$ is the projection from $S$ to the $i$th copy of $K$, and the notation means that $K$ is considered as a left $S$-module by restriction via $\pi_{h(a)}$ and as a right $S$-module by restriction via $\sigma_a\pi_{t(a)}$. 

For any path $p$ we define an automorphism $\sigma_p$ of $K$ as follows. For a trivial path $e_i$ it is the identity, for an arrow $a$ it is $\sigma_a$ and for a path $a_1\dots a_n$, with each $a_{i}$ an arrow, it is $\sigma_{a_1}\dots \sigma_{a_n}$. 
It follows that $K_{\boldsymbol{\sigma}} Q$ is the left $K$-module with basis the paths in $Q$, where $e_i$ corresponds to the trivial path at vertex $i$,
and the multiplication satisfies
\[
\lambda p \cdot \mu q = \lambda \sigma_p(\mu) p q \quad (\lambda,\mu\in K,\text{ paths }p,q\text{ in }Q).
\]
The category of left modules for $K_{\boldsymbol{\sigma}} Q$ is equivalent to the category of semilinear representations of $Q$. Here a semilinear representation of $Q$ is a tuple $V = (V_i,V_a)$ consisting of a left $K$-module $V_i$ for each $i\in Q_0$, and a $\sigma_a$-semilinear map $V_a:V_i\to V_j$ for each arrow $a:i\to j$; and a morphism $\theta:V\to W$ of semilinear representations is a tuple $\theta = (\theta_i)$ consisting of a $K$-linear map $\theta:V_i\to W_i$ for each vertex $i\in Q_0$, satisfying $\theta_{h(a)} V_a = V_a \theta_{t(a)}$ for each arrow $a\in Q_1$. 
To a $K_{\boldsymbol{\sigma}} Q$-module $M$ corresponds the representation $V$ consisting of the $K$-modules $V_i = e_i M$ for $i\in Q_0$, with the map $V_a$ for $a\in Q_1$ given by the action of~$a$.

As an example, note that the semilinear path algebra for the quiver with one vertex, a loop $x$ and $\sigma_x=\sigma$ is the skew polynomial ring $K[x;\sigma]$ with $x\lambda = \sigma(\lambda)x$.
%
Note that semilinear representations of a quiver are nothing new: they are a special case of representations of a species \cite[\S7.4]{GabrielirII} (but without the nilpotence condition, hence without the need to complete the tensor algebra) or of a realization of a valued graph \cite{DlabRin1976}.

\subsection{Quadratic polynomials}
\label{s:quadpolyns}
Let $\sigma$ be an automorphism of the division ring $K$, and let $R = K[x;\sigma]$ be the skew polynomial ring. Recall that we say that a polynomial is \emph{non-singular} if its constant term is non-zero, and that an element $r$ in $R$ is said to be \emph{normal} in $R$ if $r R=R r$.

\begin{lemma}
\label{lemma:characterising_central_skew_quadratics}
A monic quadratic polynomial $q(x) = x^2-\beta x + \gamma$ in $R = K[x;\sigma]$ (with $\beta,\gamma\in K$) is
\begin{itemize}
\item[(i)]normal in $R$ if and only if $\sigma(\beta) = \beta$, $\sigma(\gamma) = \gamma$, and $\sigma(\lambda)\beta = \beta \lambda$ and $\sigma^2(\lambda)\gamma = \gamma\lambda$ for all $\lambda\in K$; and
\item[(ii)]central in $R$ if and only if in addition $\sigma^2 = 1$.
\end{itemize}
\end{lemma}

\begin{proof}
(i) If the conditions hold then $q(x)x = xq(x)$ and $q(x)\lambda = \sigma^2(\lambda)q(x)$ for $\lambda\in K$, so $q(x)$ is normal. Thus suppose that $q(x)$ is normal. We must have $q(x)x = p(x)q(x)$ for some polynomial $p(x)$, which by degree arguments must be of the form $a+bx$. The term in $x^3$ gives $b=1$ and the constant term gives $a\gamma=0$. The other terms give $\beta = \sigma(\beta)-a$ and $\gamma = \sigma(\gamma)-a\beta$. Now $\sigma(\beta) = \beta$ and  $\sigma(\gamma) = \gamma$ because either $a=0$ or, if $a\neq 0$, then $\gamma=0$, so also $a\beta=0$, so $\beta=0$. For $\lambda\in K$ we must also have $q(x)\lambda = r(x)q(x)$ for some polynomial $r(x)$, which by degree arguments must be a constant polynomial, and the term in $x^2$ gives $r(x)=\sigma^2(\lambda)$. It follows that $\gamma\lambda = \sigma^2(\lambda)\gamma$ and $\beta \sigma(\lambda) = \sigma^2(\lambda)\beta$, which gives the claim.

(ii) Clear. 
\end{proof}

Take the automorphism $\sigma$ (of order $4$) of the quarternions $\HH=\R\oplus\R i\oplus \R j\oplus \R k$ which conjugates by $1+i$. Then $x^{2}+2i$ is normal but not central in $\HH[x;\sigma]$. 
For later purposes we note the following.

\begin{corollary}
\label{corollary:twisting-(and-inverting-non-singular)-preserves-normal/central-quadratics}
For $q(x)=x^{2}-\beta x+\gamma$ in $K[x;\sigma]$ the following statements are equivalent.
\begin{enumerate}
    \item The polynomial $q(x)$ is normal (respectively, central) in $K[x;\sigma]$.
    \item For any $\phi \in  \Aut(K)$, $x^{2}-\phi(\beta)x+\phi(\gamma)$ is normal (respectively, central) in $K[x;\phi\sigma\phi^{-1}]$.
\end{enumerate}
If additionally $q(x)$ is non-singular, then (i) and (ii) are equivalent to saying $q'(x)=x^{2}-\gamma^{-1}\beta x+\gamma^{-1}$ is normal (respectively, central) in $K[x;\sigma^{-1}]$.
\end{corollary}

\begin{proof}
That (ii) implies (i) is trivial by taking $\phi$ to be the identity on $K$. That (i) implies (ii) follows from considering the extension of $\phi$ to a ring isomorphism $K[x;\sigma]\to K[x;\phi\sigma\phi^{-1}]$ which sends $ \lambda x^i$ to $ \phi(\lambda) x^i$. 
Now we assume $q(x)$ is non-singular and that (i) holds. Applying $\sigma^{-1}$ to the equations from Lemma \ref{lemma:characterising_central_skew_quadratics} yields 
equations which show that $q'(x)$ is normal (respectively, central), also by Lemma \ref{lemma:characterising_central_skew_quadratics}. The reverse implication follows by symmetry since the polynomial $q'$ is non-singular.
\end{proof}

The next result characterises when the quotient $S=K[x;\sigma]/\ideal{q(x)}$ by a normal monic quadratic $q(x)$ is a semisimple artinian ring.

\begin{lemma}
\label{lemma:characterisingsemisimple-quadratics}
Let $q(x) = x^2 - \beta x + \gamma$ be normal in $R=K[x;\sigma]$, and let
$S=R/\ideal{q(x)}$. Exactly one of the following four cases occurs.

\begin{itemize}
\item[(1)]$q(x)$ only factors trivially (through a constant polynomial) in $R$, and $S$ is a division ring;
        
\item[(2)]$q(x)=(x-\eta)(x-\mu)$ for $\eta,\mu\in K$ with $\sigma(\lambda)\eta\neq \eta\lambda$ for some $\lambda\in K$, and $S\cong M_{2}(D)$ for some division ring $D$;

\item[(3)]$q(x)=(x-\eta)(x-\mu)$ for distinct $\eta,\mu\in K$ with $\sigma(\lambda)\eta=\eta\lambda$ for all $\lambda\in K$, and $S\cong K\times K$; 

\item[(4)]$q(x)=(x-\eta)^{2}$ for $\eta\in K$ with $\sigma(\lambda)\eta=\eta\lambda$ for all $\lambda\in K$, and $S$ is not a semisimple ring.
\end{itemize}
\end{lemma}

\begin{proof}
Any $f(x)$ in $R$ can be written uniquely as $p(x)q(x) + r(x)$ for some $p(x),r(x)\in R$ with $r(x)$ of degree $\le 1$, so $1$ and $x$ give a basis for $S$ as a left or right $K$-module, and hence $S$ is artinian.

Suppose $q(x)$ only factorizes trivially. Then $S$ is a domain, for if a product of linear polynomials is zero in $S$, then the product must be a scalar multiple of $q(x)$, and hence $q(x)$ factorizes. It follows that $S$ is a division ring, so we have case (1). 

Thus we may suppose that $q(x)$ factorizes non-trivially, so $q(x)=f(x)g(x)$ for $f(x)=a_{1}x+a_{0}$ and $g(x)=b_{1}x+b_{0}$ where $a_{i},b_{i}\in K$ for $i=0,1$ and $a_{1}\neq 0\neq b_{1}$. Equating coefficients of $x^{2}$ gives $a_{1}\sigma(b_{1})=1$ and hence $q(x)=(x-\eta)(x-\mu)$ where $\eta=-a_{0}b_{1}$ and $\mu=-b_{1}^{-1}b_{0}$. Clearly exactly one of (2),(3),(4) holds. 

From $q(x)=(x-\eta)(x-\mu)$ we obtain that $\beta = \eta+\sigma(\mu)$ and $\gamma=\eta\mu$. Since $q(x)$ is normal, we have $\beta = \sigma(\beta) = \sigma^2(\mu)+\sigma(\eta)$. Also $\gamma\mu = \sigma^2(\mu)\gamma$, so $\eta\mu^2 = \sigma^2(\mu)\eta\mu$. Now considering the cases $\mu=0$ and $\mu\neq 0$ separately, we deduce that $\eta\mu = \sigma^2(\mu)\eta$. It follows that $q(x) = (x-\sigma^2(\mu))(x-\eta)$.

We have $R(x-\mu) = (R(x-\eta)+K)(x-\mu) = \ideal{q(x)} + K(x-\mu)$ and $R(x-\eta) = (R(x-\sigma^2(\mu)) + K) (x-\eta) = \ideal{q(x)} + K(x-\eta)$, so they define left ideals $I$ and $J$ of $S$ of dimension 1 over $K$. If $\sigma(\lambda)\eta = \eta\lambda$ for all $\lambda\in K$, then $K(x-\eta) = (x-\eta)K$, so
\[
R (x-\eta) = \ideal{q(x)} + K(x-\eta) = \ideal{q(x)} + (x-\eta)K = (x-\eta) ((x-\mu)R + K) = (x-\eta) R
\]
so $J$ is a two-sided ideal in $S$.

In case (4), $J$ is a non-trivial nilpotent two-sided ideal in $S$, so $S$ is not semisimple. In case (3) $I$ and $J$ are distinct left ideals in $S$, so $S = I\oplus J$. Moreover $J$ is also a two-sided ideal. Thus $S$ is semisimple, but not simple, so it must be isomorphic to $K\times K$. In case (2) we have $J \neq J \lambda^{-1}$ since $J \lambda^{-1}$ corresponds to 
\[
R(x-\eta)\lambda^{-1} = \ideal{q(x)} + K(x-\eta)\lambda^{-1} = \ideal{q(x)} + K(x-\sigma(\lambda)\eta\lambda^{-1}).
\]
So $S=J\oplus J\lambda^{-1}$, a direct sum of isomorphic simple left ideals, so $S\cong M_2(D)$ for some division ring $D$.
\end{proof}

For example let $K$ be (commutative, and) a Galois field extension of some subfield $F$, of degree $2$, and let  $q(x)=x^{2}-\gamma$ with $\gamma\neq 0$. Here $S\cong M_{2}(F)$ if and only if $\gamma=\sigma(\alpha)\alpha$ for some $\alpha\in K$, and otherwise  $S$ is a division ring, see for example \cite[Theorem 1.3.16]{Jac2009}. 
The following fact is trivial, but nonetheless important.

\begin{lemma}\label{lemma:invertible-variable-in-quotient-by-non-singular-quadratic}
If $q(x) = x^2-\beta x+\gamma\in R = K[x;\sigma]$, then
$x$ is invertible in $S=R/\ideal{q(x)}$ if and only if $\gamma\neq 0$, that is, $q(x)$ is non-singular. In this case, in the quotient $S$ we have
\[
x^{-1} = \gamma^{-1}\beta - \gamma^{-1} x
\quad\text{and}\quad
x = \beta - \gamma x^{-1}.
\]
\end{lemma}

\subsection{Semilinear clannish algebras}
Given $K$, $Q$, a collection $\boldsymbol{\sigma}$ of automorphisms of $K$, a set $\Sp$ of special loops, quadratic polynomials $q_s(x) = x^2 - \beta_s x+\gamma_s \in K[x;\sigma_s]$ for each $s\in \Sp$, and a set $Z$ of zero relations, the notion of a semilinear clannish algebra $R = K_{\boldsymbol{\sigma}} Q/I$ is defined in the introduction. It is naturally a $K$-ring, and when we speak of finite-dimensional $R$-modules, we mean left $R$-modules that are finite-dimensional as a left $K$-module. It is an algebra in the more usual sense over the field of elements in the centre of $K$ which are invariant under all~$\sigma_a$.

\begin{remark}
\label{remark:condsonpolys}
For our main theorem we impose three conditions on the quadratic polynomials $q_s(x) = x^2-\beta_s x + \gamma_s\in K[x;\sigma_s]$ associated to special loops $s\in\Sp$.

(i) Normality is a sensible condition to impose, for otherwise $K[x;\sigma_s]/\ideal{q_s(x)}$ has dimension $\le 1$ over $K$, so if $s$ is a loop at vertex $i$, then $s$ acts as a scalar on $e_i M$ for any $R$-module $M$.

(ii) We assume that $q_s(x)$ is non-singular, that is, $\gamma_s\neq 0$. 
%
This is needed for our functorial filtration approach, but we conjecture that the condition is not necessary. For example, using a classification for matrix problems due to Bondarenko~\cite{Bon1991} rather than the functorial filtration method, Hansper~\cite{Hansper} has shown that an analogue of our main theorem holds for clannish algebras in which all special loops have polynomial $x^2-x$.
Note that if the centre of $K$ has at least three elements, then one can change the generators to ensure that the polynomials $q_s(x)$ are non-singular. Namely, suppose that $s$ is a special loop at vertex $i$ with $q_s(x) = x^2-\beta_s x$ normal. If $\beta_s=0$ we can consider $s$ as an ordinary loop. Otherwise, we can make a change of variable $s' = s - \mu\beta_s e_i$ with $\mu\neq 0,1$ in the centre of $K$. Then by Lemma~\ref{lemma:characterising_central_skew_quadratics} we have $s'\lambda = \sigma_s(\lambda) s'$ for $\lambda\in K$, and $(s')^2 - bs' + ce_i = 0$, where $q_{s'}(y) = y^2 - b y + c \in K[y;\sigma_s]$ is normal, with $b = \beta_s - \sigma_s(\mu\beta_s) - \mu\beta_s$ and $c = (\mu-1)\mu \beta_s^2 \neq 0$.


(iii) We assume that $q_s(x)$ is semisimple, meaning that $K[x;\sigma_s]/\ideal{q_s(x)}$ is a semisimple artinian ring. If not, then, assuming that $q_s(x)$ is normal, by Lemma~\ref{lemma:characterisingsemisimple-quadratics} we have $q_s(x) = (x-\eta)^2$ with $\eta \lambda = \sigma_s(\lambda) \eta$ for all $\lambda\in K$. Then with the change of variable $y = x-\eta$ the polynomial becomes $y^2=0$, and we can consider $s$ as an ordinary loop. 
\end{remark}

Next we consider a special type of semilinear clannish algebra. Instead of classifying its finite-dimensional modules, our main theorem uses the indecomposable modules for algebras of this form to parameterize modules associated to a symmetric band for any other semilinear clannish algebra. Here we show that these special semilinear clannish algebras are hereditary noetherian prime rings.
We point the reader to the survey \cite{Lev2000} by Levy on modules over hereditary noetherian prime rings. 

Let $\rho$ and $\tau$ be automorphisms of $K$, let $r(x)\in K[x;\rho]$ and $p(y)\in K[y;\tau]$ be normal monic non-singular quadratics, and let the factor rings be $S'=K[x;\rho]/\ideal{r(x)}$ and $S''=K[y;\tau]/\ideal{p(y)}$. 
We write $S = S' *_K S''$, the free product (or coproduct) of $S'$ and $S''$ over $K$ (see for example the end of \cite[\S 4]{Coh1959}).

\begin{theorem}
\label{theorem:symmetricbandparameterringisHNP}
The algebra $S$ defined above is isomorphic as a $K$-ring to the semilinear clannish algebra given over $K$ by a quiver with one vertex and two special loops $x,y$, with $\sigma_x=\rho$, $\sigma_y=\tau$, $q_x(x)=r(x)$ and  $q_y(y)=p(y)$. The alternating monomials in $x$ and $y$ give a basis for $S$ as a left or right $K$-module, and $S$ is a prime noetherian ring. 

If additionally $S'$ and $S''$ are semisimple, then $S$ is hereditary.
\end{theorem}

\begin{proof}
Let $R$ be the semilinear clannish algebra. By construction there are ring homomorphisms from $S'$ and $S''$ to $R$ giving a commutative square with the maps from $K$. One easily checks that these maps satisfy the universal property of the free product.

That the alternating monomials give a basis follows from \cite[Proposition 4.1]{Berg1974} (or  Theorem~\ref{theorem:K-basis-of-paths-for-clannish-algebra} below). Since $x\lambda = \rho(\lambda)x$ and $y\lambda = \tau(\lambda)y$ for $\lambda \in K$, any product of non-zero elements of $K$ and copies of $x$ and $y$ can be written in the form $\lambda m$ with $0\neq\lambda\in K$ and $m$ a monomial in $x$ and $y$. Moreover, the quadratic relations mean that any monomial which is not alternating can be written as a linear combination of monomials of smaller degree, hence by induction as a linear combination of alternating monomials of smaller degree.

To show that $S$ is prime, it suffices to show that if $a,a'$ are non-zero elements of $S$ then $aba'\neq 0$ for some $b\in S$. 
Now $a$ can be written in the form $a = \lambda m+c$ where $m$ is an alternating monomial, $0\neq \lambda\in K$, and $c$ is a linear combination of monomials of smaller degree than $m$, and possibly also another alternating monomial of the same degree as $m$, but with $x$ and $y$ exchanged.  Similarly $a' = \lambda' m' + c'$. Choose a non-trivial monomial $b$ such that $mbm'$ is alternating. Then $aba' = \mu mbm' + d$ where $0\neq \mu\in K$ and $d$ is a linear combination of monomials of smaller degree than $mbm'$, and possibly also a non-alternating monomial of the same degree as $mbm'$, but that too is a linear combination of monomials of smaller degree. Thus $aba'\neq 0$.

Since $x$ and $y$ are units in $S$ by Lemma \ref{lemma:invertible-variable-in-quotient-by-non-singular-quadratic}, and since $xy\lambda = \rho(\tau(\lambda)) xy$, there is a ring homomorphism from $T=K[z,z^{-1};\rho\tau]$ to $S$ sending $z$ to $xy$. Writing $yx$ as a linear combination of $(xy)^{-1}$, $x$, $y$ and $1$, it follows that $S$ is generated as a left or right $T$-module by $1$. $x$ and $y$. Now $T$ is noetherian by \cite[Theorem 1.4.5]{McCRob2001}, hence so is $S$. Finally, if $S'$ and $S''$ are semisimple, they have global dimension 0, so $S$ is hereditary by \cite[Corollary 2.5]{Berg1974}.
\end{proof}

We remark that another way to study $S = S'*_K S''$ when $S'$ and $S''$ are semisimple, is to use that $M_2(S)$ is a universal localization of the hereditary artinian ring
\[
\begin{pmatrix} S' & S'\otimes_K S'' \\ 0 & S'' \end{pmatrix}.
\]
See \cite[Theorem 4.10]{Sch1985}. Note that if $K$ is finite-dimensional over a central subfield $k$ which is fixed by $\rho$ and $\tau$ then $S$ is a classical hereditary order. Namely, the ring $K[z,z^{-1};\rho\tau]$ appearing in the proof above is module-finite over $k[z,z^{-1}]$, hence so is $S$, so it is a classical hereditary order by~\cite{RobsonSmall}.

\subsection{Words and strings and bands}
\label{s:words}
We consider `words' composed of the following types of letters:
\begin{itemize}
\item[(i)]An \emph{ordinary direct letter} is a symbol of the form $a$ with $a$ an ordinary arrow in $Q$;
\item[(ii)]An \emph{ordinary inverse letter} is a symbol of the form $a^{-1}$ with $a$ an ordinary arrow in $Q$;
\item[(iii)]A \emph{$*$-letter} is a symbol of the form $s^*$ with $s$ a special loop in $Q$.
\end{itemize}
We define the \emph{head} and \emph{tail} (in $Q_0$) of any of these letters as follows. The head and tail of a direct letter are the head and tail of the corresponding arrow, the head and tail of an inverse letter are the tail and head, respectively, of the corresponding arrow, and the head and tail of a $*$-letter are the head and tail of the special loop. The \emph{inverse} of a letter is defined by $(a^{-1})^{-1} = a$ and $(s^*)^{-1} = s^*$, so inversion swaps the head and tail of a letter.

Because of the conditions defining a semilinear clannish algebra, we can and do choose a \emph{sign} $\pm 1$ for each letter, such that if distinct letters $x,y$ have the same head and sign, then $\{x,y\} = \{a^{-1},b\}$ for some zero relation $ab$ defining the algebra (in particular, the arrows $a$ and $b$ are both ordinary). 

Let $I$ be a subset of $\Z$ of one of the following types: $\{0,\dots,n\}$ with $n\ge 0$ or $\N = \{0,1,\dots\}$ or $-\N=\{0,-1,\dots\}$ or $\Z$. We write $I'=\{i\in I\mid i-1\in I\}$. 
A \emph{word} $w$ indexed by $I$ is determined by letters $w_i$ for all $i\in I'$, vertices $v_i(w)\in Q_0$ for $i\in I$, and a sign $\epsilon=\pm1$, subject to the following conditions:
\begin{itemize}
\item[(i)]
The head of $w_i$ is $v_{i-1}(w)$ and the tail of $w_i$ is $v_i(w)$.
\item[(ii)]
If $w_i$ and $w_{i+1}$ are consecutive letters, then $w_i^{-1}$ and $w_{i+1}$ have opposite signs. (In particular this implies that $w_i$ and $w_{i+1}$ cannot be inverses of each other, or both equal to the same $*$-letter.)
\item[(iii)]
If $1 \in I'$ then the sign of $w_1$ is $\epsilon$. If $0\in I'$ then the sign of $w_0^{-1}$ is $-\epsilon$.
\end{itemize}
A word is \emph{finite} of \emph{length} $n$ if $I = \{0,1,\dots,n\}$, and otherwise \emph{infinite}. 
We denote by $1_{\ell,\epsilon}$ the trivial word of length 0 with $v_0(1_{\ell,\epsilon}) = \ell$ and sign $\epsilon$. Any other word is uniquely determined by the sequence of letters, and so we write $w = w_1 w_2 \dots w_n$ (if $I = \{0,1,\dots,n\}$), $w = w_1 w_2 \dots$ (if $I=\N$), $w = \dots w_{-1} w_0$ (if $I=-\N$) or $w = \dots w_{-1} w_0 | w_1 w_2 \dots$ (if $I=\Z$). 

The \emph{inverse} $w^{-1}$ of a word $w$ is obtained by inverting the letters and reversing their order. In particular we define the inverse of a $\Z$-indexed word by $(w^{-1})_i = (w_{-i})^{-1}$. (Warning: this differs from the convention in \cite{Cra2018}). For words of length zero, we define $(1_{\ell,\epsilon})^{-1} = 1_{\ell,-\epsilon}$. 

The $n$th \emph{shift} of a $\Z$-indexed word $w$ is given by $w[n] = \dots w_n|w_{n+1} \dots$, so $w[n]_i = w_{n+i}$. The shift operation is defined on words with other indexing sets, but has no effect. We say that a word $w$ is \emph{periodic} if it is $\Z$-indexed and $w=w[n]$ for some $n>0$, in which case the minimal such $n$ is the \emph{period}.

\begin{definition}
We say that two words $u,w$ are \emph{equivalent} if $u = w[n]$ or $u = (w^{-1})[n]$ for some $n\in\Z$. This defines an equivalence relation on the set of words.
\end{definition}

By a word with \emph{head} $\ell$ we mean a finite or $\N$-indexed word $w$ with $v_0(w)=\ell$. We define the product $uw$ of words $u$ and $w$ by concatenating the sequences of letters, provided that $u^{-1}$ and $w$ have the same head and opposite signs. Thus $1_{\ell,\epsilon} 1_{\ell,\epsilon} = 1_{\ell,\epsilon}$. If $u$ is $-\N$-indexed and $w$ is $\N$-indexed, then the product of $u$ and $w$ is $uw = \dots u_{-1} u_0 | w_1 w_2 \dots$. 

If $w = w_1\dots w_n$ is a finite word where $w$ and $w^{-1}$ have the same head and opposite signs, there is an $\N$-indexed word $w^\infty = w_1\dots w_n w_1 \dots w_n \dots$ and a periodic word $^{\infty} w^\infty = \dots w_1\dots w_n | w_1 \dots w_n \dots$.

If $w$ is an $I$-indexed word and $i\in I$, there are words $w_{\le i} = \dots w_{i-1} w_i$ and $w_{>i} = w_{i+1} w_{i+2} \dots$, with appropriate conventions if either of these has length 0, so that $w[i] = w_{\le i} w_{>i}$.

\begin{definition}
Given a non-trivial path $p=a_1 \dots a_n$, we obtain a sequence of letters $p^* = a_1^* \dots a_n^*$ on replacing any special loop $s$ by the corresponding $*$-letter $s^*$ (but leaving ordinary arrows unchanged). A word $w$ is \emph{relation-admissible} if $w$ and $w^{-1}$ do not contain, as a subword of consecutive letters, a sequence of letters of the form $r^*$ where $r$ is one of the zero relations defining the algebra. 
\end{definition}

\begin{definition}
We say that an $I$-indexed word $w$ is \emph{right-end-admissible} if either $I$ is not bounded above, or it is bounded above and there is no special loop $s$ with $w s^*$ a word. A word $w$ is \emph{end-admissible} if $w$ and $w^{-1}$ are right-end-admissible. For example if $w$ is $\Z$-indexed this is automatic.
\end{definition}

\begin{definition}
By a \emph{string} we mean a finite end-admissible relation-admissible word. A string $w$ is \emph{symmetric} if $w=w^{-1}$ and otherwise it is \emph{asymmetric}. By a \emph{band} we mean a relation-admissible $\Z$-indexed word $w$ which is periodic, so $w[n]=w$ for some $n>0$. A band $w$ is \emph{symmetric} if $w^{-1}$ is equal to some shift of $w$ and otherwise it is \emph{asymmetric}.
\end{definition}

\begin{lemma}
\label{lem-end-adm}
An $I$-indexed word $w$ is end-admissible if and only if for each $i\in I$ and each special loop $s$ at $v_i(w)$, either $i\in I'$ and $w_i = s^*$ or $i+1\in I'$ and $w_{i+1} = s^*$.
\end{lemma}

\begin{proof}
Clear.
\end{proof}

\begin{definition}
\label{definition:order-on-words}
Given a vertex $\ell\in Q_0$ and a sign $\epsilon = \pm 1$, let $H(\ell,\epsilon)$ be the set of right-end-admissible words $w$ which are finite or $\N$-indexed and have head $\ell$ and sign $\epsilon$. We define a total ordering on $H(\ell,\epsilon)$, with $w < w'$ if and only if one of (a), (b) or (c) below holds.
\begin{itemize}
\item[(a)]
$w = u y v$ and $w' = u x^{-1} v'$ where $u$ is a finite word, $x,y$ are ordinary arrows, and $v,v'$ are words.
\item[(b)]
$w = u y v$ and $w' = u$ where $u$ is a finite word, $y$ an ordinary arrow, and $v$ a word.
\item[(c)]
$w = u$ and $w' = u x^{-1} v'$ where $u$ is a finite word, $x$ an ordinary arrow, and $v'$ a word.
\end{itemize}
\end{definition}

\begin{definition}
Suppose $w$ is an end-admissible word, say indexed by $I$. Let $i\in I'$ and suppose that $w_i$ is a $*$-letter. It follows that the words $(w_{\le i-1})^{-1}$ and $w_{>i}$ have the same head and sign, so they are comparable. We say that $i$ is a \emph{symmetry} for $w$ if $(w_{\le i-1})^{-1} = w_{>i}$, that $i$ is \emph{naturally direct} for $w$ if $(w_{\le i-1})^{-1} > w_{>i}$, and that $i$ is \emph{naturally inverse} for $w$ if $(w_{\le i-1})^{-1} < w_{>i}$. 
\end{definition}

\begin{lemma}\label{lemma:form-of-symmetric-bands}
\begin{enumerate}
\item[(i)]
A finite end-admissible word $w$ has a symmetry if and only if $w=w^{-1}$. In this case $w$ is of the form $u s^* u^{-1}$, for some $s\in\Sp$ and finite word $u$. The unique symmetry for $w$ is $n+1$, where $n$ is the length of $u$.
\item[(ii)]
A periodic word $w$ has a symmetry if and only if $w^{-1}$ is equal to some shift of $w$. In this case there exist $s,t\in \Sp$ and words $u,v$ of lengths $p,r$ respectively, such that the period of $w$ is $2p+2r+t$ and such that $w$ has the form 
\[
\dots v s^* v^{-1} u^{-1} t^* u|v s^* v^{-1} u^{-1} t^* u v s^* v^{-1} u^{-1} t^* \dots={}^{\infty}(v s^* v^{-1} u^{-1} t^* u )^{\infty}.
\]
The symmetries of $w$ are the translates of $-p$ and $r+1$ by multiples of the period of $w$.
\end{enumerate}
\end{lemma}

\begin{proof}
(i) If $w=w^{-1}$ then $w$ must have odd length, for if it has length $2n$, then $(w_n)^{-1} = w_{n+1}$, contrary to the definition of a word. If it has length $2n+1$, then $(w_{n+1})^{-1} = w_{n+1}$, so it is a $*$-letter, say $s^*$, $n+1$ is a symmetry and $w$ has the stated form.

(ii) Note first that if $w^{-1} = w[j]$ then $j$ is even, for if $j = 2k+1$ then $(w_k)^{-1} = (w^{-1})_{-k} = w[j]_{-k} = w_{k+1}$, which is impossible. Now if $j=2i$, then $w^{-1} = w[j]$ if and only if $w[i]^{-1} = w[i]$, or equivalently $i$ is a symmetry for $w$.

The isometries of $\Z$ are the maps $k\mapsto c k+d$ with $c=\pm1$ and $d\in \Z$. 
Those preserving $w$, in the sense that $w_{ck+d} = w_k^c$, include translations $\tau_d(k) = k+d$ with $d$ a multiple of the period of $w$ and reflections $\rho_i(k) = 2i-k$ for $i$ a symmetry. 
So they form an infinite dihedral group, see for example \cite[p. 32]{Wad2017}. 

This group can be generated by the reflections associated to adjacent symmetries, say $-p$ and $r+1$ with $p,r\ge 0$. Then $\rho_{r+1}\rho_{-p} = \tau_{2p+2r+2}$, so $w$ has period $2p+2r+2$. It follows that $w$ has the stated form.
\end{proof}

\subsection{Walks and the canonically associated walk}
By definition a \emph{walk} $C$ is given by the same data as a word, except now the allowed letters $C_i$ are of the form
\begin{itemize}
\item[(i)]a \emph{direct letter} is a symbol of the form $a$ with $a$ an arrow in $Q$, either ordinary or special,
\item[(ii)]an \emph{inverse letter} is a symbol of the form $a^{-1}$ with $a$ an arrow in $Q$, either ordinary or special,
\end{itemize}
with the requirement that one obtains a word, denoted $C^*$, when one replaces any letter of the form $s$ or $s^{-1}$, where $s$ is a special loop, by the corresponding $*$-letter $s^*$. 


A walk $C$ is said to be  $I$-\emph{indexed}, or \emph{trivial}, or \emph{finite} of \emph{length} $n$, or \emph{infinite}, respectively, provided the same is true for the word $C^{*}$.   We define the sign of the letters $s$ and $s^{-1}$ to be that of $s^*$ for any special loop $s$. Hence for any $I$-indexed walk $C$ and any $i\in I$ we have that $C_{i}$ and $C^{*}_{i}$ have the same sign, and we let  $v_{i}(C)=v_{i}(C^{*})$. The \emph{head} and \emph{sign} of a finite or $\N$-indexed walk $C$ are defined to be that of $C^{*}$.

The \emph{inverse} $C^{-1}$ of a walk $C$  is defined by inverting the letters of and reversing their order. Similarly, following the definitions above which concerned words, we define: the $n$th shift of a walk; periodic walks and their period; the product of walks $C$ and $D$ where $C^{-1}$ and $D$ have the same head and opposite signs; and the walks $C_{\leq i}$ and $C_{>i}$ for any $I$-indexed walk $C$ and any $i\in I$.


Associated to a walk $C$ there is a quiver $Q_C$ and a morphism of quivers $f_C:Q_C\to Q$ defined as follows. The vertex set of $Q_C$ is the indexing set $I$ for $C$, and the morphism $f_C$ sends $i$ to $v_i(C)$. For each $i\in I'$, if $C_i = a$ is a direct letter, then there is an arrow $\alpha_i:i\to i-1$, while if $C_i = a^{-1}$ is an inverse letter, then there is an arrow $\alpha_i:i-1\to i$, and in both cases $f_C$ sends $\alpha_i$ to $a$.

\begin{example}
\label{example:walks-words-quivers}
Define $Q$ by two loops $a$ and $s$ at a single vertex. Take the division ring $K$ and the automorphisms $\sigma_{a},\sigma_{s}$ to be arbitrary, and relabel by $\sigma=\sigma_{s}$ and $\theta=\sigma_{a}$. 
Define $R$ by taking $\Sp=\{s\}$,  $Z=\{a^{2}\}$ and $q_{s}(x)$ to be any monic, non-singular, normal and semisimple quadratic in $K[x;\sigma]$.
In this case we have $R=K_{\sigma,\theta}\langle s,a \rangle/\ideal{q_{s}(s),a^{2}}$. 
For this semilinear clannish algebra the words are given by alternating sequences in ($a$ or $a^{-1}$) and $s^{*}$, and such a word is right-end admissible if and only if it ends in $s^{*}$. 
Since $Z$ only contains paths of length $2$, all words (for this example of $R$) are relation-admissible. 

Given the walk $C = s^{-1}a^{-1}s^{-1}a^{-1}s^{-1}a s a s$, the associated word is $C^* = s^*a^{-1}s^*a^{-1}s^*a s^* a s^*$ and the quiver $Q_C$ is given as follows, where we label the arrows by their images under $f_C$.
\[\begin{tikzcd}
0\arrow{r}{s} & 1 \arrow{r}{a}
& 2\arrow{r}{s} & 3\arrow{r}{a}
& 4\arrow{r}{s} & 5 
& 6\arrow[swap]{l}{a}  & 7\arrow[swap]{l}{s} 
& 8\arrow[swap]{l}{a} & 9\arrow[swap]{l}{s}
\end{tikzcd}
\]
\end{example}

\begin{definition}
Let $w$ be an end-admissible word, say $I$-indexed.
Let $C$ be a walk such that $C^*=w$. We say that $C$ is \emph{naturally oriented} if $C_i = s$ whenever $w_i = s^*$ and $i$ is naturally direct for $w$, and $C_i = s^{-1}$ whenever $w_i=s^*$ and $i$ is naturally inverse for $w$. 

The \emph{canonically associated walk} for $w$, denoted $C_w$, is the naturally oriented walk with $(C_w)^*=w$, such that if $i$ is a symmetry for $w$, we have $(C_w)_i = s^{-1}$ if $i>0$ and $(C_w)_i = s$ if $i\le 0$.
\end{definition}

If $w$ is a finite end-admissible word, with $w=w^{-1}$, then by Lemma~\ref{lemma:form-of-symmetric-bands}, $w$ is of the form $u^{-1}s^*u$ for some finite word $u$, and the canonically associated walk is of the form $C_w = D s^{-1}D^{-1}$ for some walk $D$ with $D^*=u$. If $w$ is a periodic word and $w^{-1}$ is equal to a shift of $w$, then $w$ is of the form in Lemma~\ref{lemma:form-of-symmetric-bands}, and the canonically associated walk $C_w$ has the form
\[
\dots E s E^{-1} D^{-1} t D|E s^{-1} E^{-1} D^{-1} t^{-1} D E s^{-1} E^{-1} D^{-1} t^{-1} \dots
\]
for some walks $D,E$ with $D^*=u$ and $E^*=v$.

For example, in Example~\ref{example:walks-words-quivers}, if $w$ is the periodic word ${}^\infty (a^{-1}s^*a^{-1}s^*as^{*}as^{*})^\infty$, then we have $D$ trivial, $E = a^{-1}s^{-1}a^{-1}$, $t=s$ and $Q_{C_w}$ is the following quiver.
\[
\cdots 
\xrightarrow{a} -5 \xleftarrow{\mathbf{s}} 
-4 \xleftarrow{a} -3 \xleftarrow{s} -2  \xleftarrow{a} -1 \xleftarrow{\mathbf{s}} 
0 \xrightarrow{a} 1 \xrightarrow{s} 2  \xrightarrow{a} 3 \xrightarrow{\mathbf{s}} 
4 \xleftarrow{a} 5 \xleftarrow{s} 6  \xleftarrow{a} 6 \xrightarrow{\mathbf{s}} 
7 \xrightarrow{a} 8 \xrightarrow{s} 9  \xrightarrow{a} 10 \xrightarrow{\mathbf{s}} 
11 \xleftarrow{a} 
\cdots
\]
The letters $(C_w)_i$ with $i$ a symmetry are shown in bold.

\subsection{The modules $M(C)$}
\begin{definition}\label{definition:modules-M(C)}
If $C$ is walk and the word $C^*$ is end-admissible, we define $M(C)$ to be the $R$-module with generators $(b_i)_{i\in I}$, where $I$ is the indexing set for $C$, subject to the following relations.
\begin{itemize}
\item[(i)]
If $i\in I$ then $e_\ell b_i = b_i$ where $\ell = v_i(C)$.
\item[(ii)]
If $\alpha:i\to j$ is an arrow in $Q_C$ then $a b_i = b_j$ where $a = f_C(\alpha)$.
\item[(iii)]
If $a$ is an ordinary arrow in $Q$, $i\in I$, $v_i(C) = t(a)$, and no arrow in $Q_C$ with tail $i$ is sent to $a$ by $f_C$, then $a b_i = 0$.
\end{itemize}
\end{definition}

\begin{lemma}\label{lemma:M(C)-spanned-over-K}
Suppose that $C$ is walk and that the word $C^*$ is end-admissible.
\begin{itemize}
\item[(i)]
If $s$ is a special loop in $Q$, $i\in I$, $v_i(C)=t(s)$, and no arrow in $Q_C$ with tail $i$ is sent to $s$ by $f_C$, then for some $j\in I$ there is an arrow $j\to i$ in $Q_C$ sent to $s$ by $f_C$, and 
$s b_i = \beta_s b_i - \gamma_s b_j$.
\item[(ii)]
The module $M(C)$ is spanned as a $K$-module by the elements $(b_i)_{i\in I}$.
\end{itemize}
\end{lemma}

\begin{proof}
(i) The arrow exists by Lemma \ref{lem-end-adm}, since $C^*$ is end-admissible. Now the defining relations give $s b_j = b_i$, so the quadratic relation for $s$ gives $sb_i = s^2 b_j = (\beta_s s - \gamma_s e_{t(s)}) b_j = \beta_s b_i - \gamma_s b_j$.

(ii) From the defining relations and (i) we know that the generators of the algebra $R$ send the elements $b_i$ to $K$-linear combinations of the $b_j$. Thus we get the result by semilinearity, using that $p \lambda b_i = \sigma_p(\lambda) p b_i$ for a path $p$ in $Q$, see \S\ref{s:basicdefns}. 
\end{proof}

\begin{lemma}\label{lemma:K-basis-of-M(C)-iff-C-relation-admissible}
Suppose that $C$ is walk and that the word $C^*$ is end-admissible. 
The collection $(b_i)_{i\in I}$ is a $K$-basis of $M(C)$ if and only if the word $C^*$ is relation-admissible.
\end{lemma}

\begin{proof}
We can define a free $K$-module $M'(C)$ with basis $(b'_i)_{i\in I}$ and use the relations in the definition of $M(C)$ and Lemma~\ref{lemma:M(C)-spanned-over-K}(i) to turn $M'(C)$ into a representation of the algebra $K_{\boldsymbol{\sigma}} Q/I'$, where $I'$ is the ideal generated by the quadratics $s^2 - \beta_s s + \gamma_s e_{h(s)}$ for each special loop $s\in \Sp$.

Now the $b_i$ are linearly independent over $K$ if and only if $M'(C)$ is annihilated by all the zero-relations defining the algebra $R$. Namely, if the $b_i$ are linearly independent over $K$, then we can identify $M(C)$ and $M'(C)$, and so $M'(C)$ is annihilated by the zero-relations for $R$. Conversely, if $M'(C)$ is annihilated by all the zero-relations for the algebra $R$, then it becomes an $R$-module, and there is a homomorphism $M(C)\to M'(C)$ sending each $b_i$ to $b'_i$, from which it follows that the $b_i$ are linearly independent over $K$. 

By the choice of signs of letters, any path $ab$ of length 2 where the letters $a^{-1}$ and $b$ have the same sign must be a zero-relation. By the condition on signs in the definition of a word, the quiver $Q_C$ cannot contain a path sent to $ab$ by $f_C$, and so the zero-relation $ab$ is automatically satisfied by $M'(C)$.

Let $r$ be any other zero-relation occurring in the definition of $R$. We may suppose that no length 2 path as above occurs as a sub-path of $r$, for otherwise $r$ automatically annihilates $M'(C)$. Moreover, by the definition of a semilinear clannish algebra, $r$ does not involve the square of a special loop. It follows that the sequence of letters $r^*$ is a word. Recall also that a zero-relation must not start or end with a special loop. The result thus follows from the following assertion.

Let $w = C^*$. Let $p=a_1\dots a_n$ be a path in $Q$ with $n\ge 1$. Suppose that $a_n$ is an ordinary arrow and that $p$ does not have a special loop occurring twice in succession. If $i\in I$, then
\begin{itemize}
\item[(i)]
If $i-n\in I$ and $p^* = w_{i-n+1} \dots w_{i-1} w_i$, then $p b'_i$ is a non-zero scalar multiple of $b'_{i-n}$, plus, if $a_1$ is a special loop $s$, a scalar multiple of $b'_{i-n+1}$.
\item[(ii)]
If $i+n\in I$ and $p^* = w_{i+n}^{-1} \dots w_{i+1}^{-1}$, then $p b'_i$ is a non-zero scalar multiple of $b'_{i+n}$, plus, if $a_1$ is a special loop $s$, a scalar multiple of $b'_{i+n-1}$.
\end{itemize}
Conversely, if $p b'_i\neq 0$ then one of these two cases must occur.

We prove (i) by induction on $n$. If $n=1$ it is clear since there is an arrow in $Q_C$ with tail $i$ sent to $a_1$ by $f_C$. Thus suppose $n>1$. By induction $a_2\dots a_n b'_i$ is a non-zero scalar multiple of $b'_{i-n+1}$, plus, if $a_2$ is a special loop $s$, a scalar multiple of $b'_{i-n+2}$. Now in the latter case $n\ge 3$ and $a_3$ is ordinary, so $C_{i-n+3}$ is an ordinary direct letter. But then, since $C_{i-n+2} = s^{\pm 1}$, in $Q_C$ there is no arrow with tail $i-n+2$ sent to $a_1$ by $f_C$, and hence $a_1 b'_{i-n+2} = 0$. Thus $p b'_i$ is a non-zero scalar multiple of $a_1 b'_{i-n+1}$, and the claim follows.

Case (ii) is similar to case (i).

For the converse direction suppose that $p b_i \neq 0$. If $n=1$ then $a_1$ must be an ordinary arrow, and to have $a_1 b_i\neq0$ there must be an arrow in $Q_C$ with tail $i$ sent to $a_1$ by $f_C$. The head is either $i-1$ or $i+1$, giving (i) or (ii). Next suppose $n>1$. By induction (i) or (ii) hold for the path $a_2\dots a_n$; without loss of generality say (i) holds. Thus $a_2\dots a_n b_i$ is a non-zero scalar multiple of $b'_{i-n+1}$, plus, if $a_2$ is a special loop $s$, a scalar multiple of $b'_{i-n+2}$. In the latter case, by the discussion above we have $a_1 b'_{i-n+2} = 0$. Thus we must have $a_1 b'_{i-n+1}\neq 0$. It follows that $Q_C$ must contain an arrow with tail $i-n+1$ sent to $a_1$ by $f_C$. This forces $i-n\in I$ and  $w_{i-n+1}=a_1$, as required.
\end{proof}

For a finite walk $C = C_1\dots C_n$ let $\sigma_C = \sigma_{C_1}\dots \sigma_{C_n}$, where $\sigma_{x^{-1}} = \sigma_x^{-1}$ for an inverse letter $x^{-1}$.

\begin{definition}
\label{definition:pi-i-automorphisms}
Let $C$ be an end-admissible walk with index set $I$. For each $i\in I$, we choose an automorphism $\pi_i$ of $K$ (depending also on $C$, but we suppress this), such that for any arrow $i\to j$ in $Q_C$, say sent to an arrow $a$ in $Q$ by $f_C$, we have $\pi_j = \sigma_a \pi_i$. These conditions uniquely determine the $\pi_i$, but for the free choice of one of them; for definiteness one could fix $\pi_0$ to be the identity automorphism. 
\end{definition}

\begin{lemma}\label{lemma:M(C)-is-an-R-K-bimodule}
There is a unique $R$-$K$-bimodule structure on $M(C)$ with $b_i \lambda = \pi_i(\lambda) b_i$ for $i\in I$ and $\lambda\in K$.
\end{lemma}

\begin{proof}
Let $F$ be the free $R$-module with basis $(b_i)_{i\in I}$. For any automorphisms $\pi_i$, it has a unique $R$-$K$-bimodule structure with the property that $b_i \lambda = \pi_i(\lambda) b_i$ for $i\in I$ and $\lambda\in K$. Now the relations defining $M(C)$ define an $R$-submodule $F'$ of $F$ with $M(C) \cong F/F'$, and the condition on the $\pi_i$ ensures that $F'$ is a sub-bimodule. For example, if $i\to j$ is an arrow in $Q_C$, sent to $a$ by $f_C$, then $ab_i-b_j\in F'$ and
\[
(ab_i - b_j)\lambda = a \pi_i(\lambda) b_i - \pi_j(\lambda) b_j = \sigma_a(\pi_i(\lambda)) a b_i - \pi_j(\lambda) b_j = \pi_j(\lambda)( ab_i - b_j) \in F'.
\]
Thus $M(C)$ is an $R$-$K$-bimodule. The uniqueness property is clear (for the given left action of $R$).
\end{proof}

\subsection{Bases for semilinear clannish algebras}
Let $R$ be a semilinear clannish algebra given over $K$ by a quiver $Q$. 
We say that a path in $Q$ is \emph{special-admissible} if it doesn't have the same special loop occurring as consecutive arrows, that is, it is not of the form $ps^2q$ for some special loop $s$ and paths $p,q$. 

We say that a path is \emph{admissible} if it is special-admissible and doesn't factor through a zero relation, so is not of the form $prq$ with $r$ a zero-relation. We can now use the modules $M(C)$, and in particular Lemma \ref{lemma:K-basis-of-M(C)-iff-C-relation-admissible}, to prove the following theorem.

\begin{theorem}\label{theorem:K-basis-of-paths-for-clannish-algebra}
The admissible paths form a basis for $R$ as a left or right $K$-module.
\end{theorem}

\begin{proof}
Using the semilinear property of the path algebra, it suffices to prove that the admissible paths form a left $K$-basis. To see that they span $R$ over $K$, note that the quadratic relations for special loops show that any path which is not special-admissible is a linear combination in $R$ of shorter paths, so by induction the special-admissible paths span, and then the paths which factor through a zero-relation are zero in $R$.

Now suppose there is a linear relation between admissible paths. Composing with a trivial path, it follows that for some vertex $\ell$ there is a relation between admissible paths starting at $\ell$.

Let $D$ and $E$ be the unique longest possible walks consisting entirely of inverse letters, with head $\ell$ and signs 1 and $-1$, such that $D^*$ and $E^*$ are relation-admissible. These exist by the definition of a semilinear clannish algebra. Then $C = D^{-1}E$ is a walk with $C^*$ relation-admissible. Also, since none of the zero-relations starts or ends with a special loop, $C^*$ is end-admissible. Let $I$ be the indexing set for $C$ and let $i\in I$ be the element with $D^{-1} = C_{\le i}$ and $E = C_{>i}$. Then there is a 1-1 correspondence between admissible paths starting at $\ell$ and elements of $I$, with $i$ corresponding to the trivial path $e_\ell$. But now if a path $p$ corresponds to vertex $j$, then $p b_i = b_j$ in the module $M(C)$. Since the $b_j$ are $K$-linearly independent by Lemma \ref{lemma:K-basis-of-M(C)-iff-C-relation-admissible}, so are the admissible paths starting at $\ell$.
\end{proof}


\subsection{Parameterizing rings and bimodule structure}
\label{s:asymstring}
Let $w$ be a string or band and let $C_w$ be the canonically associated walk. There are four possible types for $w$, and in this and the following subsections, we will in each case define a $K$-ring $R_w$ and turn $M(C_w)$ into an $R$-$R_w$-bimodule, finitely generated free as a right $R_w$-module, with basis $b_i$ for $i\in J_w$, where $J_w$ is a subset of $I$, the indexing set for $w$. 

In case $w$ is an asymmetric string we define $R_w = K$ and $J_w = I$. Then by Lemma \ref{lemma:M(C)-is-an-R-K-bimodule}, $M(C_w)$ is an $R$-$R_w$-bimodule, and as a right $R_w$-module, $M(C_w)$ is free with basis $b_i$ ($i\in J_w$) by Lemma \ref{lemma:K-basis-of-M(C)-iff-C-relation-admissible}.

\begin{example}\label{example:asymmetric-string}
We continue with the semilinear clannish algebra  $R=K_{\sigma,\theta}\langle s,a \rangle/\ideal{q_{s}(s),a^{2}}$ from Example \ref{example:walks-words-quivers}. Let $w=s^{*}as^{*}as^{*}$, an asymmetric string whose canonically associated walk is $C_{w}=sasas$.

Choose $\pi_{i}$ as in Definition \ref{definition:pi-i-automorphisms} by setting $\pi_{0}$ to be the identity, so that $\pi_{1}=\sigma^{-1}$, $\pi_{2}=\theta^{-1}\sigma^{-1}$, $\pi_{3}=\sigma^{-1}\theta^{-1}\sigma^{-1}$, $\pi_{4}=\theta^{-1}\sigma^{-1}\theta^{-1}\sigma^{-1}$ and  $\pi_{5}=\sigma^{-1}\theta^{-1}\sigma^{-1}\theta^{-1}\sigma^{-1}$.  

Using the $R$-$K$-bimodule structure from Lemma \ref{lemma:M(C)-is-an-R-K-bimodule} the $R$-module $M(C_{w})\otimes V=\bigoplus_{i}b_{i}\otimes V$ may be depicted by the following diagram, in which $b_{i}\otimes V$ is identified with a copy of the twisted $K$-module ${}_{\pi{_i}^{-1}}V$, since $\lambda b_i \otimes v = b_i \pi_i^{-1}(\lambda) \otimes v = b_i \otimes \pi_i^{-1}(\lambda) v$ for $v\in V$ and $\lambda\in K$.



\[
\begin{tikzcd}[column sep=1.2cm,row sep=0.2cm]
V & 
{}_{\sigma}V\arrow["", "s"']{l} & 
{}_{\sigma\theta}V\arrow["", "a"']{l} & 
{}_{\sigma\theta\sigma}V\arrow["", "s"']{l} & 
{}_{\sigma\theta\sigma\theta}V\arrow["", "a"']{l} & {}_{\sigma\theta\sigma\theta\sigma}V\arrow["", "s"']{l}
\end{tikzcd}
\]
\end{example}

\subsection{Symmetric strings}
\label{s:symstring}

Let $w$ be a symmetric string which, by Lemma \ref{lemma:form-of-symmetric-bands}(i), is of the form $u s^* u^{-1}$ for some finite word $u$ and special loop $s$. Let $u$ have length $k$, say, so $w$ has length $n=2k+1$.

\begin{lemma}
\label{lemma:symmetric-strings-R_w-basis}
There is a unique way to turn $M(C_w)$ into an $R$-$R_w$-bimodule, such that
\begin{itemize}
\item[(i)] $R_w = K[x;\tau]/\ideal{q(x)}$ for $\tau\in\Aut(K)$ and $q(x) = x^2 - \beta x + \gamma\in K[x;\tau]$, 
\item[(ii)] the action of $x$ satisfies $b_i x = b_{n-i}$ for $i\le k$, and
\item[(iii)] the action extends the $R$-$K$-bimodule structure of $M(C_w)$.
\end{itemize}
Namely, we need $\tau = \pi_k^{-1}\sigma_s \pi_k$, $\beta = \pi_k^{-1}(\beta_s)$, $\gamma = \pi_k^{-1}(\gamma_s)$ and $b_i x = b_i \beta - b_{n-i} \gamma$ for $i > k$. As a right $R_w$-module, $M(C_w)$ is free with basis $b_i$ \emph{(}$i\in J_w$\emph{)} for $J_w=\{0,1,\dots,k\}$.
\end{lemma}

Note that as in the proof of Corollary \ref{corollary:twisting-(and-inverting-non-singular)-preserves-normal/central-quadratics}, the map sending $\sum_{i=0}^n \lambda_i x^i$ to $\sum_{i=0}^n \pi_k^{-1}(\lambda_i) x^i$ defines a ring isomorphism $K[x;\sigma_s]\to K[x;\tau]$ which we also denote by $\pi_k^{-1}$. This map sends $q_s(x)$ to $q(x)$, so $q(x)$ is normal, and $R_w \cong K[x;\sigma_s]/\ideal{q_s(x)}$.

\begin{proof}
Let $C = C_w$; it is of the form $D s^{-1} D^{-1}$ for some walk $D$ with $D^*=u$.
%
For $i \le k$ and $\lambda\in K$ we need $b_i (x \lambda) = (b_i x)\lambda$.
Now $b_i(x \lambda) = b_i (\tau(\lambda) x) = \pi_i(\tau(\lambda)) b_i x = \pi_i(\tau(\lambda)) b_{n-i}$, and $(b_i x)\lambda = b_{n-i} \lambda = \pi_{n-i}(\lambda) b_{n-i}$. Thus we need $\tau = \pi_i^{-1} \pi_{n-i}$. Note that this doesn't depend on $i$, for if $\alpha : i \to j$ ($i,j\le k$) is an arrow in $Q_{C}$ with $f_{C}(\alpha)=a$, then by symmetry $a=f_{C}(\alpha')$ for an arrow $\alpha' : n-j \to n-i$ in $Q_{C}$,  giving $\pi_j = \sigma_a \pi_i$ and $\pi_{n-j} = \sigma_a \pi_{n-i}$ and hence $\pi_j^{-1} \pi_{n-j} = \pi_i^{-1} \pi_{n-i}$. In particular $\tau = \pi_k^{-1} \pi_{k+1} = \pi_k^{-1} \sigma_s \pi_k$, which is the stated condition on $\tau$.

We also need $(s b_k) x = s (b_k x)$. Now $(s b_k) x = b_{n-k} x = b_k x^2 = b_k (\beta x - \gamma) = \pi_k(\beta) b_k x - \pi_k(\gamma) b_k = \pi_k(\beta) b_{n-k} - \pi_k(\gamma) b_k$. On the other hand $s (b_k x) = s b_{n-k} = s^2 b_k = (\beta_s s - \gamma_s) b_k = \beta_s b_{n-k} - \gamma_s b_k$, corresponding to the conditions on $\beta$ and $\gamma$.

For $i\le k$ we need $b_{n-i} x = b_i x^2 = b_i (\beta x - \gamma) = \pi_i(\beta) b_i x - b_i \gamma
= \pi_{n-i}(\beta) b_{n-i} - b_i \gamma = b_{n-i} \beta - b_i \gamma$, where we have used that $\beta = \tau(\beta)$ by normality, and $\tau = \pi_i^{-1} \pi_{n-i}$, so $\pi_i(\beta) = \pi_{n-i}(\beta)$. This corresponds to the condition on $b_i x$ for $i>k$.

Now if the stated conditions on $\tau,\beta,\gamma$ and the action of $x$ hold, we need to check that
$M(C)$ becomes a right $R_w$-module and that the action commutes with the $R$-module structure.

To check that we have a right action of $K[x;\tau]$ on $M(C)$, we need to check $b_{i} x \lambda = b_{i} \tau(\lambda) x$ for all $i\in I=\{0,\dots,n\}$ and $\lambda\in K$. If $i\leq k$ this is straightforward, as $\pi_{i}\tau=\pi_{n-i}$. Now let $i\geq n-k$, which means $\pi_{n-i}=\pi_{i}\tau^{-1}$. By Lemma \ref{lemma:characterising_central_skew_quadratics}, since $q(x)$ is normal in $K[x;\tau]$ we have  $\tau^{-1}(\beta)=\beta$, $\beta\mu=\tau(\mu)\beta$ and $\gamma\mu=\tau^{2}(\mu)\gamma$. Using that $\pi_{n-i}\tau^{2}=\pi_{i}\tau$ this altogether gives $(b_{i}x)\mu=\pi_{i}(\tau(\mu))(b_{i}x)$, which is precisely the image of $\pi_{i}(\tau(\mu))b_{i}=b_{i}\tau(\mu)$ under $x$. Hence $K[x;\tau]$ acts on the right.
By construction $q(x)$ acts as zero, so $M(C)$ becomes a right $R_w$-module.

To see that $M(C)$ is an $R$-$R_{w}$-bimodule, as in Lemma~\ref{lemma:M(C)-is-an-R-K-bimodule} it suffices to show that the action of $x$ is compatible with the relations defining $M(C)$.
For example if $\alpha:i\to j$ is an arrow in $Q_C$ and $a=f_C(\alpha)$, there is a relation $ab_i = b_j$, and we need to check that $a(b_i x) = b_j x$. By the discussion above, the choice of $\beta$ and $\gamma$ ensures this holds for the arrow from $k$ to $k+1$, so we may assume that $i,j\le k$ or $i,j\ge k$. It suffices to check it for $i,j\le k$, for then multiplying on the right by $x^{-1}$ and using that $b_i x = b_{n-i}$, we obtain $a(b_{n-i}x^{-1}) = (a b_{n-i})x^{-1}$. Since $x$ is a $K$-linear combination of $1$ and $x^{-1}$ in $R_w$, we deduce that $a(b_{n-i}x) = (a b_{n-i})x$.
Now for any $i\le k$ we have $C_{n-i+1} = C_i^{-1}$. Thus there is an arrow $n-i\to n-j$ in $Q_C$, also sent to $a$ by $f_C$. Thus $a(b_i x) = a b_{n-i} = b_{n-j} = b_j x$, as required.



Since $b_{i}=b_{n-i}x$ for $i\geq n-k$, and since the elements $b_{j}$ ($j\in I$) span $M(C)$ over $K$ by Lemma \ref{lemma:M(C)-spanned-over-K},  the elements $b_{i}$ ($i\in J_{w}$) span $M(C)$ over $R_{w}$. Likewise, by a straightforward application of Lemma \ref{lemma:K-basis-of-M(C)-iff-C-relation-admissible}, one can show the elements $b_{i}$ ($i\in J_{w}$) are linearly independent over $R_{w}$, as required.
\end{proof}

%
%
%

\begin{example}
\label{example:symmetric-string}
Recall the semilinear clannish algebra  $R=K_{\sigma,\theta}\langle s,a \rangle/\ideal{q_{s}(s),a^{2}}$ from Example \ref{example:asymmetric-string}. 
Let $w=s^{*}a^{-1}s^{*}as^{*}a^{-1}s^{*}as^{*}$, a symmetric string with $C_{w}=s^{-1}a^{-1}s^{-1}as^{-1}a^{-1}sas$, $k=4$ and $n=9$.  
%
Again set $\pi_{0}$ to be the identity, so that $\pi_{1}=\sigma$, $\pi_{2}=\theta\sigma$, $\pi_{3}=\sigma\theta\sigma$ and $\pi_{4}=\sigma\theta\sigma\theta^{-1}$. 
Note $\sigma^{-1}(q_{s}(x))=q_{s}(x)$. 
Let $\tau=\pi_{4}^{-1}\sigma \pi_{4}=\theta\sigma^{-1}\theta^{-1}\sigma\theta\sigma\theta^{-1}$ and $q(x)=\pi_{4}^{-1}(q(x))=\theta\sigma^{-1}\theta^{-1}(q_{s}(x))$. 
Let $V$ be a left module over $R_{w}=K[x;\tau]/\ideal{q(x)}$. 

As in Example \ref{example:asymmetric-string}, but using Lemma \ref{lemma:symmetric-strings-R_w-basis}, the $R$-module $M(C_{w})\otimes V=\bigoplus_{i}b_{i}\otimes V$ may be depicted by the following diagram, where $b_{i}\otimes V$ is identified with ${}_{\pi_{i}^{-1}}V$ and where the loop on the right encodes the action of $x$ on $V$.
\[
\begin{tikzcd}[column sep=1.2cm,row sep=0.3cm]
V\arrow["s", ""']{r} & 
{}_{\sigma^{-1}}V\arrow["a", ""']{r} & 
{}_{\sigma^{-1}\theta^{-1}}V\arrow["s", ""']{r} & 
{}_{\sigma^{-1}\theta^{-1}\sigma^{-1}}V &
\tensor*[_{\sigma^{-1}\theta^{-1}\sigma^{-1}\theta}]{V}{}{}\arrow["", "a"']{l}\arrow[out=10,in=350,loop,  "s=x", ""',distance=1.5cm]
\end{tikzcd}
\]
\end{example}

\subsection{Asymmetric bands}
\label{s:asymband}
For an automorphism $\sigma$ of $K$ we write $K[x,x^{-1};\sigma]$ for the skew Laurent polynomial ring with $x\lambda = \sigma(\lambda)x$. Since $K$ is a division ring, $K[x,x^{-1};\sigma]$ is a principal left and right ideal domain by \cite[Theorem 1.4.5]{McCRob2001} (applied to this ring and its opposite, which is also a skew Laurent polynomial ring).

Let $w$ be an asymmetric band, so it is a periodic $\Z$-indexed word, without inversion symmetry. Letting it have period $n$, we have $w = {}^\infty u^\infty = \dots u|u u \dots$ for some word $u$ of length $n$.

\begin{lemma}\label{lemma:asymmetric-bands-R_w-basis}
There is a unique way to turn $M(C_w)$ into an $R$-$R_w$-bimodule, such that
\begin{itemize}
\item[(i)] $R_w = K[x,x^{-1};\tau]$ for $\tau\in\Aut(K)$,
\item[(ii)] the action of $x$ satisfies $b_i x = b_{i-n}$ for all $i$, and
\item[(iii)] the action extends the $R$-$K$-bimodule structure of $M(C_w)$.
\end{itemize}
Namely, we need $\tau = \pi_n^{-1}\pi_0$. As a right $R_w$-module, $M(C_w)$ is free with basis $b_i$ \emph{(}$i\in J_w$\emph{)} for $J_w=\{0,1,\dots,n-1\}$.
\end{lemma}

Note that $\pi_n^{-1}\pi_0 = \pi_0^{-1} \sigma_C \pi_0$, so $R_w \cong K[x,x^{-1};\sigma_C]$.

\begin{proof}
Since there is no inversion symmetry, the walk $C = C_w$ is also periodic, of the form ${}^\infty D^\infty$ where $D^*=u$.
For $i\in\Z$ and $\lambda\in K$ we must have $(b_i x)\lambda = b_i (x \lambda)$.
Now $(b_i x)\lambda = b_{i-n}\lambda = \pi_{i-n}(\lambda) b_{i-n}$ and
$b_i(x \lambda) = b_i (\tau(\lambda) x) = \pi_i(\tau(\lambda)) b_i x = \pi_i(\tau(\lambda)) b_{n-i}$.
Thus we need $\tau(\lambda) = \pi_i^{-1}\pi_{i-n}$. Note that this doesn't depend on $i$ since if an arrow $i \to j$ in $Q_{C}$ is sent to $a$ under $f_{C}$ then by periodicity $a$ can also be written as the image under $f_{C}$ of an arrow $ i-n \to j-n$, so $\sigma_a \pi_i = \pi_j$ and $\sigma_a \pi_{i-n} = \pi_{j-n}$, so $\pi_j^{-1} \pi_{j-n} = \pi_i^{-1}\pi_{n-i}$. In particular $\tau_w = \pi_n^{-1} \pi_0$.


To show $M(C)$ is an $R$-$R_{w}$ bimodule it suffices to check $a(b_{i}x)=(ab_{i})x$ for any arrow $a$ and any $i\in\Z$. As in the proof of Lemma \ref{lemma:symmetric-strings-R_w-basis}, this follows from straightforward case analysis, using that $C=C[n]$.

To see that the $b_{i}$ ($i\in J_{w}$) give an $R_{w}$-basis one may apply Lemma \ref{lemma:K-basis-of-M(C)-iff-C-relation-admissible}, as in Lemma \ref{lemma:symmetric-strings-R_w-basis}.
\end{proof}

\begin{example}
\label{example:asymmetric-band}
Recall $R=K_{\sigma,\theta}\langle s,a \rangle/\ideal{q_{s}(s),a^{2}}$ from Example \ref{example:symmetric-string}. 
Let $w = {}^{\infty} (s^{*} a s^{*} a s^{*} a^{-1})^{\infty}$, which is an asymmetric band with  $C_{w}={}^{\infty} (sa s a s a^{-1})^{\infty}$ and period $6$.  
As in Example \ref{example:symmetric-string} we take $\pi_{0}$ to be the identity, and let $\tau=\sigma\theta\sigma\theta\sigma\theta^{-1}$ so that
$R_{w}=K[x,x^{-1};\tau]$. 
Note $a(\lambda(b_{5}\otimes v))=
b_{0}\otimes x^{-1}(\tau\theta(\lambda)v)$. 
By Lemma \ref{lemma:asymmetric-bands-R_w-basis}, 
$M(C_{w})\otimes _{R_{w}} V$ may be depicted as follows.
\[
\begin{tikzcd}[column sep=1.2cm,row sep=0.2cm]
V & 
{}_{\sigma}V\arrow["", "s"']{l} & 
{}_{\sigma\theta}V\arrow["", "a"']{l} & 
{}_{\sigma\theta\sigma}V\arrow["", "s"']{l} & 
{}_{\sigma\theta\sigma\theta}V\arrow["", "a"']{l} & 
{}_{\sigma\theta\sigma\theta\sigma}V\arrow["", "s"']{l}\arrow[out=200,in=330, "", "a=x^{-1}"']{lllll} 
\end{tikzcd}
\]
\end{example}

\subsection{Symmetric bands}
\label{s:symband}
By Lemma \ref{lemma:form-of-symmetric-bands}, any symmetric band has the form $w = {}^\infty (v s^* v^{-1} u^{-1} t^* u)^\infty$ for some $s,t\in \Sp$ and words $u,v$ of lengths $p,r\geq 0$  respectively. Then $w$ has period $2n$ where $n=p+r+1$.

\begin{lemma}
\label{lemma:symmetric-bands-R_w-basis}
There is a unique way to turn $M(C_{w})$ into an $R$-$R_{w}$-bimodule, such that
\begin{itemize}
\item[(i)] $R_w = R_w' *_K R_w''$ with $R_w' =K[x;\rho]/\ideal{q_x(x)}$ and $R''_{w}=K[y;\tau]/\ideal{q_y(y)}$, where $\rho,\tau\in \Aut(K)$,  $q_x(x)=x^{2}-\beta x+\gamma\in K[x;\rho]$ and $q_y(y)=y^{2}-\mu y+\eta\in K[y;\tau]$,
\item[(ii)] the actions of $x$ and $y$ satisfy $b_{i}x=b_{2r+1-i}$ for $i\leq r$ and $b_{i}y=b_{-i-1-2p}$ for $i\geq -p$, and 
\item[(iii)] the action extends the $R$-$K$-bimodule structure of $M(C_w)$. 
\end{itemize}
Namely, we need $\rho=\pi_{r}^{-1}\sigma_{s}\pi_{r}$, $\tau=\pi_{-p}^{-1}\sigma_{t}\pi_{-p}$, 
$\beta=\pi_{r}^{-1}(\beta_{s})$, $\gamma=\pi_{r}^{-1}(\gamma_{s})$, 
$\mu=\pi_{-p}^{-1}(\beta_{t})$, $\eta=\pi_{-p}^{-1}(\gamma_{t})$, 
$b_i x = b_i \beta - b_{2r+1-i} \gamma$ for $i>r$, and
$b_i y = b_i \mu - b_{-i-1-2p} \eta$ for $i < -p$.
As a right $R_{w}$-module, $M(C)_{w}$ is free with basis $b_{i}$ ($i\in J_{w}$) for $J_{w}=\{-p,\dots,r\}$.
\end{lemma}

Note that $R_{w}$ is a free product as in Theorem \ref{theorem:symmetricbandparameterringisHNP}. 
As for symmetric strings, we have an isomorphism $K[x;\sigma_s]\to K[x;\rho]$ which sends
$\sum_{i=0}^n \lambda_i x^i$ to $\sum_{i=0}^n \pi_r^{-1}(\lambda_i) x^i$, and sending $q_s(x)$ to $q_x(x)$, so $q_x(x)$ is normal and $R_w' \cong K[x;\sigma_s]/\ideal{q_s(x)}$. Similarly $q_y(y)$ is normal in $K[y;\tau]$ and $R_w'' \cong K[x;\sigma_t]/\ideal{q_t(x)}$.

\begin{proof}
Letting $C=C_w$, there are walks $D,E$ such that $D^*=u$, $E^*=v$ and 
\[
C = \dots E s E^{-1} D^{-1} t D|E s^{-1} E^{-1} D^{-1} t^{-1} D E s^{-1} E^{-1} D^{-1} t^{-1} \dots.
\]
Now having an $R$-$R_w$-bimodule structure extending the given $R$-$K$-bimodule structure is equivalent to having an $R$-$R_w'$-bimodule structure and an $R$-$R_w''$-bimodule structure, both extending the $R$-$K$-bimodule structure. Observe that $C_{2r-i} = C_i^{-1}$ for $i\neq r+1$ and $C_{-i-2p} = C_i^{-1}$ for $i\neq -p$, so each of these is similar to the case of a symmetric string, and hence the assertion follows from the considerations in section~\ref{s:symstring}.

By Theorem \ref{theorem:symmetricbandparameterringisHNP} the  alternating monomials in $x$ and $y$ give a $K$-basis of $R_{w}$, which we label here by
\[
(\dots,z_{-3},z_{-2}, z_{-1},z_{0},z_{1},z_{2},z_{3},\dots)=(\dots,yxy,xy,y,1,x,yx,xyx,\dots),
\]
where $z_{0}=1$. 
In this notation, an induction shows that for any $i\in J_{w}$ we have $b_{i} z_{l} = b_{ln+i}$ for $l$ even and $b_{i} z_{l} = b_{l n + r-p-i}$ for $l$ odd. It follows that the elements $b_{i}$ ($i\in J_{w}$) give an $R_{w}$-basis of $M(C)$. 
\end{proof}

\begin{example}
\label{example:symmetric-band}
Let $R$ be the semilinear clannish algebra $K_{\sigma,\theta}\langle s,a \rangle/\ideal{q_{s}(s),a^{2}}$ from Example \ref{example:asymmetric-band}. 
Consider the symmetric band $w={}^{\infty}(s^{*}as^{*}a^{-1}s^{*}a^{-1}s^{*}a)^{\infty}$ whose canonically associated walk is $C_{w}={}^{\infty}(sasa^{-1}s^{-1}a^{-1}sa)|(sas^{-1}a^{-1}s^{-1}a^{-1}s^{-1}a)^{\infty}$. 
In the notation above Lemma \ref{lemma:symmetric-bands-R_w-basis}, we have $D=a$, $E=sa$, $r=2$ and $p=1$. 
Again take $\pi_{0}$ to be the identity.   
Let $\rho=\sigma\theta\sigma\theta^{-1}\sigma^{-1}$, $\tau=\theta^{-1}\sigma\theta$, 
$q_x(x)=\sigma\theta(q_{s}(x))$, $q_y(y)=\theta^{-1}(q_{s}(y))$.
Let $V$ be a left module over $R_{w}=K[x;\rho]/\ideal{q_x(x)}*_{K}K[y;\tau]/\ideal{q_y(y)}$. 
By Lemma \ref{lemma:symmetric-bands-R_w-basis} the module $M(C_{w})\otimes _{R_{w}} V$ may be depicted by
\[
\begin{tikzcd}[column sep=1.2cm,row sep=0.3cm] 
\tensor*[_{\theta^{-1}}]{V}{} \arrow[out=165,in=195,loop,  "", "s=y"',distance=1.5cm] & 
{}_{}V\arrow["", "a"']{l} & 
{}_{\sigma}V\arrow["", "s"']{l} &
\tensor*[_{\sigma\theta}]{V}{}{}\arrow["", "a"']{l}\arrow[out=15,in=345,loop,  "s=x", ""',distance=1.5cm]
\end{tikzcd}
\]
\end{example}

\section{Proof of the main theorem}
\subsection{Semilinear Relations}
Linear relations have already been considered, see for example \cite[\S 2]{Cra1988ff} and \cite[\S 4]{Cra1989}. Let $\sigma$ be an automorphism of the division ring $K$, and let $V$ and $W$ be left $K$-modules. A \emph{$\sigma$-semilinear relation} $C:V\to W$ is by definition a $K$-submodule of $V \oplus {}_\sigma W$, where ${}_\sigma W$ is the $K$-module obtained from $W$ by restriction via $\sigma$. Any $\sigma$-semilinear map $\theta:V\to W$ defines a $\sigma$-semilinear relation via its graph $C = \{ (v,\theta(v)) \mid v\in V\}$. Conversely, thinking of a $\sigma$-semilinear relation $C$ as a generalization of a mapping, we write $w\in Cv$ (and in diagrams we use $C$ to label an arrow $v\xrightarrow{} w$) to mean that $(v,w)\in C$. Thus if $w \in Cv$ and $w'\in Cv'$ then $w+w'\in C(v+v')$ and $\sigma(\lambda) w \in C(\lambda v)$ for $\lambda\in K$.

Let $C:V\to W$ be a $\sigma$-semilinear relation. If $U$ is a subset of $V$, we define $C U = \bigcup_{u\in U} C u$. Observe that if $U$ is a $K$-submodule of $V$, then $CU$ is a $K$-submodule of $W$. 
%
%
We write $C^{-1} = \{ (w,v) : (v,w) \in C \}$; it is a $\sigma^{-1}$-semilinear relation from $W$ to $V$.
%
%
If $D:U\to V$ is a $\tau$-semilinear relation, then 
\[
CD = \{ (u,w) \in U\oplus W \mid \text{$(u,v)\in D$ and $(v,w)\in C$ for some $v\in V$}\}
\]
is a $\sigma\tau$-semilinear relation $U\to W$.

Let $C$ be a $\sigma$-semilinear relation \emph{on} $V$, that is, of the form $V\to V$. 
We define  subsets  $C'\subseteq C''$ of $V$ by
\begin{align*}
C'' &= \{v\in V\mid \text{there exists } v_{0},v_{1},\dots \in V\text{ such that }v=v_{0}\text{ and  }v_{i}\in Cv_{i+1}\text{ for all }i\}, \quad \text{and}
\\
C' &=\{v\in V\mid \text{there exists } v_{0},v_{1},\dots \in V\text{ such that }v=v_{0},v_{i}\in Cv_{i+1}\text{ for all }i \text{ and  } v_{i}=0\text{ for }i\gg 0\}.
\end{align*}

\begin{lemma}
\label{stable-image-kernel-are-K-submodules}
Let $C$ be a $\sigma$-semilinear relation on $V$.
Then $C',C''$ are $K$-submodules of $V$.

Furthermore if $V$ is finite-dimensional over $K$ then the following statements hold.
\begin{enumerate}
    \item We have the equations $C''=C'+C''\cap(C^{-1})''$ and $(C^{-1})''=(C^{-1})'+C''\cap(C^{-1})''$.
    \item We have the inclusions $C''\cap(C^{-1})'\subseteq C'$ and $C'\cap(C^{-1})''\subseteq (C^{-1})'$.
\end{enumerate}
\end{lemma}

\begin{proof}
Let $D=C^{-1}$. 
Fix $v\in C''$ and $\lambda\in K$. By definition there exists  $v_{0},v_{1},\dots \in V$ such that $v=v_{0}$ and  $v_{i}\in C v_{i+1}$ for all $i$. Since $C$ is $\sigma$-semilinear we have $(\sigma^{-1}(\mu)v_{i+1},\mu v_{i})\in C$ for each $i$ and each $\mu \in K$.  Letting $v'_{i}=\sigma^{-i}(\lambda)v_{i}$ for all $i$ gives a sequence $v'_{0},v'_{1},\dots \in V$ such that $v'_{0}=\lambda v$ and $v'_{i}\in Cv'_{i+1}$ for all $i$. 
Hence $C''$ is a $K$-subspace of $V$. Since $v_{i}=0$ implies $v'_{i}=0$, $C'$ is also a $K$-subspace. 
We now assume $V$ is finite-dimensional,
and so $C''=C^{d}M$, $C'=C^{d}0$, $D''=D^{d}M$ and $D'=D^{d}0$ for some $d$.  

(i) 
Let $v\in C''$ with  $v_{0},v_{1},\dots \in V$ as above. 
Since $v_{d}\in D^{d}v_{0}$ and $D^{d}M=D^{2d}M$, there is a sequence $u_{-d},\dots,u_{0},\dots u_{d}\in V$ where $v_{d}=u_{d}$ and $u_{i}\in Du_{i-1}$ for $-d<i\leq d$. Note that $u_{0}\in C^{d}M\cap D^{d}M=C'' \cap D''$. Taking the difference of $v_{0},u_{0}\in C^{d}u_{d}$ gives $v_{0}-u_{0}\in C^{d}0$ and hence $v=u_{0}+v_{0}-u_{0}$ which lies in $C'+C''\cap(C^{-1})''$.  This gives the first equality; the second follows by replacing $C$ with $C^{-1}$. 

(ii) As above, by symmetry we can just prove the first inclusion. Let $u\in C''\cap(C^{-1})'$, giving a sequence $u_{-d},\dots,u_{0},\dots u_{d}\in V$ where $u=u_{0}$ and $u_{i}\in Cu_{i+1}$ for $-d\leq i<d$. Note $u_{d}\in D^{2d}0=D^{d}0$. Taking the difference of $u_{0},0\in C^{d}u_{d}$ gives $u_{0}\in C^{d}0$ and hence $u\in C'$.
%
%
%
%
\end{proof}



Let $q(x)\in K[x;\sigma]$ be a monic quadratic polynomial, say $q(x)=x^2-\beta x+\gamma$. Let $X$ be a $\sigma$-semilinear relation on $V$. We say that $X$ is \emph{$q(x)$-bound} if $w\in X v $ implies that $\beta w - \gamma v\in X w$. Written another way, $X$ is a $q(x)$-bound if and only if $(v,w)\in X$ implies $(w,\beta w-\gamma v)\in X$. See \cite[\S4]{Cra1989}. For example if $X$ is the graph of a mapping $\theta$, then it is $q(x)$-bound if and only if $q(\theta)=0$. If $q(x)$ is normal then, by Lemma \ref{lemma:characterising_central_skew_quadratics},  $X$ becomes a $K[x;\sigma]$-module, with the action of $x$ given by $x(v,w) = (w,\beta w - \gamma v)$. Moreover, and likewise, the element $q(x)$ acts on $X$ as zero, so $X$ becomes a module for the quotient $K[x;\sigma]/\ideal{q(x)}$.
%

\begin{lemma}
\label{inverting-bound-semilinear-relations}
Let $q(x)=x^{2}-\beta x+\gamma\in K[x;\sigma]$ and let $X$ be a $\sigma$-semilinear relation on $V$.

(i) If $q(x)$ is non-singular and normal, then $X$ is $q(x)$-bound if and only if the $\sigma^{-1}$-semilinear relation $X^{-1}$ is $q'(y)$-bound, where $q'(y)=y^{2}-\gamma^{-1}\beta y+\gamma^{-1}$.

(ii) If $X$ is $q(x)$-bound and $C:V\to W$ is a $\tau$-semilinear relation, then the 
$\tau^{-1} \sigma \tau$-semilinear relation $CXC^{-1}$ on $W$ is $p(y)$-bound, where $p(y) = y^2 - \tau(\beta) y + \tau(\gamma)$.
\end{lemma}

\begin{proof}
(i)  follows by combining Lemma \ref{lemma:characterising_central_skew_quadratics} and Corollary \ref{corollary:twisting-(and-inverting-non-singular)-preserves-normal/central-quadratics}, and (ii) is straightforward.
%
%
%
%
%
%
%
\end{proof}

\begin{lemma}
\label{lemma:oneqrelation-1}
Let $\lambda \in K$ with $\lambda\neq0$. Let $u\in Xw$ where $X$ is a $\sigma$-semilinear $q(x)$-bound relation on $V$.
\begin{itemize}
\item[(i)] We have $\lambda u\in X^{-1}(\mu w+\eta u)$ for some $\mu,\eta\in K$ with $\mu\neq0$. 
\item[(ii)] If $q(x)$ is non-singular then we have $\mu u\in X^{-1}(\lambda w+\eta u)$ for some $\mu,\eta\in K$  with $\mu\neq0$. 
\end{itemize}
\end{lemma}

\begin{proof}
Let $q(x)=x^{2}-\beta x+\gamma$. For (i) let $\mu=-\gamma\sigma^{-1}(\lambda),\eta=\beta\lambda$; for (ii) let $\mu =-\gamma^{-1}\sigma(\lambda),\eta=\beta \mu$. 
\end{proof}

\begin{lemma}
\label{lemma:oneqrelation}
Let $X$ be a $\sigma$-semilinear $q(x)$-bound relation on $V$. 
If $U \subseteq W \subseteq V$ are $K$-submodules, then
\begin{itemize}
\item[(i)]
$U \cap X W \subseteq U \cap X^{-1} W$, so in particular $W \cap X W \subseteq W \cap X^{-1} W$, and
\item[(ii)]
$(W \cap X U) + (U \cap X W) \subseteq (W \cap X^{-1} U) + (U \cap X^{-1} W)$.
\end{itemize}
Equality holds in all of these inclusions if $q(x) = x^2-\beta x+\gamma\in K[x;\sigma]$ is non-singular and normal.
\end{lemma}

\begin{proof}
Taking $\lambda=1$ in Lemma \ref{lemma:oneqrelation-1}(i), part (i) is immediate.  For (ii), note that by (i) it suffices to show that $W \cap X U \subseteq (W \cap X^{-1} U) + (U \cap X^{-1} W)$.
So let $w \in W\cap X U$. 
Now $w\in X u$, for some $u\in U$, so $\beta w - \gamma u \in X w$. By semilinearity, $\beta w \in X (\sigma^{-1}(\beta)u)$. Thus $\gamma u \in X( \sigma^{-1}(\beta)u - w)$. Thus $\sigma^{-1}(\beta)u - w \in W \cap X^{-1}U$. Now $\sigma^{-1}(\beta)u \in U \cap X^{-1} W$, and hence $w \in (W \cap X^{-1} U) + (U \cap X^{-1} W)$.

Now if $q(x)$ is non-singular and normal, then the relation $X^{-1}$ is $q'(y)$-bound in the sense of Lemma \ref{inverting-bound-semilinear-relations}, and the same results applied to $X^{-1}$ give the reverse inclusions.
\end{proof}

\begin{lemma}\label{lemma:rewriting-relations-symmetric-bands}
Let $r(x)\in K[x;\rho]$ and $p(y)\in K[y;\tau]$ be monic, normal and non-singular quadratics.
Let $X$ be a $\rho$-semilinear $r(x)$-bound relation, and $Y$ a $\tau$-semilinear $p(y)$-bound relation, on $V$. 
Then we have 
\begin{equation}\label{top-functor-equations-symmetric-band-1}
    \begin{split}
    (Y^{-1}X^{-1})''\cap (X^{-1}Y^{-1})''=(Y^{-1}X^{-1})''\cap (XY)''=(YX)''\cap (X^{-1}Y^{-1})''=(YX)''\cap (XY)''
    \end{split}
\end{equation}
and
\begin{equation}\label{top-functor-equations-symmetric-band-2}
    \begin{split}
    (Y^{-1}X^{-1})'\cap (X^{-1}Y^{-1})''=(Y^{-1}X^{-1})'\cap (XY)'',\\
    (YX)'\cap (X^{-1}Y^{-1})''=(YX)'\cap (XY)'',\\
     (Y^{-1}X^{-1})''\cap (X^{-1}Y^{-1})'=(YX)''\cap (X^{-1}Y^{-1})',\\
      (Y^{-1}X^{-1})''\cap (XY)'=(YX)''\cap (XY)'.
\end{split}
\end{equation}
Furthermore if $V$ is finite-dimensional then the sets listed in \eqref{top-functor-equations-symmetric-band-2} above are all equal, and they are equal to 
\begin{equation}\label{bottom-functor-equations-symmetric-band}
   \begin{split}
        (Y^{-1}X^{-1})'\cap (X^{-1}Y^{-1})'=(YX)'\cap (X^{-1}Y^{-1})'=(Y^{-1}X^{-1})'\cap (XY)'=  (YX)'\cap (XY)'.
   \end{split}
\end{equation}
\end{lemma}

\begin{proof}
Let $W_{-}=(Y^{-1}X^{-1})''$ and $W_{+}=(X^{-1}Y^{-1})''$ which means $ W_{+}=X^{-1}W_{-}$ and $W_{-}=Y^{-1}W_{+}$. 
Since $X$ and $Y$ are semilinear relations bound by non-singular quadratics, by Lemma \ref{lemma:oneqrelation} we have that the intersection $W=W_{-}\cap W_{+}$ is equal to both $W_{-} \cap XW_{-}$ and $W_{+}\cap Y W_{+}$. 
Assuming $v\in W$ gives $v\in X v_{1}$ for some $v_{1}\in W_{-} \cap X^{-1}W_{-}=W$, which gives $v_{1}\in Yv_{2}$ for some $v_{2}\in Y^{-1}W_{+}\cap W_{+}=W$. 
Repeating this argument defines a sequence $v_{i}\in V$ ($i>0$) such that $v_{0}=v$, $v_{i}\in Yv_{i+1}$ for $i$ odd and $v_{i}\in X v_{i+1}$ for $i$ even. 
Hence we have $W\subseteq (XY)''$ and by symmetry this gives all of the equalities in \eqref{top-functor-equations-symmetric-band-1}.   

Let $U_{-}= (Y^{-1}X^{-1})'$ which lies in $W_{-}$, so $U_{-}\cap W_{+}\subseteq U_{-}\cap (XY)''$ by \eqref{top-functor-equations-symmetric-band-1}. This gives the first of the eight inclusions required for \eqref{top-functor-equations-symmetric-band-2}. The remaining inclusions follow by symmetry.

For the final claim consider the $\rho\tau$-semilinear relation $C=XY$ on $V$. 
By \eqref{top-functor-equations-symmetric-band-2} we have that $U_{-}\cap W_{+}=(C^{-1})'\cap C''$, which is contained in $C'$ by Lemma \ref{stable-image-kernel-are-K-submodules}. 
To see that $C'\cap(C^{-1})'$ is contained in $(X^{-1}Y^{-1})'$, the proof of \cite[\S 4.5, Lemma]{Cra1988ff} may be adapted to our situation; see Lemma \ref{lemma:oneqrelation-1}.
\end{proof}

\subsection{The functors $C^\pm$ and one-sided filtrations}
Recall that $R = K_{\boldsymbol{\sigma}} Q/I$. We consider (covariant) functors from the category of $R$-modules (or finite-dimensional $R$-modules) to the category of $K$-modules. 
Recall that if $F$ is such a functor, then a \emph{subfunctor} $G$ is given by fixing a $K$-subspace $G(M)$ of $F(M)$ for each $R$-module $M$, in such a way that $F(\theta)(G(M))\subseteq G(M')$ for any homomorphism $M\to M'$. Given a vertex $\ell$ in $Q$, the \emph{$\ell$-forgetful functor} is the functor sending an $R$-module $M$ to the $K$-module $e_\ell M$ (a subfunctor of the forgetful functor sending $M$ to $M$ as a $K$-module).

Let $M$ be an $R$-module. If $a$ is an arrow in $Q$, then the action of $a$ defines a $\sigma_a$-semilinear mapping $e_{t(a)} M\to e_{h(a)} M$, which we also denote by $a$. Considering this as a relation $e_{t(a)} M\to e_{h(a)} M$, we have the inverse relation $a^{-1} : e_{h(a)} M\to e_{t(a)} M$. Using products of relations, any finite walk $C$ induces a $\sigma_C$-semilinear relation $e_{\ell'} M \to e_\ell M$ where $\ell$ is the head of $C$ and $\ell'$ is the head of $C^{-1}$. It is easy to see that the assignments sending a module $M$ to $C\, 0$ or $C\, e_{\ell'} M$ define subfunctors of the $\ell$-forgetful functor.







\begin{definition}
Let $C$ be a walk with $C^* \in H(\ell,\epsilon)$. If $M$ is an $R$-module, we define $C^\pm(M)$ as follows. If $C$ is a finite walk, say with $C^{-1}$ having head $\ell'$, then 
\begin{align*}
C^+(M) &= 
\begin{cases}
Ca^{-1} \, 0 & \text{(if there is an arrow $a$ with $Ca^{-1}$ a walk)} \\
C \, e_{\ell'} M & \text{(otherwise)}
\end{cases}
\\
C^-(M) & =
\begin{cases}
C b \, e_{t(b)} M & \text{(if there is an arrow $b$ with $C b$ a walk)} \\
C \, 0 & \text{(otherwise).} 
\end{cases}
\end{align*}
If $C$ is an infinite walk, then
\begin{align*}
C^+(M) &= \{ x\in e_\ell M : \text{there exist $x=x_0,x_1,\dots$ such that $x_{i-1} \in C_i x_i$ for all $i$} \}, \\
C^-(M) &= \{ x\in e_\ell M : \text{there exist $x=x_0,x_1,\dots$ such that $x_{i-1} \in C_i x_i$ for all $i$ and $x_i = 0$ for $i\gg 0$} \}.
\end{align*}
\end{definition}





We say that a walk $D$ is \emph{special-direct} if special loops only appear in it as direct letters, and \emph{special-inverse} if $D^{-1}$ is special-direct. Given a word $w$, we denote by $D_w$ the special-direct walk with $D_w^*=w$ and by $D'_w$ the special-inverse walk with $(D'_w)^*=w$. (Warning: this conflicts with the notation in \cite{Cra1988ff}.)



\begin{lemma}
\label{lemma:one-sided-filtration}
Let $\ell$ be a vertex in $Q$, let $\epsilon = \pm1$, and let $M$ be an $R$-module.
\begin{itemize}
\item[(i)]
If $w\in H(\ell,\epsilon)$ then $D_w^-(M)\subseteq D_w^+(M) \subseteq e_\ell M$.
\item[(ii)]
If $w,z \in H(\ell,\epsilon)$ and $w<z$ then $D_w^+(M) \subseteq D_z^-(M)$.
\item[(iii)]
Suppose $M$ is finite-dimensional. 
If $S$ is a non-empty subset of $e_\ell M$ with $0\notin S$, then there is some $w\in H(\ell,\epsilon)$ such that $S$ meets $D_w^+(M)$ and does not meet $D_w^-(M)$.
\end{itemize}
Moreover the same statements hold for the special-inverse walks $D'_w$.
\end{lemma}

\begin{proof}
(i) Let $C=D_{w}$. We can and do assume there exist arrows $a,b$ such that $Ca^{-1}$ and $Cb$ are walks. 
Hence $w$ is finite, say of length $n$. 
Let $u'=(Ca^{-1})^{*}$ and $u=(Cb)^{*}$ which by assumption are both words. 
Since $w$ is end-admissible, $a$ and $b$ must both be ordinary by Lemma \ref{lem-end-adm}. 
%
%
Since the letters $u_{n+1}'=a^{-1}$ and $u_{n+1}=b$ have the same sign, $ab$ is a zero-relation defining the algebra, and so $C^{-}(M)\subseteq C^{+}(M)$. 

(ii)  See for example the lemma on \cite[p. 23]{Rin1975}.  Taking $z=w'$ in Definition \ref{definition:order-on-words}, for  case (a) note that $xy$ must be a zero-relation defining the algebra. 
The proof for cases (b) and (c) are straightforward.

(iii) Since $M$ is finite-dimensional we have that, for any infinite walk $C$, $C^{+}(M)=C_{\leq n}M$ and $C^{-}(M)=C_{\leq n}0$ for some $n>0$. 
From here one can follow the proof of \cite[Lemma 10.3]{Cra2018}. 
\end{proof}

\subsection{Top and bottom functors}
We continue to study functors from the category of $R$-modules to the category of $K$-modules. 
Let $C$ be an end-admissible walk $C$, say $I$-indexed. We have a functor $T_C = \Hom_R(M(C),-)$, where we use the right action of $K$ on $M(C)$. 

Given any $R$-module $M$ we define $K$-subspaces $B_C^+(M)$ and $B_C^-(M)$ of $T_C(M)$ as follows. 
\begin{align*}
B_C^+(M) &= \begin{cases}
\{ \theta \in \Hom_R(M(C),M) \mid \text{$\theta(b_i)=0$ for $i\ll 0$} \} & 
\text{if $I$ is not bounded below,}
\\
\{ \theta \in \Hom_R(M(C),M) \mid \theta(b_0)\in 1_{\ell,-\epsilon}^-(M) \} & 
\text{if $C$ has head $\ell$, sign $\epsilon$, and}
\end{cases}
\\
B_C^-(M) &= \begin{cases}
\{ \theta \in \Hom_R(M(C),M) \mid \text{$\theta(b_i)=0$ for $i\gg 0$} \} & 
\text{if $I$ is not bounded above,}
\\
\{ \theta \in \Hom_R(M(C),M) \mid \theta(b_n)\in 1_{\ell,-\epsilon}^-(M) \} & 
\text{if $C^{-1}$ has head $\ell$, sign $\epsilon$ and $n=\max I$.}
\end{cases}
\end{align*}
We define $B_C(M) = B_C^+(M)+B_C^-(M)$. Clearly $B_C^+$, $B_C^-$ and $B_C$ define subfunctors of $T_C$, and we define $F_C = T_C/B_C$. We normally consider $F_C$ as a functor from the category of $R$-modules (or finite-dimensional $R$-modules) to the category of $K$-modules. In view of the following lemma, if $w$ is a string or band, we may also consider $F_{C_w}$ as a functor to the category of $R_w$-modules (or finite-dimensional $R_{w}$-modules).

\begin{lemma}
\label{lemma:B_C-is-an-R_w-submodule}

Let $w$ be a string or band with canonically associated walk $C$. 
Let $M$ be an $R$-module. 
If we consider $T_{C}(M) = \Hom_R(M(C),M)$ as a left $R_{w}$-module using the right action of $R_{w}$ on $M(C)$, then $B_{C}(M)$ is an $R_{w}$-submodule of $T_{C}(M)$.
\end{lemma}
\begin{proof}

Let $C=C_{w}$. 
Fixing $\theta\in B_{C}(M)$ and $z\in R_{w}$ we require that $z\theta\in B_{C}(M)$. 
Without loss of generality we can assume $\theta\in B^{-}_{C}(M)$. 
Note that for any $w$ we have $K\subseteq R_{w}$. 
It is straightforward to check that $B_{C}(M)$ is a $K$-submodule of $T_{C}(M)$. 
%
%
In particular, there is nothing more to prove when $w$ is an asymmetric string. We now consider the remaining cases for $w$.

(1) 
Let $w$ be a symmetric string of length $n$, and so $R_{w}=K[x;\sigma]/\ideal{q(x)}$ for some monic, normal and non-singular quadratic $q(x)$. 
It suffices to assume $z=x$. Since $b_{0}x=b_{n}$ it is trivial that $\theta(b_{n})=0$ implies $(x\theta)(b_{0})=0$, and likewise $\theta(b_{n})\in a(M)$ implies $(x\theta)(b_{0})\in a(M)$. 
Hence $x\theta \in B^{+}_{C}(M)$.

For cases (2) and (3) $w$ is a band, and so $I=\Z$ and there is some $j\in \Z$ such that $\theta(b_{i})=0$ for all $i\geq j$.

(2) 
If $w$ is a asymmetric band then $R_{w}=K[x,x^{-1};\sigma]$ for some automorphism $\sigma$ of $K$. Since $B_{C}(M)$ is a $K$-submodule it suffices to assume $z=x^{-1}$ or $z=x$. Assuming $w$ has period $n$ we have that $b_{i}x^{\pm 1}=b_{i\mp n}$ for all $i$, and so $(x^{\pm 1}\theta)(b_{i})=0$ for all $i\geq j\pm n$, which shows $x^{\pm 1}\theta\in B_{C}^{-}(M)$.

(3) 
If $w$ is a symmetric band then $R_{w}=R_{w}'*_{K} R_{w}''$ where $R_{w}'=K[x;\rho]/\ideal{r(x)}$, $R_{w}''=K[y;\tau]/\ideal{p(y)}$ for some monic, normal and non-singular quadratics $r(x),p(y)$. 
In particular, by Lemma \ref{lemma:invertible-variable-in-quotient-by-non-singular-quadratic}: there exist inverses $x^{-1}$ and $y^{-1}$ of $x$ and $y$ in $R_{w}$; the element $x^{-1}$ can be written as a $K$-linear combination of $x$ and $1$; and the element $y$ can be written as a $K$-linear combination of $y^{-1}$ and $1$. 
Thus, since $B_{C}(M)$ is a $K$-submodule, it suffices to show $x\theta,y^{-1}\theta\in B_{C}(M)$. 

By Lemma \ref{lemma:form-of-symmetric-bands} there are integers $p,r\geq 0$ such that $b_{i}y^{-1}=b_{-i-1-2p}$ for $i\leq -1-p$ and $b_{i}x=b_{2r+1-i}$ for $i\leq r$.
Letting $h$ denote the minimum of $-1-p$ and $-j-1-2p$, we have by construction that $(y^{-1}\theta)(b_{i})=0$ for all $i\leq h$, which shows that $y^{-1}\theta\in B_{C}(M)$. 
Similarly $x\theta \in B_{C}(M)$.
\end{proof}

Let $D,E$ be walks such that one of the words $D^{*},E^{*}$ lies in $H(\ell,1)$ and the other lies in $H(\ell,-1)$. Hence the product  $D^{-1}E$ is a walk. 
For such $D$ and $E$ we define functors $T_{D,E}$,  $B_{D,E}$ and $F_{D,E}$, from the category of $R$-modules to that of $K$-modules, by
\[
T_{D,E}(M) = D^+ (M)\cap E^+(M),
\hspace{0.5em}
B_{D,E}(M) = D^+(M) \cap E^-(M) + D^-(M) \cap E^+(M),
\hspace{0.5em}
F_{D,E} = T_{D,E}/B_{D,E}.
\]
Clearly $T_{D,E}$,  $B_{D,E}$ and $F_{D,E}$ are symmetric in $D$ and $E$. Given an end-admissible walk $C$, say $I$-indexed, and $i\in I$, we define functors $T_{C,i}$, $B_{C,i}$ and $F_{C,i}$ of the same form, by 
\[
T_{C,i} = T_{C_{>i},(C_{\le i})^{-1}},
\quad
B_{C,i} = B_{C_{>i},(C_{\le i})^{-1}},
\quad
F_{C,i} = T_{C,i}/B_{C,i}.
\]
%
%
Given a functor $F$ from the category of $R$-modules to that of $K$-modules, and given an automorphism $\sigma$ of $K$, we write ${}_\sigma F$ for the composition of $F$ with the functor of restriction by $\sigma$.
We call it a \emph{twist} of $F$.

\begin{lemma}\label{lemma:hom-functors-relation-functors}
Let $C$ be an end-admissible $I$-indexed walk and let $i\in I$. For each $R$-module $M$ let $\xi_{M}$ be the map $T_{C}(M)\to M$ sending $\theta$ to $\theta(b_i)$. The following statements hold.
\begin{itemize}
\item[(i)]
$\xi_M$ is a $\pi_i$-semilinear map, so a $K$-linear map $T_C(M)\to {}_{\pi_i} M$.
\item[(ii)]
The image of $\xi_{M}$ is $T_{C,i}(M)$.
\item[(iii)]
The image of the restriction of $\xi_{M}$ to $B_C(M)$ is  $B_{C,i}(M)$.
\item[(iv)]
The kernel of $\xi_{M}$ is contained in $B_C(M)$.
\item[(v)]
The maps $\xi_{M}$ define a natural transformation $T_C\to {}_{\pi_i} T_{C,i}$.
\item[(vi)]
There is an induced natural isomorphism $F_C \to {}_{\pi_i} F_{C,i}$.
\end{itemize}
\end{lemma}

\begin{proof}
Let $D=C_{>i}$ and $E=(C_{\leq i})^{-1}$ so that $T_{C,i}=T_{D,E}$, $B_{C,i}=B_{D,E}$ and $C=E^{-1}D$. 
Suppose that $D$ is $J$-indexed and $E$ is $H$-indexed.

(i) For $\lambda\in K$ we have $(\lambda \theta)(b_i) = \theta(b_i \lambda) = \theta (\pi_i(\lambda) b_i) = \pi_i(\lambda) \theta(b_i)$.

(ii) Let $\theta \in T_{C}(M)$ and $m=\theta(b_{i})$. 
For each $j\in I$ let $m_{j-i}=\theta(b_{j})$ so that $m=m_{0}$. 
From here it is straightforward to see that $\Ima (\xi_{M})\subseteq T_{D,E}(M)$. 
%

We now prove $\Ima (\xi_{M})\supseteq T_{D,E}(M)$, so let $m\in T_{D,E}(M)$. 
Since $m\in D^{+}(M)$ there is a sequence $m_{j}\in M$ ($j\in J$) such that $m=m_{0}$,  $m_{j}\in D_{j}m_{j+1}$ whenever $j+1\in J$ and (if $n=\max J$ and $Da^{-1}$ is a walk then $am_{n}=0$). Likewise there is a sequence $m'_{h}\in M$ ($h\in H$) such that $m=m_{0}'$, $m_{h}'\in E_{h}m_{h+1}'$ whenever $h+1\in H$ and (if $l=\max H$ and $Ec^{-1}$ is a walk then $cm_{l}'=0$). Now define $\theta$ by setting $\theta(b_{i+j})=m_{j}$ for all $j\in J$ and $\theta(b_{i-h})=m_{h}'$ for all $h\in H$. It is straightforward to check that $\theta(ab_{d})=a\theta(b_{d})$ for any $d\in I$ and any arrow $a$ in $Q$. 
Hence $\theta$ is $R$-linear, and so $\Ima (\xi)\supseteq T_{D,E}(M)$. 

(iii) In the notation above, it is straightforward to see that $m\in T_{D,E}(M)$ implies  $\theta\in B_{C}(M)$. 

Conversely, we claim $\theta\in B^{+}_{C}(M)$ implies $\theta(b_{i})\in E^{-}(M)\cap D^{+}(M)$. 
As above let $m_{j}=\theta(b_{i+j})$ and $m_{h}'=\theta(b_{i-h})$ for all $j\in J$ and $h\in H$.
If $I$ is unbounded below then $H$ in unbounded above, in which case the sequence $m'_{0},m'_{1},\dots$ is eventually zero. 
If instead $I$ is bounded below by $0$, then $H$ is bounded above by $i$, and we have $m'_{i}=\theta(b_{0})$.  Note also that $Ea$ is a walk if and only if $a^{-1}C$ is a walk, in which case $\theta(b_{0})\in a(M)$ if and only if $m_{i}'\in a(M)$; and otherwise $\theta(b_{0})=0=m'_{i}$. 

Thus our claim holds. Hence, and by symmetry, $\xi_{M}$ sends elements of $B^{\pm}_{C}(M)$ to $E^{\mp}(M)\cap D^{\pm}(M)$. 

(iv) Let $\theta:M(C)\to M$ be an $R$-module homomorphism with $\theta(b_{i})=0$. Define $K$-module homomorphisms $\theta_{\pm}:M(C)\to M$ by setting $\theta_{\pm}(b_{j})=\theta(b_{j})$ for $\pm(i-j)<0$ and $\theta_{\pm}(b_{j})=0$ for $\pm(i-j)\geq0$. It is straightforward to check that, since $\theta$ is $R$-linear, so too is $\theta_{\pm}$.  
It is trivial that $\theta=\theta_{-}+\theta_{+}$ and that $\theta_{\pm}\in B_{C}^{\pm}(M)$, and so together we have shown $\Ker(\xi_{M})\subseteq B_{C}(M)$.

(v) Straightforward.

(vi) This follows from (ii), (iii), (iv) and (v).
\end{proof}



\subsection{Orientation changes}
We have defined various functors using a walk $C$. In this section we show that sometimes they only depend on the underlying word $C^*$. In order to study the functors, we evaluate them on a fixed $R$-module $M$.
Given two walks $C,D$ with the same underlying $I$-indexed word $w$, and $i\in I$, we write $C \sim_i D$ to mean that $T_{C,i} = T_{D,i}$ and $B_{C,i} = B_{D,i}$.

\begin{definition}
We adapt a notion from \cite{Cra1988ff}. Let $w$ be an $I$-indexed end-admissible word and let $i\in I'$ be such that $w_i$ is a $*$-letter. We define the \emph{norm} $\norm{i}\in \N \cup\{\infty\}$ of $i$ by the rule that
$\norm{i} \ge n$ if and only if there are words $x,y,z$, with $x$ of length $n$, such that $(w_{\leq i-1})^{-1}=xy$ and $w_{>i}=xz$.
Note that if $i$ is not a symmetry for $w$, then $\norm{i}$ is finite.
\end{definition}


\begin{lemma}\label{lemma:orientation_change_norm_induction_step}
Let $w$ be an $I$-indexed end-admissible word and suppose that $i\in I'$ is such that $w_i$ is a $*$-letter and $i$ is naturally inverse. For any $j$ with $0 < j \le \norm{i}$ such that $w_{i-j}$ and $w_{i+j}$ are  $*$-letters, either $i+j$ is naturally direct with $\norm{i+j}< \norm{i}$, or $i-j$ is naturally direct with $\norm{i-j}< \norm{i}$.
%
\end{lemma}

\begin{proof}
Same as in the proof of \cite[Lemma A]{Cra1988ff}. 
\end{proof}

\begin{lemma}
\label{lemma:swapset}
Let $w$ be an $I$-indexed end-admissible word and let $C$ and $D$ be walks with $C^*=D^*=w$. Let $\Delta = \{ i\in I' \mid C_i\neq D_i \}$, let $k\in I$, and suppose the following hold.
\begin{itemize}
\item[(A1)]
All $i\in \Delta$ are naturally inverse for $w$;
\item[(A2)]
For distinct $i,j\in \Delta$, we have $|i-j| > \norm{i} + \norm{j} + 1$;
\item[(A3)]
For $i\in \Delta$ and $0<j\le \norm{i}$, we have $C_{i+j} = C_{i-j}^{-1}$ (or equivalently $D_{i+j} = D_{i-j}^{-1}$);
\item[(A4)]
There is no $i\in \Delta$ with $i \le k \le i+\norm{i}$.
\end{itemize}
Then $C \sim_k D$, so $F_{C,k} \cong F_{D,k}$, and hence $F_C$ and $F_D$ are isomorphic up to twist.
\end{lemma}

\begin{proof}
%
By symmetry it suffices to show inclusions $T_{C,k}(M)\subseteq T_{D,k}(M)$
and $B_{C,k}(M)\subseteq B_{D,k}(M)$. This gives $F_{C,k} \cong F_{D,k}$, and hence $F_C$ and $F_D$ are isomorphic up to twist by Lemma~\ref{lemma:hom-functors-relation-functors}.

Given $m\in T_{C,k}(M)$, there is a tuple of elements $(m_i)_{i\in I}$ of $M$, such that 
\begin{itemize}
\item[(1)]
$m_k = m$,
\item[(2)]
$m_{i-1} \in C_i m_i$ for all $i \in I'$,
\item[(3)]
if $I$ has a minimal element, necessarily 0, then $m_0\in 1^+_{\ell,-\epsilon}(M)$ where $C$ has head $\ell$ and sign $\epsilon$,
\item[(4)]
and if $I$ has a maximal element $u$, then $m_u\in 1^+_{\ell',-\epsilon'}(M)$, where $C^{-1}$ has head $\ell'$ and sign $\epsilon'$.
\end{itemize}
By the operation of `adding to the right' we will construct new elements $m_i'\in M$ satisfying the same properties for the walk $D$. We take 
\[
m'_i = \begin{cases}
\lambda_i m_i + \mu_i m_{d-1-j} & \text{($i=d+j$ for some $d\in \Delta$ and $0\le j\le \norm{d}$)}
\\
\lambda_i m_i & \text{(otherwise)}
\end{cases}
\]
for suitable $\lambda_i,\mu_i\in K$ with $\lambda_i\neq 0$. 

Let $d\in \Delta$. Since $d$ is naturally inverse for $w$, either
\begin{itemize}
\item[(L1)]
$d-1-\norm{d}$ is a minimal element of $I$, or  
\item[(L2)]
$C_{d-1-\norm{d}}$ is an inverse letter, say $a^{-1}$, 
\end{itemize}
and either 
\begin{itemize}
\item[(R1)]
$d+\norm{d}$ is a maximal element of $I$, or 
\item[(R2)]
$C_{d+1+\norm{d}}$ is an inverse letter, say $b^{-1}$.
\end{itemize}
Moreover at least one of (L2) and (R2) must be satisfied. 

In case (L1), condition (3) gives $m_{d-1-\norm{d}} \in 1^+_{\ell,-\epsilon}(M)$ where $C_{d-\norm{d}}$ has head $\ell$ and sign $\epsilon$. This also holds in case (L2), for then $m_{d-1-\norm{d}} \in a(M)$. Now $m'_{d+\norm{d}} = \lambda_{d+\norm{d}} m_{d+\norm{d}} + \mu_{d+\norm{d}} m_{d-1-\norm{d}}$. In case (R1), condition (4) holds for the elements $(m'_i)$ since $m_{d-1-\norm{d}} \in 1^+_{\ell,-\epsilon}(M)$. In case (R2), the same property gives $b(m_{d-1-\norm{d}}) = 0$, so that $m'_{d+\norm{d}} \in C_{d+1+\norm{d}} m'_{d+1+\norm{d}}$. 

Thus, to ensure $(m_{i}')$ satisfies conditions (1)--(4) for the walk $D$, we need to solve the following.

\begin{itemize}
\item[(a)]
$\lambda_k=1$.
\item[(b)]
For any $i\in I'$ not of the form $d+j$ with $d\in \Delta$ and $0\le j\le\norm{d}$
we have $m_{i-1} \in C_i m_i$ and want $m_{i-1}' \in C_i m'_i$, so want $\lambda_{i-1} = \sigma_{C_i}(\lambda_i)$.

\item[(c)]
For any $d\in \Delta$ we have $m_{d-1} \in C_d m_d$, and $C_d=s^*$ for some special loop $s$. Using the non-singular quadratic $q_s(x)$, given $\lambda_d$ we can find $\lambda_{d-1},\mu_d$, or given $\lambda_{d-1}$ we can find $\lambda_d,\mu_d$, such that $\lambda_{d-1} m_{d-1} \in C_d^{-1} (\lambda_d m_d + \mu_d m_{d-1})$, or equivalently $m'_{d-1} \in D_d m'_d$. 
See Lemma \ref{lemma:oneqrelation-1}. 

\item[(d)]
For any $i\in I'$ of the form $d+j$ where $d\in \Delta$ and $1\le j\le\norm{d}$
we have $m_{i-1} \in C_i m_i$ and $m_{d-1-j} \in C_{d-j} m_{d-j}$ with $C_{d-j} = C_i^{-1}$, so $m_{d-j} \in C_i m_{d-1-j}$. We want $m_{i-1}' \in C_i m'_i$ so want $\lambda_{i-1} = \sigma_{C_i}(\lambda_i)$ and $\mu_{i-1} = \sigma_{C_i}(\mu_i)$.
\end{itemize}

This is straightforward, working away from $k$, both increasing and decreasing, to find the $\lambda_i$, and then working away from each element of $\Delta$ to find the $\mu_i$. Thus we can construct the tuple $(m'_i)$ and hence $m\in T_{D,k}(M)$. Thus $T_{C,k}(M)\subseteq T_{D,k}(M)$.

Now recall that $B_{C,k}(M) = B_{C,k}^+(M) + B_{C,k}^-(M)$, where
\[
B_{C,k}^+(M) = ((C_{\le k})^{-1})^-(M) \cap C_{>k}^+(M)
\quad\text{and}\quad
B_{C,k}^-(M) = ((C_{\le k})^{-1})^+(M) \cap C_{>k}^-(M).
\]
The elements $m\in B_{C,k}^+(M)$ are those for which there is a tuple $(m_i)$ as above, with the boundary condition (3) replaced by
($3'$) if $I$ has a minimal element, necessarily 0, then $m_0\in 1^-_{\ell,-\epsilon}(M)$ where $C$ has head $\ell$ and sign $\epsilon$, and if $I$ has no minimal element, then $m_i=0$ for $i\ll 0$. 
The elements $m\in B_{C,k}^-(M)$ are those for which there is a tuple $(m_i)$ as above, with the boundary condition (4) replaced by
($4'$) if $I$ has a maximal element $u$, then $m_u\in 1^-_{\ell',-\epsilon'}(M)$, where $C^{-1}$ has head $\ell'$ and sign $\epsilon'$, and if $I$ has no maximal element, then $m_i=0$ for $i\gg 0$. 

In both cases, the tuple $(m'_i)$ which is constructed preserves the new boundary condition, so $B_{C,k}(M) \subseteq B_{D,k}(M)$. 
For example, if $I$ is unbounded above and $m_{i}=0$ for all $i>t$, then by (A3) there is an upper bound on the set of $d+j$ with  $d\in \Delta$, $0\leq j\leq \norm{d}$  and $d-j\leq t$.
\end{proof}

Recall that if $w$ is a word, we write $D_w$ and $D'_w$ for corresponding special-direct and special-inverse walks. Hence, for example, this means that $D_{w}^{*}=w$ and that every $*$-letter $s^{*}$ in $w$ occurs as $s$ in $D_{w}$.

\begin{theorem}
\label{theorem:orientations-main-result}
Let $w$ be an $I$-indexed end-admissible word which is finite or periodic. Then the functors $F_{D_w}$, $F_{C_w}$ and $F_{D'_w}$ are isomorphic up to twist, that is, there are automorphisms $\sigma,\sigma'$ of $K$ such that ${}_\sigma F_{D_w} \cong F_{C_w} \cong {}_{\sigma'} F_{D'_w}$. Moreover there exists $j\in I$ such that either $D_w \sim_j C_w$ or $D'_w \sim_j C_w$.
\end{theorem}

\begin{proof}
We shall consider walks $C$ satisfying
\begin{itemize}
\item[($\ast$)]
$C^*=w$ and for any $i\in I'$ with $w_i$ a $*$-letter, if $i$ is not naturally inverse (so is naturally direct or a symmetry), then $C_i$ is a direct letter.
\end{itemize}
(1) Suppose that $w$ is a finite word of length $n$. By induction on the norm, using Lemma~\ref{lemma:orientation_change_norm_induction_step}, and using  Lemma~\ref{lemma:swapset} with $\Delta$ a singleton set, we see that $C \sim_0 D$ for $C$ and $D$ satisfying ($\ast$).

(1)(i) Suppose $w$ has no symmetry. Then $D_w$ and $C_w$ satisfy ($\ast$), and we deduce that $D_w \sim_0 C_w$.

We apply this to $w^{-1}$ to obtain $(D'_w)^{-1} = D_{w^{-1}} \sim_0 C_{w^{-1}} = (C_w)^{-1}$, and hence $D'_w \sim_n C_w$.

(1)(ii) Suppose $w$ has a symmetry, so $w = w^{-1}$, so it is of the form $u s^* u^{-1}$ for some word $u$. Now $(C_w)^{-1}$ and $D_w$ satisfy ($\ast$), so $D_w \sim_0 (C_w)^{-1}$. Suppose that $C_w$ has head $\ell$ and sign $\epsilon$.  
By Corollary \ref{corollary:twisting-(and-inverting-non-singular)-preserves-normal/central-quadratics} and Lemma \ref{inverting-bound-semilinear-relations} we can consider $C_{w}$ as a $p(y)$-bound $\tau$-semilinear relation for some monic, normal and non-singular quadratic $p(y)\in K[y;\tau]$. 
Letting $U = 1_{\ell,-\epsilon}^-(M)$ and $W = 1_{\ell,-\epsilon}^+(M)$, by Lemma~\ref{lemma:oneqrelation} we have
\begin{align*}
T_{C_w,0}(M) &= W \cap C_w W = W \cap (C_w)^{-1} W = T_{(C_w)^{-1},0}(M), \quad\text{and}
\\
B_{C_w,0}(M) &= (W \cap C_w U) + (U \cap C_w W) = (W \cap (C_w)^{-1} U) + (U \cap (C_w)^{-1} W) =
B_{(C_w)^{-1},0}(M).
\end{align*}
Thus $C_w \sim_0 (C_w)^{-1}$, and hence $D_w \sim_0 C_w$.

Now apply the induction above to $w^{-1}$. The walks $(D'_w)^{-1} = D_{w^{-1}}$ and $(C_w)^{-1}$ satisfy ($\ast$), so $(D'_w)^{-1} \sim_0 (C_w)^{-1}$. Thus $D'_w \sim_n C_w$.

(2) Suppose that $w$ is periodic. It follows that there is a bound on the norms of the $i\in I'$ which are naturally inverse. Consider walks $C$ satisfying ($\ast$). We show by induction on $k$ that, if $C$ and $D$ are walks satisfying ($\ast$), and if $C_{i}= D_{i}$ for all $i\in I'$ with $\norm{i}>k$, then $F_C$ and $F_D$ are isomorphic up to twist. 
By the inductive hypothesis, we may assume that $C_i = D_i$ for $i$ naturally inverse with $\norm{i}<k$. Now the set $\Delta = \{ i \mid C_i\neq D_i\}$ may fail property (A2) in Lemma~\ref{lemma:swapset}, but if so, we can write it as a finite disjoint union of subsets $\Delta_1\cup\dots\Delta_m$ which satisfy this property, and by induction it suffices to show the result $\Delta$ being one of these subsets. Thus we may assume that $\Delta$ satisfies (A2).

By the inductive hypothesis and Lemma~\ref{lemma:orientation_change_norm_induction_step}, we may suppose that (A3) holds. We now choose $k$ arbitrary so that (A4) holds. Then Lemma~\ref{lemma:swapset} implies that $F_C$ and $F_D$ are isomorphic up to twist.

(2)(i) Suppose that $w$ has no symmetries. Then $C_w$ and $D_w$ are of the specified form, so $F_{C_w}$ and $F_{D_w}$ are isomorphic up to twist. 

We apply this to $w^{-1}$. The walks $(D'_w)^{-1} = D_{w^{-1}}$ and $C_{w^{-1}} = (C_w)^{-1}$ satisfy ($\ast$), so give functors which are isomorphic up to twist, and hence so are $F_{D'_w}$ and $F_{C_w}$.

Now consider the words $w_{>i},(w_{\le i})^{-1}$ for $i\in I$. Since $w$ is periodic, only finitely many such words appear. 

If there is $j$ such that $w_{>j}$ is minimal amongst these words, then in any application of Lemma~\ref{lemma:swapset} in the induction above, the choice $k=j$ satisfies property (A4). Namely, suppose that $i\in\Delta$, and $i \le j \le i + \norm{i}$. Then $w_i$ is a $*$-letter, say $s^*$ and $w_{>i} = u w_{>j}$ and $(w_{\le i-1})^{-1} = u (w_{\le 2i-j-1})^{-1}$ for some word $u$ of length $j-i$. Now since $i$ is naturally inverse, $(w_{\le i-1})^{-1} < w_{>i}$, so $(w_{\le 2i-j-1})^{-1} < w_{>j}$, contradicting minimality. Thus the induction shows that $C\sim_j D$ for any walks $C,D$ satisfying ($\ast$). Thus $D_w \sim_j C_w$.

If there is no such $j$ as above, then there is $j$ such that $(w_{\le j})^{-1}$ is minimal. By considering $w^{-1}$, we deduce that $D'_w \sim_j C_w$.

(2)(ii) Now suppose that $w$ has symmetries. Recall Lemma \ref{lemma:form-of-symmetric-bands}. Since $w$ is periodic, it is of the form 
\[
\dots v s^* v^{-1} u^{-1} t^* u|v s^* v^{-1} u^{-1} t^* u v s^* v^{-1} u^{-1} t^* \dots
\]
Let $u$ have length $p$ and $v$ have length $r$. Then $w$ is periodic with period $2p+2r+2$. Using the periodicity and the symmetry, any word of the form $w_{>i}$ or $(w_{\le i})^{-1}$ with $i\in I$ is equal to one of this form with $-p\le i \le r$. Thus there is $-p\le j\le r$ such that either $w_{>j}$ or $(w_{\le j})^{-1}$ is minimal amongst this set of words. 
Then we can write $uv = \overline{u}\, \overline{v}$ with $\overline u$ of length $p+j$ and $\overline v$ of length $r-j$, and $w[j]$ is of the form
\[
\dots \overline v s^* \overline v^{-1} \overline u^{-1} t^* \overline u|\overline v s^* \overline v^{-1} \overline u^{-1} t^* \overline u \overline v s^* \overline v^{-1} \overline u^{-1} t^* \dots
\]
Now $(D_w)[j] = D_{w[j]}$,  and using the induction we obtain that $F_{(D_w)[j]}$ and $F_C$ are isomorphic up to twist, where $C$ is naturally oriented and of the form
\[
\dots E s E^{-1} D^{-1} t D|E s E^{-1} D^{-1} t D E s E^{-1} D^{-1} t \dots
\]
where $D^*=\overline u$ and $E^*=\overline v$. 
Writing $X = E s^{-1} E^{-1}$ and $Y = D^{-1} t^{-1} D$ we have
\[
T_{C,0} = (X^{-1}Y^{-1})'' \cap (Y X)'',
\quad
B_{C,0} = ((X^{-1}Y^{-1})' \cap (Y X)'') + ((X^{-1}Y^{-1})'' \cap (Y X)').
\]
Now since $-p\le j\le r$ we have $(C_w)[j] = C_{w[j]}$, and this is of the form
\[
\dots E s E^{-1} D^{-1} t D|E s^{-1} E^{-1} D^{-1} t^{-1} D E s^{-1} E^{-1} D^{-1} t^{-1} \dots
\]
so
\[
T_{(C_w)[j],0} = (X Y)'' \cap (Y X)'',
\quad
B_{(C_w)[j],0} = ((X Y)' \cap (Y X)'') + ((X Y)'' \cap (Y X)').
\]
Now we have $C \sim_0 (C_w)[j]$ by Lemma~\ref{lemma:rewriting-relations-symmetric-bands}, giving that $F_C$ and $F_{(C_w)[j]}$ are isomorphic up to twist. 
Thus $F_{(D_w)[j]}$ and $F_{(C_w)[j]}$ are isomorphic up to twist, hence so are $F_{D_w}$ and $F_{C_w}$.

Also $(D'_w)[j] = D'_{w[j]}$, and by considering $w^{-1}$, the induction shows that $F_{(D'_w)[j]}$ and $F_{C'}$ are isomorphic up to twist, where $C'$ is naturally oriented and of the form
\[
\dots E s^{-1} E^{-1} D^{-1} t^{-1} D|E s^{-1} E^{-1} D^{-1} t^{-1} D E s^{-1} E^{-1} D^{-1} t^{-1} \dots
\]
Then
\[
T_{C',0} = (X Y)'' \cap (Y^{-1} X^{-1})'',
\quad
B_{C',0} = ((X Y)' \cap (Y^{-1} X^{-1})'') + ((X Y)'' \cap (Y^{-1} X^{-1})').
\]
Now we have $C' \sim_0 (C_w)[j]$ by Lemma~\ref{lemma:rewriting-relations-symmetric-bands}, giving that $F_{C'}$ and $F_{(C_w)[j]}$ are isomorphic up to twist. Thus $F_{(D'_w)[j]}$ and $F_{(C_w)[j]}$ are isomorphic up to twist, hence so are $F_{D'_w}$ and $F_{C_w}$.

Next we want to show that $D_w \sim_j C_w$ or $D'_w \sim_j C_w$. Recall that either $w_{>j}$ or $(w_{\le j})^{-1}$ is minimal amongst words of the form $w_{>i}$ or $(w_{\le i})^{-1}$ with $i\in I$.


Suppose that $w_{>j}$ is minimal. Then in any application of Lemma~\ref{lemma:swapset} in the induction above, the choice $k=j$ satisfies property (A4). Thus the induction shows that $C\sim_j D$ for any walks $C,D$ satisfying ($\ast$). Thus $(D_w)[j] \sim_0 C$ where $C$ is as above.
Then, by the argument using Lemma~\ref{lemma:rewriting-relations-symmetric-bands}, we have $C\sim_0 (C_w)[j]$, so $(D_w)[j] \sim_0 (C_w)[j]$. Thus $D_w\sim_j C_w$. 

Suppose on the other hand that $(w_{\le j})^{-1}$ is minimal. By considering $w^{-1}$, the induction shows that $(D'_w)[j] \sim_0 C'$ where $C'$ is as above,
and another application of Lemma~\ref{lemma:rewriting-relations-symmetric-bands} gives $C'\sim_0 (C_w)[j]$. Thus $(D'_w)[j] \sim_0 (C_w)[j]$, and hence $D'_w \sim_j C_w$.
\end{proof}

\subsection{Splitting for asymmetric bands}
In previous work, various splittings were developed for relations or pairs of relations; see for example the second Corollary in \cite[p. 21]{Rin1975}. See also \cite[p. 391]{Cra1988ff}, \cite[Lemma 4.7]{Cra1989} and \cite[Lemma 4.6]{Cra2018}. In this paper we formulate things differently, and prove splittings directly for top and bottom functors. That is, if $w$ is a string or band, $C$ is the canonically associated walk, and $M$ is an $R$-module, we prove that $B_C(M)$ is a direct summand of $T_C(M)=\Hom_{R}(M(C),M)$ as an $R_w$-module (recall $B_{C}(M)$ is an $R_{w}$-submodule of $T_{C}(M)$ by Lemma \ref{lemma:B_C-is-an-R_w-submodule}). 
If $w$ is a string then $R_w$ is semisimple artinian, so this is immediate. 
For $w$ a band, see Lemmas \ref{splitting-lemma-asymmetric-band} and \ref{lemma:splitting-symmetric-bands}. 

The assumptions made in Lemmas  \ref{splitting-lemma-asymmetric-band} and Lemma \ref{lemma:splitting-symmetric-bands} ensure $F_{C}(M)$ is finite-dimensional over $K$. 
For the proof we lift a $K$-basis of $F_{C}(M)$ to $T_{C}(M)$ so that the $K$-span of these lifts is closed under the action of $R_{w}$. 
To do so, in Remark \ref{remark:for-splitting-suffices-to-find-nice-lifts} we discuss how finite-dimensional $K[x;\sigma]$-modules correspond to matrices.

Let $m,n\in\N$. For $mn\neq0$ write $M_{m,n}(K)$ for the set of matrices $\Omega$ with $m$ rows and $n$ columns, whose entry $\omega_{ij}$ in row $i$ and column $j$ lies in $K$. 
Write $O$ for the zero matrix. When $m=n$ let $M_{m,n}(K)=M_{n}(K)$ and write  $I=(\delta_{ij})$ for the identity matrix. 
In case $mn=0$ we write $M_{m,n}(K)=\{(-)_{m,n}\}$ with the appropriate conventions to
extend matrix multiplication $M_{l,m}(K)\times M_{m,n}(K)\to M_{l,n}(K)$.

For any automorphism $\sigma$ of $K$ we use the same symbol to denote the bijection on $M_{m,n}(K)$ defined by $\sigma(\Omega)=(\sigma(\omega_{ij}))$. 
When $m=n>0$ this is an automorphism of the ring $M_{n}(K)$.

\begin{remark}

\label{remark:for-splitting-suffices-to-find-nice-lifts}

Let $\sigma$ be an automorphism of $K$. 
Consider the category whose objects are matrices $\Lambda\in M_{n}(K)$ ($n\in \N$), where a morphism from $\Lambda$ to $\Gamma\in M_{m}(K)$ is given by $\Phi\in M_{n,m}(K)$ with $\Lambda\Phi=\sigma(\Phi)\Gamma$, and where the composition of $\Phi$ and $\Psi:\Gamma\to \Omega$ is defined by $\Phi\Psi$.  
%
%
%
To each object $\Lambda\in M_{n}(K)$ one can make the set $K^{n}=M_{1,n}(K)$ into a $K[x;\sigma]$-module using the formula $x\Omega =\sigma(\Omega)\Lambda$ ($\Omega\in K^{n}$). 
Hence this category is equivalent to the category of finite-dimensional $K[x;\sigma]$-modules.

\begin{enumerate}
    \item Let $w$ be an asymmetric band.  Then $R_{w}$ has the form $K[x,x^{-1};\sigma]$. 
    Let $V$ be a finite-dimensional $R_{w}$-module. Suppose, under the above equivalence, that $V$ corresponds to $\Lambda$ considered as a $K[x;\sigma]$-module, and to $\Phi$ as a $K[x^{-1};\sigma^{-1}]$-module. 
    Then $\Lambda$ is invertible with inverse $\sigma(\Phi)$. 
%
    \item Let $w$ be a symmetric band. 
    So, $R_{w}$ is the free product over $K$ of two semisimple $K$-rings, each of the form $S=K[x;\sigma]/\ideal{q(x)}$ for $q(x)=x^{2}-\beta x+\gamma$ normal and non-singular. 
    Since $q(x)$ is normal, finite-dimensional $S$-modules correspond to matrices $\Lambda$ with  $\sigma(\Lambda)\Lambda-\beta\Lambda+\gamma I=O$; see Lemma \ref{lemma:characterising_central_skew_quadratics}.  
    Since $q(x)$ is non-singular, $\Lambda$ has inverse $\gamma^{-1}(\beta I-\sigma(\Lambda))$ in $M_{n}(K)$; see Lemma \ref{lemma:invertible-variable-in-quotient-by-non-singular-quadratic}. 
    %
    %
  %
%
%
    %

Since $S$ is semisimple, by Lemma \ref{lemma:characterisingsemisimple-quadratics}  one of (1), (2) or (3) below holds.
    \begin{enumerate}
        \item[(1)] Assume $q(x)$ is  irreducible. Then $S$ is a division ring, and as the unique simple considered with basis given by $1$ and $x$, it corresponds to 
        %
        %
        \[
        \begin{pmatrix}
        0 & 1 \\
        -\gamma & \beta
        \end{pmatrix}
        \in M_{2}(K).
        \]
        \end{enumerate}
        Suppose instead $q(x)=(x-\eta)(x-\mu)$ for $\eta,\mu\in K$. Hence the factors $(x-\eta)$ and $(x-\mu)$ of $q(x)$ each define a $K$-basis for left ideals $I$ and $J$ of $S$ (each of dimension $1$ over $K$). 
        \begin{enumerate}
        \item[(2)] Assume $\sigma(\lambda)\eta\neq \eta\lambda$ for some $\lambda\in K$, which gives $S=J\oplus J\lambda^{-1}$. Hence any simple $S$-module is isomorphic to $J$, which corresponds to the matrix $\begin{pmatrix}
        \eta
        \end{pmatrix}\in M_{1}(K)$.
        %
        %
        %
        \item[(3)] Assume $\eta\neq \mu$ and $\sigma(\lambda)\eta=\eta\lambda$ for all $\lambda\in K$, and so $S=I\oplus J$. 
        Also, $q(x)=(x-\sigma^{2}(\mu))(x-\eta)$, so the matrices corresponding to $I$ and $J$ are given by $\begin{pmatrix}
     \sigma^{2}(\mu)
        \end{pmatrix}$ and $\begin{pmatrix}
        \eta
        \end{pmatrix}$ respectively.
        %
        %
    \end{enumerate}
\end{enumerate}
\end{remark}
In what follows we consider infinite sequences of matrices, say $\Omega_{d}$ ($d\in\Z$), whose entries are denoted~${}^{d}\omega_{ij}$.

\begin{lemma}\label{splitting-lemma-asymmetric-band}
Let $w$ be an asymmetric band with canonically associated walk $C$. 
If $T_C(M)/B_C(M)$ is finite dimensional over $K$, then $B_C(M)$ has a complement in $T_C(M)$ as an $R_{w}=K[x,x^{-1};\sigma]$-submodule.
\end{lemma}

\begin{proof}
Let $(h_{1},\dots,h_{n})$ be a $K$-basis for $F_{C}(M)=T_{C}(M)/B_{C}
(M)$. 
For each $i=1,\dots,n$ write $xh_{i}=\sum_{j}\lambda_{ij}h_{j}$ and $x^{-1}h_{i}=\sum_{j}\varphi_{ij}h_{j}$ for some $\lambda_{ij},\varphi_{ij}\in K$. 
By Remark \ref{remark:for-splitting-suffices-to-find-nice-lifts}(i) the matrices $\Lambda=(\lambda_{ij})$ and $\Phi=(\varphi_{ij})$  satisfy $\Lambda^{-1}=\sigma(\Phi)$ in $M_{n}(K)$. 
For each $d\in \Z$ define matrices  $\Omega_{d}=({}^{d}\omega_{ij})$ and $\Psi_{d}=({}^{d}\psi_{ij})$  by
\[
\begin{array}{ccc}
\Omega_{d}=
\begin{cases}
 O & (d<0),\\
 \Lambda^{-1} & (d=0),\\
\Lambda^{-1}\sigma(\Omega_{d-1}) & (d>0),
\end{cases}  & \text{and}  &  
\Psi_{d}=
\begin{cases}
\sigma^{-1}(\Lambda\Psi_{d+1}) & (d<-1),\\
-I & (d=-1),\\
O & (d>-1).
\end{cases}
\end{array}
\]
%
Now lift each $h_{i}$ to $f_{i}\in T_{C}(M)$ and write $xf_{i}=g_{i}^{-}+g_{i}^{+}+\sum\lambda_{ij}f_{j}$ for $g_{i}^{\pm}\in B^{\pm}_{C}(M)$. 
For each $i$, $j$ and $d$ let ${}^{d}z_{ij}={}^{d}\omega_{ij}x^{d}g_{j}^{-} + {}^{d}\psi_{ij}x^{d}g_{j}^{+}$. 
For any $r\in\Z$ we have $x^{d}g_{j}^{-}(b_{r})=g_{j}^{-}(b_{r-dp})$, which is $0$ for $d\gg 0$. Similarly $x^{d}g_{j}^{+}=0$ for $d\ll 0$, and so for each $i$ we have ${}^{d}z_{ij}=0$ for all but finitely many $j$ and $d$, so $\sum_{j,d} {}^{d}z_{ij}$ defines an element of $B_{C}(M)$.
For each $i$ let $f_{i}'=f_{i}+\sum_{j,d} {}^{d}z_{ij}$. 
This gives 
\[
\begin{split}
    xf'_{i}-\sum_{j}\lambda_{ij}f'_{j}
    =g_{i}^{-}+g_{i}^{+}+\sum_{j,d}\left(\sigma({}^{d-1}\omega_{ij})-\sum_{k}\lambda_{ik}{}^{d}\omega_{kj}\right)x^{d}g_{j}^{-}+\left(\sigma({}^{d-1}\psi_{ij})-\sum_{k}\lambda_{ik}{}^{d}\psi_{kj}\right)x^{d}g_{j}^{+},
\end{split}
\]
and so $xf'_{j}=\sum_{i}\lambda_{ji}f'_{i}$ for all $j$. 
\end{proof}

Lemma \ref{lemma:splitting-symmetric-bands} is analogous to Lemma \ref{splitting-lemma-asymmetric-band}, but we take $w$ to be a symmetric band, and for our proof we require the $R$-module $M$ to be finite-dimensional.  

\subsection{Preliminaries for symmetric splitting}\label{s:Solutions-to-simultaneous-equations}

For Lemma \ref{lemma:symmetric-splitting-sequence-of-matrices} we fix solutions in $M_{n}(K)$ to a pair of equations, each given by a monic, normal, non-singular and semisimple quadratic $q(x)\in K[x;\sigma]$.
We then define a sequence of matrices satisfying equations $\eqref{eq1}$--$\eqref{eq5}$. 
Lemma \ref{lemma:technical-symmetric-band-splitting-matrices} deals with \eqref{eq1}. 

\begin{lemma}

\label{lemma:technical-symmetric-band-splitting-matrices}
Let $q(x)=x^{2}-\beta x+\gamma$ be normal, non-singular and semisimple in $K[x;\sigma]$. 
If $\Lambda\in M_{n}(K)$ and $\sigma(\Lambda)\Lambda-\beta\Lambda+\gamma I=O$ then there is a matrix $\Xi\in M_{n}(K)$ such that $\Lambda\Xi+\sigma(\Xi)(\sigma(\Lambda)-\beta I)=I$.
\end{lemma}

\begin{proof}
Write $\mathrm{char}(K)$ for the characteristic of $K$. Since $q(x)$ is normal, by Lemma \ref{lemma:characterising_central_skew_quadratics} we have that $\beta\lambda=\sigma(\lambda)\beta$ and $\gamma\lambda=\sigma^{2}(\lambda)\gamma$ for all $\lambda\in K$, and that $\beta$ and $\gamma$ are fixed by $\sigma$.

(i) Suppose either $\mathrm{char}(K)\neq 2$ or $\beta\neq0$. 
We firstly assert that there exists $\nu,\zeta\in K$ such that $\sigma(\nu)=\nu$, $\sigma(\zeta)=\zeta$, $\zeta\beta-2\nu\gamma=1$, $\nu\beta=2\zeta$ and $\nu\sigma^{2}(\lambda)=\lambda\nu$ and $\zeta\sigma(\lambda)=\lambda\zeta$ for all $\lambda\in K$. 
If $\mathrm{char}(K)=2$ then $\beta^{-1}\sigma(\lambda)=\lambda\beta^{-1}$ and $\sigma(\beta^{-1})=\beta^{-1}$, and it suffices let $\zeta=\beta^{-1}$ and $\nu=0$. 

Suppose instead $\mathrm{char}(K)\neq2$. 
We claim that $\beta^{2}\neq 4\gamma$. 
Otherwise setting $\eta=\frac{1}{2}\beta$ gives $\gamma=\eta^{2}$ which means $q(x)=(x-\eta)^{2}$ where $\eta\lambda=\sigma(\lambda)\eta$ for all $\lambda\in K$. 
By Lemma \ref{lemma:characterisingsemisimple-quadratics} this contradicts our assumption that the quotient ring $S=K[x;\sigma]/\ideal{q(x)}$ is semisimple. 
Let $\alpha=\beta^{2}-4\gamma$, and so $\alpha\neq 0$. 
Note that $\beta^{2}\lambda=\sigma^{2}(\lambda)\beta^{2}$ and $(4\gamma)\lambda=4\sigma^{2}(\lambda)\gamma$, and so altogether $\sigma(\alpha)=\alpha$ and  $\alpha\lambda=\sigma^{2}(\lambda)\alpha$ for all $\lambda\in K$. 

It straightforward to check $\zeta=\beta\alpha^{-1}$ and $\nu=2\alpha^{-1}$ satisfy the asserted properties, 
%
%
%
%
%
%
and likewise straightforward to check that $\Xi=\nu\sigma(\Lambda)-\zeta I$ satisfies $\Lambda\Xi+\sigma(\Xi)(\sigma(\Lambda)-\beta I)=I$. 

(ii) Suppose instead $\mathrm{char}(K)= 2$ and  $\beta=0$, meaning $q(x)=x^{2}+\gamma$.  
Since $S=K[x;\sigma]/\ideal{q(x)}$ is semisimple, according to Lemma \ref{lemma:characterisingsemisimple-quadratics} there are three possibilities for $q(x)$ corresponding to parts (1), (2) and (3) of Remark \ref{remark:for-splitting-suffices-to-find-nice-lifts}(ii).  
For (2) and (3) we have $q(x)=(x-\eta)(x-\mu)$ for distinct $\mu,\eta\in K$,  which gives $\eta=\sigma(\mu)$ since $\mathrm{char}(K)= 2$ and  $\beta=0$. 
Since $\gamma\neq 0$ we also have $\mu\neq 0$. 
Hence case (3) cannot occur, for otherwise $\sigma(\lambda)\eta=\eta\lambda$ for all $\lambda \in K$, and in particular $\sigma(\mu)(\eta-\mu)=0$, a contradiction.

Hence either (1) or (2) holds. 

Recall the category of matrices considered in Remark \ref{remark:for-splitting-suffices-to-find-nice-lifts}. 
Since $\sigma(\Lambda)\Lambda+\gamma I=0$ we can consider $\Lambda$ as an object in a category equivalent to the category of finite-dimensional $S$-modules. 
Since either (1) or (2) holds, any such module is isomorphic to a finite direct sum of copies of the unique simple.

Given $\Phi,\Gamma,\Theta\in M_{n}(K)$ with $\Phi$ invertible such that  $\Lambda=\sigma(\Phi)\Gamma\Phi^{-1}$ and $\Gamma\Theta+\sigma(\Theta)\sigma(\Gamma)=I$, setting $\Xi=\Phi\Theta\sigma(\Phi^{-1})$ gives $\Lambda\Xi+\sigma(\Xi)\sigma(\Lambda)=I$. 
%
%
Hence it suffices to consider $\Lambda$ up to isomorphism. 
In this category of matrices, note that the direct sum of $\Lambda\in M_{n}(K)$ and $\Gamma\in M_{m}(K)$ is given by forming a block diagonal matrix $\Omega=\Lambda\oplus\Gamma\in M_{n+m}(K)$.  
Hence if $\Lambda\Xi+\sigma(\Xi)\sigma(\Lambda)=I$ in $M_{n}(K)$ and $\Gamma\Theta+\sigma(\Theta)\sigma(\Gamma)=I$ in $M_{m}(K)$ then $\Omega(\Xi\oplus\Theta)+\sigma(\Xi\oplus\Theta)\sigma(\Omega)=I$ in $M_{n+m}(K)$. 
Hence we can reduce to the case where $\Lambda$ is a simple object in this category of matrices, of which there is only one, up to isomorphism. 

In case (1), the proof of the lemma now follows from the equation
\[
\begin{pmatrix}
        0 & 1 \\
        -\gamma & 0
        \end{pmatrix}\begin{pmatrix}
        0 & 0 \\
        1 & 0
        \end{pmatrix}+\sigma\left(\begin{pmatrix}
        0 & 0 \\
        1 & 0
        \end{pmatrix}\begin{pmatrix}
        0 & 1 \\
        -\gamma & 0
        \end{pmatrix}\right)=\begin{pmatrix}
        1 & 0 \\
        0 & 1
        \end{pmatrix}.
\]
In case (2), $\sigma(\lambda)\eta\neq\eta\lambda$ for some $\lambda\in K$. 
Let $\alpha=\sigma(\lambda)\eta-\eta\lambda$, and since $\alpha\neq 0$ we can consider $\xi=\lambda\alpha^{-1}\neq 0$.  
Since $\sigma(\eta)\eta=\gamma$ we have $\sigma(\alpha)\eta=\sigma(\eta) \alpha$ which gives  $\sigma(\alpha^{-1}\eta)=\eta\alpha^{-1}$ and so $\eta \xi+\sigma(\xi)\sigma(\eta)=1$.
\end{proof}

\begin{lemma}\label{lemma:symmetric-splitting-sequence-of-matrices}
Let $\rho,\tau,\pi,\zeta$ be automorphisms of $K$. Let $r(x)=x^{2}-\beta x+\gamma$ and $p(y)=y^{2}-\mu y+\eta$ be normal, non-singular and semisimple in $K[x;\rho]$ and $K[y;\tau]$ respectively. 

If $\Phi,\Lambda\in M_{n}(K)$ satisfy  $\rho(\Phi)\Phi-\beta\Phi+\gamma I=O$ and $\tau(\Lambda)\Lambda-\mu\Lambda+\eta I=O$ then there is a collection $(\Theta_{k}\mid k\in\N)$ of matrices $\Theta_{k}\in M_{n}(K)$ satisfying \eqref{eq1} and \eqref{eq2}, \eqref{eq3}, \eqref{eq4} and \eqref{eq5} for each $k$.
\begin{equation}
\tau(\pi(\Theta_{0}))(\mu I - \tau(\Lambda))  = -I + \Lambda\pi(\Theta_{0}),  \label{eq1}
\end{equation}
\begin{equation}
\Lambda\pi(\Theta_{2k+2})=\mu\pi(\Theta_{2k+2})+\tau(\pi(\Theta_{2k+1})), \label{eq2}
 \end{equation}
\begin{equation}
  \Lambda\pi(\Theta_{2k+1})=-\tau(\pi(\Theta_{2k+2}))\eta, \label{eq3}
\end{equation}
\begin{equation}
\Phi \zeta(\Theta_{2k}) =-\rho(\zeta(\Theta_{2k+1}))\gamma, \label{eq4}
\end{equation}
\begin{equation}
  \Phi\zeta(\Theta_{2k+1})=\rho(\zeta(\Theta_{2k}))+\beta\zeta(\Theta_{2k+1}). \label{eq5}
\end{equation}
\end{lemma}

\begin{proof}
By Remark \ref{remark:for-splitting-suffices-to-find-nice-lifts}(ii) the matrices $\rho(\Phi)$ and $\tau(\Lambda)$ are units in the ring $M_{n}(K)$. 
Since we have $\tau(\Lambda)\Lambda-\mu\Lambda+\eta I=O$, by Lemma \ref{lemma:technical-symmetric-band-splitting-matrices}  there is a matrix $\Xi\in M_{n}(K)$ such that $\Lambda\Xi+\tau(\Xi)(\tau(\Lambda)-\mu I)=I$. 
Setting $\Theta_{0}=\pi^{-1}(\Xi)$, \eqref{eq1} is automatic. 
Now suppose for some $k$ that the matrix $\Theta_{2k}$ has been defined. 
Let $\Theta_{2k+1}=-\zeta^{-1}(\rho^{-1}(\Phi\zeta(\Theta_{2k}))\gamma^{-1})$. 
By Lemma \ref{lemma:characterising_central_skew_quadratics} we have $\rho(\gamma)=\gamma$ and so \eqref{eq4} holds. 
Likewise $\rho^{2}(\lambda)\gamma=\gamma\lambda$ for all $\lambda\in K$, and so applying $\rho$ to \eqref{eq4} gives 
\[
\rho(\Phi \zeta(\Theta_{2k})) =-\rho^{2}(\zeta(\Theta_{2k+1}))\gamma=-\gamma\zeta(\Theta_{2k+1})=\rho(\Phi)(\Phi-\beta I)\zeta(\Theta_{2k+1}).
\]
One yields
\eqref{eq5} 
by multiplying by $(\rho(\Phi))^{-1}$. 
Now let $\Theta_{2k+2}=-\pi^{-1}(\tau^{-1}(\Lambda\pi(\Theta_{2k+1}))\eta^{-1})$. 
By construction \eqref{eq3} holds, and using a similar argument to the one above, this gives \eqref{eq2}.
\end{proof}

We apply Lemma \ref{lemma:expanding-out-relations-for-symmetric-bands} in  Lemma \ref{lemma:a-pair-of-matrices-gives-a-P-submodule} to give a procedure for constructing $R_{w}$-submodules of $T_{C}(M)$ where $w$ is a symmetric band.

\begin{lemma}
\label{lemma:expanding-out-relations-for-symmetric-bands}
Let $d,n> 0$ be integers. 
Let $V_{0},\dots,V_{n-1}$ be $K$-modules, let $\Lambda=(\lambda_{ij})$ and $\Phi=(\varphi_{ij})$ be invertible matrices in $M_{d}(K)$ and let $\tau$, $\rho$ and $\sigma_{h}$ \emph{(}$1\leq h\leq n-1$\emph{)} be automorphisms of $K$. 

Suppose there exists $v^{h}_{i}\in V_{h}$ for each $h$ and each $i=1,\dots,d$ such that the following statements hold.
\begin{enumerate}
    \item[(a)] Each pair $(v_{i}^{0},\sum_{j=1}^{d} \lambda_{ij}v^{0}_{j})$ lies in some $\tau$-semilinear relation $Y:V_{0}\to V_{0}$.
    \item[(b)] Each pair $(v^{h}_{i},v^{h-1}_{i})$ with $1\leq h\leq n-1$ lies in some  $\sigma_{h}$-semilinear relation $A_{h}:V_{h}\to V_{h-1}$.
    \item[(c)] Each pair $(v_{i}^{n-1},\sum_{j=1}^{d} \varphi_{ij}v^{n-1}_{j})$ lies in some $\rho$-semilinear relation $X:V_{n-1}\to V_{n-1}$.
\end{enumerate}
Then, for each $i=1,\dots,d$ there is a sequence $(v_{i}^{l}\mid l\in\Z)$ such that (i)--(iv) below hold for any $r\in \Z$.
\begin{enumerate}
    \item If $r$ is even then $v_{i}^{rn+h}\in A_{h}v_{i}^{rn+h+1}$ for each $h$, and if also  $\pm (r-1)<0 $ then   $v_{i}^{rn-1}\in Y^{\pm 1}v_{i}^{rn}$.
    \item If $r$ is odd then $v_{i}^{rn+n-h}\in A_{h}v_{i}^{rn+n-h-1}$ for each $h$,  and if also  $\pm r<0$ then  $v_{i}^{rn-1}\in X^{\pm 1}v_{i}^{rn}$.
    \end{enumerate}
Let $\sigma_{0}=\tau$ and, for any non-zero  $l\in\Z$, let 
\[
\sigma_{l}=
\begin{cases}
\tau^{\pm 1} & (l=\mp rn\text{ where }0<r\in \Z\text{ is even}),\\
\rho^{\pm 1} & (l=\mp rn\text{ where }0<r\in \Z\text{ is odd}),\\
\sigma_{h} & (l=h+rn\text{ where }r\in \Z\text{ is even and }1\leq h\leq n-1),\\
\sigma_{n-h}^{-1} & (l=h+rn\text{ where }r\in \Z\text{ is odd and }1\leq h\leq n-1).
\end{cases}
\]
\begin{enumerate}\setcounter{enumi}{2}
    \item If $r\leq n-1$ then  $v_{i}^{2n-(r+1)}=\sum_{j=1}^{d} (\sigma_{r+1}\dots \sigma_{n-1})(\varphi_{ij})v_{j}^{r}$ \emph{(}where $\sigma_{n}\dots\sigma_{n-1}$ is the identity\emph{)}.
    \item If $r\geq 0$ then  $v_{i}^{-(r+1)}=\sum_{j=1}^{d} (\sigma_{1}\dots \sigma_{r})^{-1}(\lambda_{ij})v_{j}^{r}$ \emph{(}where $\sigma_{1}\dots\sigma_{0}$ is the identity\emph{)}.
\end{enumerate}
\end{lemma}

\begin{proof} For each $r\in\Z$ we define a matrix $\Omega_{r}=({}^{r}\omega_{ij})$ in $ M_{d}(K)$, as follows. 
For $s\geq 0$ consider the equations
\[
 \begin{array}{cc}
 (s,\flat) & \Omega_{-s}=(\sigma_{1}\dots \sigma_{sn})^{-1}(\Lambda)\Omega_{s}((\sigma_{1}\dots \sigma_{sn-1})^{-1}(\Lambda))^{-1}\\[-1em]\\
   (s,\sharp) & \Omega_{s+1}=(\sigma_{(1-s)n}\dots \sigma_{n-1})(\Phi)\Omega_{1-s}((\sigma_{(1-s)n+1}\dots \sigma_{n-1})(\Phi))^{-1}
\end{array}
\]
where we declare that ($0,\flat$) and ($0,\sharp$) say that $\Omega_{0}=\Lambda$ and $\Omega_{1}=\Phi$ respectively. 
From here one can define $\Omega_{r}$ ($r\in \Z$) iteratively. 
Let $\sigma=\sigma_{1}\dots\sigma_{n-1}$. 
We claim ($r,\natural$) below holds for any $r\in \Z$ with $r\neq 0,1$.
\[
\begin{array}{cc}
    (r,\natural) &  
 \begin{array}{cc}
  \Omega_{r\pm1}\sigma^{-1}(\Omega_{r})=\rho(\sigma^{-1}(\Omega_{r})\Omega_{r\mp1}) & (\pm r\geq 2\text{ and }r\text{ even})\\[-1em]\\
  \Omega_{r\pm 1}\sigma(\Omega_{r})=\tau(\sigma(\Omega_{r})\Omega_{r\mp 1}) & (\pm r\geq 1,r\neq 1\text{ and }r\text{ odd})
\end{array}
\end{array}
\]
We check ($r,\natural$) for $r\geq 2$ even. 
The other cases are similar. Without loss of generality assume that $n> 1$. 

For $t\geq 1$ let $\Psi_{t}=(\sigma_{-t}\dots \sigma_{n-1})(\Phi)(\sigma_{1}\dots \sigma_{t})^{-1}(\Lambda)$.  For any $s\geq 0$ we have  $\sigma_{-sn}=\sigma_{sn}^{-1}$ which is $\rho$ for $s$ odd and $\tau$ for $s$ even. In both cases $\Psi_{(s+1)n}$ is the image of $\Psi_{(s+1)n-1}$ under this  automorphism.

For $s\geq 3$ we have  $\Omega_{s}=\Psi_{(s-2)n}\Omega_{s-2}(\Psi_{(s-2)n-1})^{-1}$ by  ($s-1,\sharp$) and ($s-2,\flat$). 
For $s=3$, since $\Psi_{n-1}=\sigma^{-1}(\Omega_{2})\Phi$  by ($1,\sharp$), this gives the base case for an induction on $r$. 
Thus, assuming $r\geq 4$, taking $s=r-1,r,r+1$ and applying the inductive hypothesis on $r-2\geq 2$ (which is even) gives 
\[
\Omega_{r+1}\sigma^{-1}(\Omega_{r})=
\rho\left(\Psi_{(r-1)n-1}\sigma^{-1}(\Omega_{r-2})
\Omega_{r-3}\right)
\left(\sigma^{-1}(\Omega_{r-2})\right)^{-1}
\Psi_{(r-1)n-1}^{-1}
\sigma^{-1}\left(\Psi_{(r-2)n}\Omega_{r-2}\Psi_{(r-2)n-1}^{-1}\right).
\] 
Using that $\sigma^{-1}(\Psi_{(r-2)n})=\Psi_{(r-1)n-1}$ and that $\sigma^{-1}(\Psi_{(r-2)n-1})=\rho(\Psi_{(r-3)n-1})$, the induction follows. 

Hence, for the matrices $\Omega_{r}\in M_{d}(K)$ constructed above, we have shown the equations ($r,\natural$) hold. 
Let $v_{i}^{-1}= \sum_{j} \lambda_{ij}v^{0}_{j}$ and  $v_{i}^{n}= \sum_{j} \varphi_{ij}v^{n-1}_{j}$. 
Recall that an arrow from $v\in V$ to $w\in W$ labelled with a relation $C:V\to W$  means that $w\in Cv$. 
For example, our assumptions (a), (b) and (c) may now be summarised by 
\[
\begin{array}{cc} (abc)& 
\begin{tikzcd}
   v_{i}^{-1} &  v^{0}_{i}\arrow[swap]{l}{Y} &  v^{1}_{i}\arrow[swap]{l}{A_{1}} &  \cdots\arrow[swap]{l} & v^{n-2}_{i}\arrow[swap]{l} & v^{n-1}_{i}\arrow[swap]{l}{A_{n-1} }\arrow{r}{X} & v_{i}^{n}
\end{tikzcd}
\end{array}
\]
For each non-zero multiple $m=rn$ of $n$ ($0\neq r\in\Z$), we consider a diagram depicting relations (as in $(abc)$ above).  
These diagrams split into one of two different types, depending on the parity of $r=m/n$. 
\[
\begin{array}{cc}
    (r,1): \begin{tikzcd}[column sep=0.8cm]
   v_{i}^{m-1} &  v^{m}_{i}\arrow[swap]{l}{X^{\pm 1}}\arrow{r}{A_{n-1}} &  v^{m+1}_{i}\arrow{r} &  \cdots\arrow{r} & v^{m+n-2}_{i}\arrow{r}{A_{1} } & v^{m+n-1}_{i} & v_{i}^{m+n}\arrow[swap]{l}{Y^{\pm 1}}
\end{tikzcd} & 
(\pm r<0\text{ odd})\\
  (r,2):  \begin{tikzcd}[column sep=0.8cm]
   v_{i}^{m-1} &  v^{m}_{i}\arrow[swap]{l}{Y^{\pm 1}} &  v^{m+1}_{i}\arrow[swap]{l}{A_{1}} &  \cdots\arrow[swap]{l} & v^{m+n-2}_{i}\arrow[swap]{l} & v^{m+n-1}_{i}\arrow[swap]{l}{A_{n-1} } & v_{i}^{m+n}\arrow[swap]{l}{X^{\pm 1}}
\end{tikzcd} & 
(\pm r<0\text{ even})
\end{array}
\]
Diagrams ($abc$), ($r,1$) and ($r,2$) together summarise the claims made in parts (i) and (ii) of the lemma.  
By induction on $|r|>0$ we shall construct the elements $v_{i}^{l}$  with $m-1\leq l\leq m+n$. We assert, for any $r\in \Z$, that the elements $v_{i}^{l}$ fit into the corresponding diagram, and that the following equations hold.
\[
 \begin{array}{cccc}
 (r,\dagger) & v_{i}^{m-1}=\sum_{j}({}^{r}\omega_{ij}v_{j}^{m}), & v_{i}^{m+n-1}=\sum_{j}({}^{r+1}\omega_{ij}v_{j}^{m+n}) &  (r<0)\\[-1em]\\
  (r,\ddagger) &
    v_{i}^{m}=\sum_{j}({}^{r}\omega_{ij}v_{j}^{m-1}), & v_{i}^{m+n}=\sum_{j}({}^{r+1}\omega_{ij}v_{j}^{m+n-1}) & (r>0)
\end{array}
\]
Once we prove our assertion, we shall prove (iii) and (iv). 
For the constructions we use induction on $|r|>0$. 
For the base case of the induction we must construct the appropriate diagrams to address the case $r=\pm 1$, so $m=\pm n$. 
Starting with the case $m=-n$, we let 
\[
v_{i}^{-h}=
\begin{cases}
\sum_{j}( (\sigma_{1}\dots\sigma_{h-1})^{-1}({}^{0}\omega_{ij})v^{h-1}_{j}), &  (1<h\leq n),\\
\sum_{j}({}^{-1}\omega_{ij}v_{j}^{-n}), & (h=n+1).
\end{cases}
\]
The elements $v_{i}^{-2},\dots,v_{i}^{-n-1}$ defined above satisfy the relations depicted in diagram ($r,1$) when $r=-1$. 
This follows by semilinearity and ($1,\flat$). 
Note also that ($-1,\dagger$) is automatic.

This completes the case $m=-n$. 
For $r=1$, so $m=n$, let 
\[
 v_{i}^{n+h-1}=
\begin{cases}
\sum_{j}( (\sigma_{n-h+1}\dots \sigma_{n-1})({}^{1}\omega_{ij})v^{n-h}_{j}), &  (1<h\leq n),\\
=\sum _{j}({}^{2}\omega_{ij}v_{j}^{2n-1}), & (h=n+1).
\end{cases}
\]
Similarly to the above, our assertion follows from semilinearity and ($1,\sharp$). 

This completes the base case. 
Now assume we have constructed $v_{i}^{l}$  fitting in diagrams ($abc$),  ($r,1$) and ($r,2$), and satisfying equations ($r,\dagger$) and ($r,\ddagger$), say for all $r$ with $-s\leq r\leq s$ for some $s>0$. 
We complete the inductive step firstly assuming $s$ is odd, and then assuming $s$ is even. 

Let $s>0$ be odd and $m=(s+1)n$. 
We already have the $v_{i}^{l}$ with $m-n-1\leq l\leq m$. Now let
\[
v_{i}^{m+h-1}=
\begin{cases}
\sum_{j}(\sigma_{1}\dots \sigma_{h-1})^{-1}({}^{s+1}\omega_{ij})v_{j}^{m-h}), & (1< h\leq n)\\
\sum _{j}({}^{s+2}\omega_{ij}v_{j}^{m+n-1}), & (h=n+1).
\end{cases}
\]
By induction the $v_{i}^{l}$ with $m-n-1\leq l\leq m$ fit into ($s,1$) and satisfy  ($s,\ddagger$). As in the base case $r=-1$, but instead using ($s+1,\natural$), it is straightforward to check that the elements $v_{i}^{l}$ with $l$ fit into ($s+1,2$) and satisfy ($s+1,\ddagger$). 
The case where $m=(s+1)n$ for $s>0$ even is similar, but where $Y$ is swapped with $X$, and likewise $\tau$ with $\rho$ and $(\sigma_{1}\dots \sigma_{h-1})^{-1}$ with  $\sigma_{n-h+1}\dots \sigma_{n-1}$ for each $h\leq n$.

This concludes the inductive step in case $s>0$. 
Suppose instead $m=-(s+1)n$ for $s>0$ odd. 
Then let 
\[
v_{i}^{m-h}=
\begin{cases}
\sum_{j}(\sigma_{1}\dots \sigma_{h-1})^{-1}({}^{-s-1}\omega_{ij})v_{j}^{m+h-1}), & (1< h\leq n)\\
\sum _{j}({}^{-s-2}\omega_{ij}v_{j}^{m-n}), & (h=n+1).
\end{cases}
\]
As in the case $m>0$, when $m=-(s+1)n$ for $r=s>0$ even, the proof is symmetric.
%
The construction of the sequences of elements $v_{i}^{l}$ is now complete, as are the proofs of (i) and (ii).  

Define $\Phi_{t}=({}^{t}\varphi_{ij})$ by  $\Phi_{t}=(\sigma_{t}\dots \sigma_{n-1})(\Phi)$ for each  $t\leq 1$. 
We now check (iii) holds. 
The proof that (iv) holds is similar. 
We use induction on $s>0$ where $r=n-s$.  
The case $s\leq n-1$ follows by construction. 
We assume (iii) holds for all $r\geq n-s$  for some $s$. Let $r=n-(s+1)$, so $2n-(r+1)=n+s$. 

Consider firstly the case where $s=qn$ for some integer $q>0$, so that $n+s=(q+1)n$. 
We assume $q$ is even, since the case where $q$ is odd is similar. 
Now we have
\[
\begin{array}{c}
v_{i}^{n+s}=\sum_{j}({}^{q+1}\omega_{ij}v_{j}^{n+s-1})
=\sum_{k}(\sum_{j}{}^{q+1}\omega_{ij}{}^{n-s+1}\varphi_{jk})v_{k}^{n-s}
\end{array}
\]
by construction and the inductive assumption. 
Note $\Omega_{q+1}\Phi_{n-s+1}=\Phi_{n-s}\Omega_{1-q}$ by ($q,\sharp$). 
Using the formula for $v_{k}^{n-s-1}$, this completes the inductive step in this first case, where $s\in n\Z$. 

The second case is when $r=n-(s+1)$ where instead $s=qn+p$ for some integers $q,p$ with $q\geq 0$ and $1\leq p\leq n-1$. 
Without loss of generality we assume $q$ is odd. 
By ($q,\sharp$) we have  
$(\sigma_{1}\dots\sigma_{p})^{-1}(\Omega_{q+1})\Phi_{(1-q)n+p+1}=\Phi_{(1-q)n-p}(\sigma_{1}\dots\sigma_{p})^{-1}(\Omega_{1-q})$. 
Combining everything so far (including the definition of $v_{j}^{k}$ above) with the inductive hypothesis 
completes the proof.
\end{proof}

\subsection{Splitting for symmetric bands}

The canonically associated walk of a symmetric band $w$ is
\[
C_{w}=\dots E s E^{-1} D^{-1} t D|E s^{-1} E^{-1} D^{-1} t^{-1} D E s^{-1} E^{-1} D^{-1} t^{-1} \dots
\]
given by some $s,t\in\Sp$ and some walks $D,E$ of length $p,r\geq 0$ respectively. 
Recall also the free product 
\[
R_{w} =K[x;\rho]/\ideal{r(x)} *_{K} K[y;\tau]/\ideal{p(y)},
\]
defined by $\rho=\pi_{r}^{-1} \sigma_{s} \pi_{r}$,  $r(x)=\pi_{r}^{-1}(q_{s}(x))$, $\tau=\pi_{-p}^{-1} \sigma_{t} \pi_{-p}$ and  $p(y)=\pi_{-p}^{-1}(q_{t}(y))$. See Section \ref{s:symband}. 

By Theorem \ref{theorem:symmetricbandparameterringisHNP} the ring $R_{w}$ has a left or right $K$-basis given by the alternating words in $x$ and $y$, all of which are units in $R_{w}$ by Lemma \ref{lemma:invertible-variable-in-quotient-by-non-singular-quadratic}.  
In Lemmas \ref{lemma:a-pair-of-matrices-gives-a-P-submodule}, \ref{symmetric-band-splitting-lifting-morphisms} and \ref{lemma:splitting-symmetric-bands} we assume the above notation, let $C=C_{w}$ and let $M$ be some $R$-module. 
We consider the letters of the walks $D$ and $E$ as semilinear relations for appropriate automorphisms of $K$.

\begin{lemma}\label{lemma:a-pair-of-matrices-gives-a-P-submodule}
Let $n>0$ and $w$ be a symmetric band.
In the notation above, suppose that for each $i=1,\dots,n$ there is a sequence $(m_{i}^{d}\mid d=-p,\dots,r)$ with $m_{i}^{d}\in M$ such that (a) and (b) below hold.
\begin{enumerate}
    \item[(a)] For any $i$ we have $m^{d}_{i}\in D_{p+d+1}m^{d+1}_{i}$ when $-p\leq d <0$, and $m^{d-1}_{i}\in E_{d}m^{d}_{i}$ when  $0<d\leq r$.
    \item[(b)] There exist matrices $\Lambda=(\lambda_{ij})$ and $\Phi=(\varphi_{ij})$  in $ M_{n}(K)$ such that 
    \[
    \begin{array}{ccc}
     \sigma_{t}(\Lambda)\Lambda-\beta_{t}\Lambda+\gamma_{t} I=O\text{, } tm^{-p}_{i}=\sum_{j=1}^{n}\lambda_{ij}m_{j}^{-p}    &   \text{and}     & \sigma_{s}(\Phi)\Phi-\beta_{s}\Phi+\gamma_{s} I=O\text{, } sm^{r}_{i}=\sum_{j=1}^{n}\varphi_{ij}m_{j}^{r}.
    \end{array}
    \]
\end{enumerate}
Then for each $i$ there is a sequence $(m_{i}^{d}\mid d\in\Z)$ in $ M$ such that $(m_{i}^{d},m_{i}^{d-1})\in C_{d}$ for all $d$, and such that defining $f_{i}\in T_{C}(M)$ by $f_{i}(b_{d})=m^{d}_{i}$ for all $d\in\Z$ gives a left $R_{w}$-submodule $U=\sum_{i=1}^{n} Kf_{i}$ of $T_{C}(M)$.
\end{lemma}

\begin{proof}

Assumption (a) may be summarised by the diagram 
\[
\begin{tikzcd}
 m^{-p}_{j} &  \cdots \arrow[swap]{l}{D_{1}} &  m_{j}^{0}\arrow[swap]{l}{D_{p}} &  \cdots\arrow[swap]{l}{E_{1}}  & m^{r}_{j}\arrow[swap]{l}{E_{r}}
\end{tikzcd}
\]
By assumption (b) and  Remark \ref{remark:for-splitting-suffices-to-find-nice-lifts}(ii) the matrices $\Lambda$ and $\Phi$ are invertible. 
Let $A=DE$ and $n=p+r+1$. 
Let $v^{h}_{i}=m_{i}^{h-p}$ for each $i$ and each integer $h$ with $1\leq h\leq n-1$, and let $\alpha_{l}=\sigma_{C_{l}}$ for each $l\in \Z$. 
Hence,
\[
\begin{array}{cc}
    \alpha_{-p}=\tau, & \alpha_{l}=
\begin{cases}
\tau^{\pm 1} & (l+p=\mp dn\text{ where }0<d\in \Z\text{ is even}),\\
\rho^{\pm 1} & (l+p=\mp dn\text{ where }0<d\in \Z\text{ is odd}),\\
\sigma_{A_{h}} & (l+p=h+dn\text{ where }d\in \Z\text{ is even and }1\leq h\leq n-1),\\
\sigma_{A_{n-h}}^{-1} & (l+p=h+dn\text{ where }d\in \Z\text{ is odd and }1\leq h\leq n-1).
\end{cases} 
\end{array}
\]

Thus, by Lemma \ref{lemma:expanding-out-relations-for-symmetric-bands}, for each $i$ there is a sequence $m_{i}^{d}$ ($d\in\Z$) such that:  by parts (i) and (ii) we have $(m_{i}^{d},m_{i}^{d-1})\in C_{d}$ for all $d$; and by parts (iii) and (iv), for the automorphisms defined above we have
\[
\begin{array}{ccc}
    m_{i}^{2r+1-d}=\sum_{j}(\alpha_{d+1}\dots\alpha _{r})(\varphi_{ij})m_{j}^{d} & (d\leq r) & \text{where } \alpha_{r+1}\dots\alpha_{r}\text{ is the identity, and} \\[-1em]\\
    m_{i}^{-1-2p-d}=\sum_{j}(\alpha_{1-p}\dots\alpha _{d})^{-1}(\lambda_{ij})m_{j}^{d} & (d\geq -p) & \text{where } \alpha_{1-p}\dots\alpha_{-p}\text{ is the identity.}
\end{array}
\]
%
That the formula $f_{j}(b_{d})=m^{d}_{j}$ defines an $R$-module homomorphism $M(C)\to M$ follows from (the proof of) Lemma \ref{lemma:hom-functors-relation-functors}(ii). 
For the remaining statement in the lemma  we assert that $xf_{i}=\sum_{j}\pi^{-1}_{r}(\varphi_{ij})f_{j}$. 

For any $d\in\Z$ we have $(\sum_{j}\pi^{-1}_{r}(\varphi_{ij})f_{j})(b_{d})=\sum_{j}\pi_{d}(\pi^{-1}_{r}(\varphi_{ij}))m^{d}_{j}$, 
and for any $k>0$ we have $C_{r-k+1}=C_{r+k+1}^{-1}$ and so $\alpha_{r-k+1}=\alpha^{-1}_{r+k+1}$. 
Firstly let $d\leq r$. 
For simplicity assume $d<r$, and let $k=r-d>0$. 

Consider the composition of  automorphisms given by $\zeta_{k}=\alpha_{r-k+1}\dots \alpha_{r}=\alpha_{r+k+1}^{-1}\dots \alpha_{r+2}^{-1}$. 

In this notation we have $\pi_{d}=\zeta_{k}\pi_{r}$ 
and $m_{i}^{r+k+1}=\sum_{j}\zeta_{k}(\varphi_{ij})m_{j}^{r-k}$.  
Thus
\[\sum\pi_{d}(\pi^{-1}_{r}(\varphi_{ij}))m^{d}_{j}=\sum\zeta_{k}(\varphi_{ij})m^{r-k}_{j}=m_{i}^{r+k+1}=m_{i}^{2r+1-(r-k)}=f_{i}(b_{r-k}x)=(xf_{i})(b_{d}),
\]
as required for the case $d\leq r$.
Let us instead consider the case where $d>r$, and for simplicity (as above) we assume  that $d>r+1$. 
Here we have $d=r+k+1$ for some $k>0$. 
This gives
\[
f_{i}(b_{d}x)=\pi_{d}(\pi^{-1}_{r}(\beta_{s}))m^{d}_{i}-\pi_{2r+1-d}(\pi^{-1}_{r}(\gamma_{s}))m_{i}^{2r+1-d}.
\]

Note that $\pi_{d}=\alpha^{-1}_{d}\pi_{d-1}$, $\pi_{d-1}=\alpha^{-1}_{d-1}\pi_{d-2}$, and so on, through to $\pi_{r+1}=\alpha^{-1}_{r+1}\pi_{r}$. 
Hence $\pi_{d}=\nu_{k}\alpha_{r+1}^{-1}\pi_{r}$ where $\nu_{k}=\alpha^{-1}_{d}\dots \alpha^{-1}_{r+2}$. 
In particular, this gives $\pi_{d}\pi_{r}^{-1}=\nu_{k}\sigma_{s}$, and combining with the above yields $f_{i}(b_{d}x)=\nu_{k}(\beta_{s})m_{i}^{d}-\nu_{k}(\gamma_{s})m_{i}^{r-k}$. 
By a similar argument to the one above, we have that $m^{d}_{i}=\sum_{j}\nu_{k}(\varphi_{ij})m^{r-k}_{j}$. 
Now applying $\nu_{k}$ to $\sigma_{s}(\Phi)\Phi-\beta_{s}\Phi+\gamma_{s} I=O$, our assertion follows. 

A similar argument shows that $yf_{i}=\sum_{j}\pi^{-1}_{-p}(\lambda_{ij})f_{j}$, and the lemma follows. 
\end{proof}

\begin{lemma}

\label{symmetric-band-splitting-lifting-morphisms}

Let $w$ be a symmetric band with $r(x)=x^{2}-\beta x+\gamma$ in the notation above Lemma \ref{lemma:a-pair-of-matrices-gives-a-P-submodule}. 
Let $(h_{1},\dots,h_{n})$ be a $K$-basis for $F_{C}(M)$ with   $xh_{i}=\sum_{j}\varphi_{ij}h_{j}$ for each $i$. The following statements hold.
\begin{enumerate}
    \item Lifting $h_{i}$ to $k_{i}\in T_{C}(M)$ gives $xg_{i}=\beta g_{i}-\sum_{j} \sigma(\varphi_{ij})g_{j}$ where $g_{i}=xk_{i}-\sum_{j}\varphi_{ij}k_{j}$ for each~$i$.
    \item There exist lifts $f_{1},\dots,f_{n}\in T_{C}(M)$ of $h_{1},\dots,h_{n}$  such that $xf_{i}=\sum_{j}\varphi_{ij}f_{j}$ for each~$i$.
\end{enumerate}
\end{lemma}

\begin{proof}
(i) Using that $T_{C}(M)$ is a $R_{w}$-module we have $r(x)k_{i}=0$ for all $i$. 
The matrix $\Phi=(\varphi_{ij})\in M_{n}(K)$ satisifies $\rho(\Phi)\Phi=\beta\Phi-\gamma I$ by Remark \ref{remark:for-splitting-suffices-to-find-nice-lifts}. 
The claim follows by writing $r(x)k_{i}$ in terms of $g_{i}$ and $k_{j}$. 

(ii) By Lemma \ref{lemma:technical-symmetric-band-splitting-matrices} we have $\Phi\Xi+ \rho(\Xi)( \rho(\Phi)-\beta I)=I$ for some $\Xi\in M_{n}(K)$. 
By part (i) it suffices to let $f_{i}=k_{i}+\sum_{j}\xi_{ij}g_{j}$ where $\Xi=(\xi_{ij})$ and each $g_{i}\in B_{C}(M)$ is defined by $g_{i}=xk_{i}-\sum_{j} \varphi_{ij}k_{j}$.
\end{proof}

For the proof of Lemma \ref{lemma:splitting-symmetric-bands} we apply Lemma \ref{symmetric-band-splitting-lifting-morphisms}(ii) directly. 
We then apply Lemma \ref{symmetric-band-splitting-lifting-morphisms}(i) indirectly, exchanging $\rho$,  $r(x)$ and $\Phi=(\varphi_{ij})$ with $\tau$,  $p(y)$ and $\Lambda=(\lambda_{ij})$ respectively.


\begin{lemma}\label{lemma:splitting-symmetric-bands}
Let $w$ a symmetric band. If $M$ is finite-dimensional over $K$ then $B_C(M)$ has a  complement in $T_C(M)$ as an $R_{w}$-module.
\end{lemma}

\begin{proof}

The notation above Lemma \ref{lemma:a-pair-of-matrices-gives-a-P-submodule} recalls: the form of the canonically associated walk $C$ in terms of $s,t\in \Sp$ and walks $D,E$ of length $p,r\geq 0$; and the free product $R_{w}$ over $K$ in terms of automorphisms $\rho, \tau$ of $K$ and monic, normal and non-singular quadratics
\[
r(x)=x^{2}-\beta x+\gamma\in K[x;\rho],\quad p(y)=y^{2}-\mu y+\eta\in K[y;\tau]
\]
where $\beta=\pi_{r}^{-1}(\beta_{s})$, $\gamma=\pi_{r}^{-1}(\gamma_{s})$, $\mu=\pi_{-p}^{-1}(\beta_{t})$ and $\eta=\pi_{-p}^{-1}(\gamma_{t})$. 

Let $(v_{1},\dots,v_{n})$ be a $K$-basis for $F_{C}(M)$. 
For each $i=1,\dots,n$ we have $xv_{i}=\sum_{j}\varphi_{ij}v_{i}$ for some $\varphi_{ij}\in K$. 
By Remark \ref{remark:for-splitting-suffices-to-find-nice-lifts} the matrix $\Phi=(\varphi_{ij})$ is invertible, and  satisfies  $\rho(\Phi)\Phi-\beta\Phi+\gamma I=O$. 
By Lemma \ref{symmetric-band-splitting-lifting-morphisms}(ii) we can choose lifts $f_{1},\dots,f_{n}\in T_{C}(M)$ of $v_{1},\dots,v_{n}$ where $xf_{i}=\sum_{j}\varphi_{ij}f_{j}$ for each $i$.

Now let $yf_{i}=\sum_{j}\lambda_{ij}f_{j}+g_{i}$ for some $g_{i}\in B_{C}(M)$ and some $\Lambda=(\lambda_{ij})\in M_{n}(K)$. 
As above, $\Lambda$ is invertible and satisfies $\tau(\Lambda)\Lambda-\mu\Lambda+\eta I=O$, and $yg_{i}=\mu g_{i}-\sum_{j}\tau(\lambda_{ij})g_{j}$ by Lemma~\ref{symmetric-band-splitting-lifting-morphisms}(i).

Let $A=DE$. 
For each $i$ and each $d$ with $0\leq d\leq r+p$, let $f_{i}^{d}=f_{i}(b_{d-p})$ and $g_{i}^{d}=g_{i}(b_{d-p})$. 
Consider the relations $X=s^{-1}$ and $Y=A^{-1}t^{-1}A$ on $M$. 
By definition, and then by Lemma \ref{lemma:rewriting-relations-symmetric-bands}, we have that
\[
g_{i}^{p+r}\in B_{C,r}(M)=(XY)''\cap (YX)'+(XY)'\cap (YX)''=(XY)'\cap (YX)''
\]
since $M$ is finite-dimensional. 
Let $q=1+p+r$. 
Since $g_{i}^{p+r}$ lies in $(XY)'\cap (YX)''$, for each $i$ there is a sequence $(h^{d}_{i}\mid d\in\N)$ in $M$ such that  (a), (b) and (c) below hold. For (b) and (c) we fix $k\in \N$. 

(a) $h^{d}_{i}=g^{d}_{i}$ for all $d\leq  r+p$, and  $h^{d}_{i}=0$ for all but finitely many $d>r+p$.

(b) $h_{i}^{2kq+d-1}\in A_{d}h_{i}^{2kq+d}$ for $1\leq d\leq q-1$, and $sh_{i}^{(2k+1)q-1}=h_{i}^{(2k+1)q}$

(c) $h_{i}^{2kq+d}\in A_{2q-d}h_{i}^{2kq+d-1}$ for $q+1\leq d\leq 2q-1$, and $th_{i}^{2(k+2)q-1}=h_{i}^{2(k+2)q}$.

For each integer $d$ with $0\leq d\leq q-1$ let $\sigma_{d}=\sigma_{A_{d}}$ and define the automorphism $\zeta_{d}$ of $K$ by setting $\zeta_{0}$ to be the identity, and letting $\zeta_{d}=\sigma^{-1}_{d}\dots\sigma^{-1}_{1}$ for $d>0$. 
Considering the automorphisms $\pi=\pi_{-p}^{-1}$ and $\zeta=\pi_{r}^{-1}\zeta_{p+r}$, by Lemma \ref{lemma:symmetric-splitting-sequence-of-matrices} one yields a sequence $(\Theta_{k}\mid k\in\N)$ in $M_{n}(K)$ satisfying: {}

\eqref{eq1} 
$\sigma_{t}(\Theta_{0})(\beta_{t}I-\sigma_{t}(\pi_{-p}(\Lambda))))=-I+\pi_{-p}(\Lambda) \Theta_{0}$; and for each $k\in\N$, 

\eqref{eq2} 
$\pi_{-p}(\Lambda)\Theta_{2k+2}=\beta_{t}\Theta_{2k+2}+\sigma_{t}(\Theta_{2k+1})$;

\eqref{eq3} 
$\pi_{-p}(\Lambda)\Theta_{2k+1}=-\sigma_{t}(\Theta_{2k+1})\gamma_{t}$;

\eqref{eq4} 
$\pi_{r}(\Phi )\zeta_{p+r}(\Theta_{2k}) =-\sigma_{s}(\zeta_{p+r}(\Theta_{2k+1}))\gamma_{s}$;
and 

\eqref{eq5} 
$\pi_{r}(\Phi)\zeta_{p+r}(\Theta_{2k+1})=\sigma_{s}(\zeta_{p+r}(\Theta_{2k}))+\beta_{s}\zeta_{p+r}(\Theta_{2k+1})$.

Let $\Theta_{k}=({}^{k}\theta_{ij})$ for each $k$. 
For any $d$ with $0\leq d\leq r+p$ let
\[
m_{i}^{d-p}=f^{d}_{i}+\sum_{j}\sum_{k}\left(\zeta_{d}({}^{2k}\theta_{ij})h^{2kq+d}_{j}+\zeta_{d}({}^{2k+1}\theta_{ij})h^{2(k+1)q-(d+1)}_{j}\right)
\]
which defines an element of $M$ since $h^{d}_{j}=0$ for $d\gg 0$ by (a). 
For $0<d\leq p+r$ the relation $A_{d}$ on $M$ is $\sigma_{d}$-semilinear, and so $(m^{d-p}_{i},m^{d-p-1}_{i})\in A_{d}$ by (b) and (c) above.
%
We claim (i) and (ii) below hold.
\[
\begin{array}{cc}
\text{(i) } sm^{r}_{i}=\sum_{j}\pi_{r}(\varphi_{ij})m^{r}_{j}, & 
\text{(ii) } tm^{-p}_{i}=\sum_{j}\pi_{-p}(\lambda_{ij})m^{-p}_{j}.
\end{array}
\]

Since $sb_{r}=b_{r}x$ we have $sf^{p+r}_{i}=\sum_{j}\pi_{r}(\varphi_{ij})f^{p+r}_{j}$. 
For (i) we require $s(m^{r}_{i}-f^{p+r}_{i})=\sum_{l=1}^{n}\pi_{r}(\varphi_{il})(m^{r}_{l}-f^{p+r}_{l})$.
Recall $sh^{(2k+1)q-1}_{i}=h^{(2k+1)q}_{i}$ for each $i,k$ by (b). 
By Corollary \ref{corollary:twisting-(and-inverting-non-singular)-preserves-normal/central-quadratics}, $r(x)$ is normal, so
\[
s\zeta_{p+r}({}^{2k+1}\theta_{ij})h^{2(k+1)q-q}_{j}
=\beta_{s}\zeta_{p+r}({}^{2k+1}\theta_{ij})h^{(2k+1)q}_{j}-\sigma_{s}(\zeta_{p+r}({}^{2k+1}\theta_{ij}))\gamma_{s}h^{(2k+1)q-1}_{j}
\] 
by Lemma \ref{lemma:characterising_central_skew_quadratics}. 
Similarly $s\zeta_{p+r}({}^{2k}\theta_{ij})h^{(2k+1)q-1}_{j}=\sigma_{s}(\zeta_{p+r}({}^{2k}\theta_{ij}))h^{(2k+1)q}_{j}$ by (b). 
Taking the sum over $j$ and $k$, (i) follows from a straightforward application of \eqref{eq4} and \eqref{eq5}. 


Recall $yf_{i}=\sum_{j}\lambda_{ij}f_{j}+g_{i}$. 
Since $tb_{-p}=b_{-p}y$ we have $tf^{0}_{i}=\sum_{j}\pi_{-p}(\lambda_{ij})f^{0}_{j}+g^{0}_{i}$. 
So, for (ii) we require $t(m^{-p}_{i}-f^{0}_{i})+g^{0}_{i}=\sum_{l}\pi_{-p}(\lambda_{il})(m^{-p}_{l}-f^{0}_{l})$. 
By (c) we have $th^{2(k+2)q-1}_{j}=h^{2(k+2)q}_{j}$ and so
\[
t({}^{2(k+1)}\theta_{ij}h^{2(k+1)q}_{j})=\beta_{t}({}^{2(k+1)}\theta_{ij}h^{2(k+1)q}_{j})-\sigma_{t}({}^{2(k+1)}\theta_{ij})\gamma_{t}h^{2(k+1)q-1}_{j}
\]
by Lemma \ref{lemma:characterising_central_skew_quadratics}. 
Recall that $yg_{j}=\mu g_{j}-\sum_{l}\tau(\lambda_{jl})g_{l}$. 
Note that for the case $k=0$ we have that $h^{0}_{j}=g^{0}_{j}$, and so using that $tb_{-p}=b_{-p}y$ one can show $th_{j}^{0}=\beta_{t}h_{j}^{0}-\sum_{l}\sigma_{t}(\pi_{-p}(\lambda_{jl}))h_{l}^{0}$.
We also have $t({}^{2k+1}\theta_{ij}h^{2(k+1)q-1}_{j})=\sigma_{t}({}^{2k+1}\theta_{ij})h^{2(k+1)q}_{j}$ for any $k$ by (c). 

From here a straightforward application of \eqref{eq1}
gives
\[
t(m^{-p}_{i}-f^{0}_{i})=
-
g_{i}^{0}+\left(\sum_{l}\left(\sum_{j}\pi_{-p}(\lambda_{ij}){}^{0}\theta_{jl}\right)g^{0}_{l}\right)+\left(\sum_{j,k}t({}^{2k+2}\theta_{ij}h^{2(k+1)q}_{j})+\sigma_{t}({}^{2k+1}\theta_{ij})h^{2(k+1)q}_{j}\right).
\]

Now (ii) follows from \eqref{eq2} and  \eqref{eq3}. 
Hence, by defining 
$\Lambda'=\pi_{r}(\Lambda)$ and $\Phi'=\pi_{-p}(\Phi)$, we have  $sm^{r}_{i}=\sum_{j}\varphi_{ij}'m^{r}_{j}$ and $tm^{-p}_{i}=\sum_{j}\lambda_{ij}'m^{-p}_{j}$. 
The proof now follows from  Lemma \ref{lemma:a-pair-of-matrices-gives-a-P-submodule}. 
\end{proof}

\subsection{Evaluation of functors on modules}
\label{s:evaluation-functors-on-modules}
Let $w$ be an $I$-indexed string or a band and let $\epsilon$ be a sign. For each $i\in I$ exactly one of the words $(w_{\leq i})^{-1}$ or $w_{>i}$ has sign $\epsilon$. Denote this word by $w(i,\epsilon)$. Recall $J_{w}$ is a subset of $I$ such that the elements $b_{i}$ ($i\in J_{w}$) define a free right $R_{w}$-module basis for $M(C_w)$.

\begin{lemma}\label{lemma:representing-by-canonical-b-i}
Let $w$ be a string or band. For any $i\in I$ there is some $j\in J_{w}$ and some unit $z\in R_{w}$ such that $b_{i}=b_{j}z$, $w(i,1)=w(j,1)$ and $w(i,-1)=w(j,-1)$.
\end{lemma}

\begin{proof}
If $w$ is an asymmetric string then $I=J_{w}$ and there is nothing to prove. Choose the sign $\epsilon$ such that $w(i,\epsilon)=w_{>i}$ and $w(i,-\epsilon)=(w_{\leq i})^{-1}$. The remainder of the proof is split into cases.

(i) 
Suppose $w$ is a symmetric string, say $w=us^{-1}u^{-1}$ for some $\{0,\dots,k\}$-indexed word $u$. 
Here $R_{w}=K[x;\sigma]/\ideal{q(x)}$ for a normal, non-singular and monic quadratic $q(x)$.
By Lemma \ref{lemma:invertible-variable-in-quotient-by-non-singular-quadratic} we have that $x$ is a unit in $ R_{w}$. If $i\leq k$ let $j=i$ and $z=1$. 
If $i>k$ then $n-i\leq n-k-1=k$, and we let $j=n-i$ and $z=x$.  
Since $w=w^{-1}$ we have $(w_{\leq i})^{-1}=w_{>n-i}$ and $w_{>i}=(w_{\leq n-i})^{-1}$ which means $w(n-i,\pm\epsilon)=w(i,\pm\epsilon)$.

(ii) Suppose $w$ is an asymmetric band, say $w={}^{\infty}u^{\infty}$ for some $\{0,\dots,p\}$-indexed word $u$. 
Here $R_{w}=K[x,x^{-1};\tau]$. Writing $i=j+np$ for some integer $n$ and some remainder $j\in\{0,\dots,p-1\}=J_{w}$ we now let $z=x^{-n}$. 
Since $w=w[pn]$ we have $(w_{\leq i})^{-1}=(w_{\leq j})^{-1}$ and $w_{>i}=w_{>j}$ which means $w(j,\pm\epsilon)=w(i,\pm\epsilon)$.

(iii) Suppose $w$ is a symmetric band. 
By Lemma \ref{lemma:form-of-symmetric-bands}, $w = {}^\infty (v s^* v^{-1} u^{-1} t^* u)^\infty$ for $s,t\in \Sp$ and words $u,v$ of lengths $p,r\geq 0$ respectively, such that $n = p+r+1$ where $2n$ is the period of $w$.
Likewise we have that the set of integers $i$ with $w[i]=(w[i])^{-1}$ is the coset $r+1+n\Z=-p+n\Z$. 
Here we have that $R_{w}$ is the free product over $K$ of semisimple quotients $K[x;\rho]/\ideal{r(x)}$ and $K[y;\tau]/\ideal{p(y)}$ by monic, non-singular and normal quadratics $r(x)$ and $p(y)$. 
As in part (i), by Lemma \ref{lemma:invertible-variable-in-quotient-by-non-singular-quadratic} the variables $x$ and $y$ define units in $R_{w}$. 

Recall that by Theorem \ref{theorem:symmetricbandparameterringisHNP} the ring $R_{w}$ has as a $K$-basis the alternating monomials in $x$ and $y$, and that $J_{w}=\{-p,\dots,r\}$. 
Recall the notation from the proof of Lemma \ref{lemma:symmetric-bands-R_w-basis}. 
Here alternating monomials $z_{l}$ ($l\in\Z$) in $x$ and $y$ were defined in such a way so that $b_{h} z_{l} = b_{l n+h}$ for $l$ even and $b_{h} z_{l} = b_{l n + r-p-h}$ for $l$ odd.
Write $i=mn+k$ for some remainder $k\in \{0,\dots,n-1\}$. We now consider cases.

(iiia) Suppose $k\leq r$ and $m$ is even. 
Let $j=k$ and $z=z_{m}$. 
By construction $b_{i}=b_{j}z$ and $j\in J_{w}$. 
Since $m$ is even we have $w=w[mn]$ since $w$ has period $2n$, which gives $w_{\leq i}=w_{\leq j}$ and $w_{>i}=w_{> j}$. 

(iiib) Suppose $k\leq r$ and $m$ is odd. Let $j=r-p-k$ and $z=z_{m}$.
Since $k\leq r$ we have $j\geq -p$, and since $k\geq 0$ we have $j\leq r$, which means $j\in J_{w}$ and $b_{i}=b_{j}z$. 
Since $m-1$ is even and $i=(m-1)n+n+r-p-j$ we have $w_{>i}=w_{>n+r-p-j}$ since $w$ has period $2n$. 
Since $w[-p]=(w[-p])^{-1}$ we have $w_{d-p}=w_{-p-d}^{-1}$ for any $d\in\Z$. 
Considering the cases where $d\geq n+r-j+1$ we have $w_{>n+r-p-j}=(w_{\leq j-2n})^{-1}$, which again by periodicity is just $(w_{\leq j})^{-1}$, which shows altogether that $(w_{\leq j})^{-1}=w_{>i}$.
Similarly, $(w_{\leq i})^{-1}=w_{>j}$.

If $k> r$, the proof for $m$ odd  (respectively, even) is similar to (iiia) (respectively, (iiib)). 
%
%
%
%
%
\end{proof}

Given any word $u\in H(\ell,\epsilon)$ we let $I_{w}^{+}(u)$ be the set of $i\in I$ such that $v_{i}(w)=\ell$ and $w(i,\epsilon)\leq u$.

\begin{lemma}\label{lemma:evaluation-b-i-is-in-the-trivial-plus}
Let $w$ be an $I$-indexed string or band, let $\ell$ be a vertex and let $\epsilon$ be a sign. Then for any $i\in I_{w}^{+}(1_{\ell,\epsilon})$ we have $b_{i}\in 1_{\ell,\epsilon}^{+}(M(C_w))$.
\end{lemma}

\begin{proof}
It suffices to assume there is an ordinary arrow $a$ with tail $\ell$ such that the sign of $a^{-1}$ is $\epsilon$, since otherwise $1_{\ell,\epsilon}^{+}(M(C_w))=e_{\ell}M(C_w)$ and there is nothing to prove. 
Hence we have $1_{\ell,\epsilon}^{+}(M(C_w))=a^{-1}(0)$ for some such $a$. 
Since $i\in I_{w}^{+}(1_{\ell,\epsilon})$ either $w(i,\epsilon)=1_{\ell,\epsilon}$ or $w(i,\epsilon)=cv$ for some arrow $c$ and some $v\in H(\ell',\epsilon')$. 

When $w(i,\epsilon)=1_{\ell,\epsilon}$ we have that either ($i-1\notin I$ and $aC_{w}$ is a walk) or ($i+1\notin I$ and $C_{w}^{-1}a^{-1}$ is a walk), which means $ab_{i}=0$. 
When $w(i,\epsilon)=cv$ there is some $j\in I$ with $b_{i}=cb_{j}$, and since $a$ is ordinary and the letters $a^{-1}$ and $c$ have sign $\epsilon$, $ac$ must be a zero relation defining the algebra. 
\end{proof}

Given any word $u\in H(\ell,\epsilon)$ we let  $I_{w}^{-}(u)$ be the set of $i\in I$ such that $v_{i}(w)=\ell$ and  $w(i,\epsilon)< u$. 

Using similar arguments to those used in Lemma \ref{lemma:evaluation-b-i-is-in-the-trivial-plus}, we have the following.

\begin{lemma}\label{lemma:evaluation-b-i-is-not-in-the-trivial-minus}
Let $w$ be an $I$-indexed string or band, let $\ell$ be a vertex and let $\epsilon$ be a sign. Then for any $R_{w}$-module $V$ we have $1_{\ell,\epsilon}^{-}(M(C_w)\otimes V)\subseteq \sum b_{i}\otimes V$ where the sum runs over all $i\in I_{w}^{-}(1_{\ell,\epsilon})$.
\end{lemma}

\begin{lemma}\label{lemma:evaluating-minus-on-strings-bands-inductive-step-technical}
Let $w$ be an $I$-indexed string or a band and let $u\in H(\ell,\epsilon)$ be a word of the form $u=x^{*}u'$ for some word $u'\in H(\ell',\epsilon')$ and some letter $x$. Let $i\in I_{w}^{-}(u')$.
\begin{enumerate}
    \item If $x=a\notin \Sp$ and $ab_{i}\neq 0$ then $ab_{i}=b_{h}$ for some $h\in \{i-1,i+1\}\cap I_{w}^{-}(u)$. 
    \item If $x=s\in\Sp$ then $sb_{i}=b_{k}$ or \emph{(}$i\in I_{w}^{-}(u)$ and $sb_{i}=\beta_{s}b_{i}-\gamma_{s}b_{k}$\emph{)} for some $k\in\{i-1,i+1\}\cap I_{w}^{-}(u)$.
\end{enumerate}
\end{lemma}

\begin{proof}
(i) Let $C=C_{w}$ and $L$ be the set of $j\in I_{w}^{-}(u')$ such that $w(j,-\epsilon')=a^{-1}v$ for some $*$-word $v$. 
For any $j\in  I_{w}^{-}(u')\setminus L$ we have $w(j,-\epsilon')\leq 1_{\ell',-\epsilon'}$ and so $b_{j}\in 1^{+}_{\ell',-\epsilon'}(M(C))$ by Lemma \ref{lemma:evaluation-b-i-is-in-the-trivial-plus}.
Since $1^{+}_{\ell',-\epsilon'}(M(C))=a^{-1}(0)$ we have that $i\in L$.
Define $h\in \{i+1,i-1\}$ by ($h=i-1$ if $w(i,-\epsilon')=(w_{\leq i})^{-1}$) and ($h=i+1$ if $w(i,-\epsilon')=w_{>i}$). 
By construction this gives $w(h,\epsilon)=aw(i,\epsilon')<au'=u$ and $ab_{i}=b_{h}$.

(ii) Since the sign of $s^{*}$ is $\epsilon$ we cannot have that $w(i,\epsilon)$ is trivial, for otherwise we would have ($w(i,\epsilon)=(w_{\leq i})^{-1}$ and $s^{*}w$ is a word) or ($w(i,\epsilon)=w_{>i}$ and $ws^{*}$ is a word), in either case contradicting Lemma \ref{lem-end-adm}. 
Hence we have $w(i,\epsilon)=y v$ for some letter $y$ of sign $\epsilon$, which gives 
$y=s^{*}$. 

Define $k\in I$ and $z\in \{s,s^{-1}\}$ by ($k=i-1$ and $z=C_{i}^{-1}$ if $w(i,\epsilon)=(w_{\leq i})^{-1}$) and ($k=i+1$ and $z=C_{i+1}$ if $w(i,\epsilon)=w_{>i}$).
Here we have $w(k,\epsilon)=s^{*}w(i,-\epsilon)<s^{*}u'=u$ and $w(k,-\epsilon)=v$. 
If $z=s^{-1}$ then $sb_{i}=b_{k}$, as required. If instead $z=s$ then $v\leq w(i,-\epsilon)$, $sb_{k}=b_{i}$ and $i,k\in I_{w}^{+}(u)$, since $s^{*}v\leq s^{*}w(i,-\epsilon)$. 
\end{proof}
Given any word $u\in H(\ell,\epsilon)$, we let $J_{w}^{\pm}(u)=J_{w}\cap I_{w}^{\pm}(u)$.
\begin{lemma}\label{lemma:evaluation-b-i-is-not-in-the-minus-inductive-step}

Let $w$ be an $I$-indexed string or a band, $V$ be an $R_{w}$-module and $u\in H(\ell,\epsilon)$ where $u=x^{*}u'$ for some $u'\in H(\ell',\epsilon')$ and some letter $x$ which is not of the form $s^{-1}$ for some $s\in \Sp$. 
Then, 
\[
x\left(\sum_{i\in J_{w}^{-}(u')}b_{i}\otimes_{R_{w}} V\right)\subseteq  \sum_{j\in J_{w}^{-}(u)}b_{j} \otimes_{R_{w}}V.
\]
\end{lemma}

\begin{proof}
Suppose $x=a$ for $a$ ordinary. 
For each $i\in J_{w}^{-}(u')$ with $ab_{i}\neq 0$ we have that $ab_{i}=b_{h}$ for some $h\in \{i-1,i+1\}\cap I_{w}^{-}(u)$ by Lemma \ref{lemma:evaluating-minus-on-strings-bands-inductive-step-technical}(i). 
By Lemma \ref{lemma:representing-by-canonical-b-i} this gives $a(b_{i}\otimes V)= b_{j}\otimes V$ for some $j\in J_{w}^{-}(u)$, since $h\in I_{w}^{-}(u)$. 
Hence the inclusion holds in case $x=a$. 
Similarly one can show the inclusion holds when $x=s$ by applying Lemmas \ref{lemma:representing-by-canonical-b-i} and \ref{lemma:evaluating-minus-on-strings-bands-inductive-step-technical}(ii). 

It remains to consider $x=a^{-1}$ for $a$ ordinary. 
Let $L$ be the set of $i\in J_{w}$ with $v_{i}(w)=\ell$. 
Let $m$ lie in the left hand side of the required inclusion, written as a sum $m=\sum b_{i}\otimes z_{i}$ over $i\in L$ where each $z_{i}\in V$. 

Let $L'$ be the set of $h\in J_{w}^{-}(u)\subseteq L$ such that $w(h,\epsilon)=a^{-1}v$ for some word $v\in H(\ell',\epsilon')$, and note that for each such $v$ we have $v<u'$. 
Let $m'$ be the sum of the terms $b_{h}\otimes z_{h}$ as $h$ runs through $L'$. 

For each $h\in L'$ with $w(h,\epsilon)=a^{-1}v$ for some $v$, if we let ($k=h-1$ when $w(h,\epsilon)=(w_{\leq h})^{-1}$) and ($k=h+1$ when $w(h,\epsilon)=w_{>h}$) then $v=w(k,\epsilon')$ and $ab_{h}=b_{k}$. 
This shows that for any $h\in L'$ there exists $k\in I_{w}^{-}(u)$ such that $ab_{h}=b_{k}$, and hence  $ab_{h}\otimes V=b_{j}\otimes V$ for some $j\in J_{w}^{-}(u)$ by Lemma \ref{lemma:representing-by-canonical-b-i}.

Let $W=\bigoplus b_{i}\otimes V$ where the sum runs over $i\in J_{w}^{-}(u')$. The argument above shows that $am'\in W$, and $am\in W$ by assumption. 
Using $L''=J_{w}^{-}(u)\setminus L'$ and $L'''=L\setminus J_{w}^{-}(u)$ one may define elements $m'',m'''$ as above with $m=m'+m''+m'''$. 
Noting that $L''=J_{w}^{+}(1_{\ell,\epsilon})$ we have $am''=0$ by Lemma \ref{lemma:evaluation-b-i-is-in-the-trivial-plus},  and hence $am'''\in W$. 
By a straightforward application of Lemma \ref{lemma:representing-by-canonical-b-i}, this means $m'''=0$, as required. 
\end{proof}


\begin{lemma}\label{lemma:evaluation-b-i-can-not-be-in-the-minus}
Let $w$ be a string or a band, let $u\in H(\ell,\epsilon)$ and let $V$ be an $R_{w}$-module. 
Then we have
\[
D_{u}^{-}(M(C_w)\otimes_{R_{w}} V)\subseteq \sum _{i\in J_{w}^{-}(u)}b_{i}\otimes_{R_{w}} V.
\]
\end{lemma}

\begin{proof}
Let $D=D_{u}$, say $J$-indexed. 
By Lemma \ref{lemma:evaluation-b-i-is-not-in-the-trivial-minus} we can assume $J\neq \{0\}$. Let $u(n)=u_{>n}$ ($n\in J$) and $V_{n}=\sum_{i\in J_{w}^{-}(u(n))}b_{i}\otimes V$.  Then $D_{n+1} V_{n+1}\subseteq V_{n}$ by Lemma \ref{lemma:evaluation-b-i-is-not-in-the-minus-inductive-step}. 
Let $m\in D^{-}(M(C_w)\otimes V)$. Let $l$ be the length of $u$ when $J\neq \N$, and when $J=\N$ choose $l\geq 0$ with $m\in D_{\leq l}(0)$. 
By Lemma \ref{lemma:evaluation-b-i-is-not-in-the-trivial-minus} we have $m\in D_{\leq l} V_{l}$. 
Combining the inclusions $D_{n+1} V_{n+1}\subseteq V_{n}$ for each $n=0,\dots l-1$ gives $m\in V_{0}$.
\end{proof}

\subsection{Completion of the proof}
Let $\Omega$ be a set of representatives of the equivalence classes of strings and bands. If $w$ is an end-admissible word,  say $I$-indexed, let $P_w$ be the set of pairs $\{u,v\}$ where $u$ and $v$ are words with $u^{-1}v=w$. Observe that $u,v$ must have the same head and opposite signs. Equivalently 
\[
P_w = \{ \{ (w_{\le i})^{-1},w_{>i} \} \mid i\in I \}.
\]
For $w$ a string or band, recall  there is an associated non-empty finite subset $J_w$ of $I$.

\begin{lemma}
\label{lemma:pairsofwordsproperties}
Let $w$ be an end-admissible word and $\epsilon=\pm1$.
\begin{itemize}
\item[(i)]
If $w$ is not finite or periodic, then $P_w$ is infinite.
\item[(ii)]
If $w$ is a string or band, then $|P_w| = |J_w|$.
\end{itemize}
\end{lemma}

\begin{proof}
(i) If $w$ is $I$-indexed, then each $i\in I$ gives a pair $\{ (w_{\le i})^{-1},w_{>i}\}$. Suppose $P_w$ is finite, but $I$ is infinite. Then infinitely many of the words $(w_{\le i})^{-1}$ must have the same sign, and then at least two of the corresponding pairs $\{ (w_{\le i})^{-1},w_{>i}\}$ and $\{ (w_{\le j})^{-1},w_{>j}\}$ must be equal. But then $(w_{\le i})^{-1} = (w_{\le j})^{-1}$ and $w_{>i} = w_{>j}$, so $w$ is periodic.

(ii) The elements in $P_w$ are exactly given by the pairs $\{ (w_{\le i})^{-1},w_{>i}\}$ with $i\in J_w$.
\end{proof}

\begin{theorem}
\label{theorem:locfiniteandreflectisos}
Let $M,M'$ be finite-dimensional $R$-modules and $\theta:M\to M'$ an $R$-module homomorphism.
\begin{itemize}
\item[(i)]
We have
\[
\dim_K M = \sum_{w\in\Omega} |J_w| \, \dim_K F_{C_w}(M),
\]
and in particular $F_{C_w}(M)=0$ for all but finitely many $w\in \Omega$. 
\item[(ii)]
If $F_{C_w}(\theta)$ is an isomorphism for all $w\in \Omega$, then $\theta$ is an isomorphism.
\end{itemize}
\end{theorem}

\begin{proof}
(i) Fix a sign $\epsilon$ and a vertex $\ell$ in $Q$. For $u\in H(\ell,\epsilon)$ and $v\in H(\ell,-\epsilon)$, we define
\[
\begin{array}{cc}
U_{u,v}(M) = (D'_{u})^{-}(M) + D_{v}^{+}(M) \cap (D'_{u})^{+}(M), 
 & 
L_{u,v}(M) = (D'_{u})^{-} (M)+ D_{v}^{-}(M) \cap (D'_{u})^{+}(M).
\end{array}
\]
These are the `upper' and `lower' functors in a `two-sided $\epsilon$-filtration' of the forgetful functor at $\ell$, obtained by refining the $D'_u$ filtration against the $D_v$ filtration. By Lemma~\ref{lemma:one-sided-filtration} we have the following properties.

(a) $L_{u,v}(M) \subseteq U_{u,v}(M)$,

(b) if $(u,v)\neq (u',v')$ then either $U_{u,v}(M) \subseteq L_{u',v'}(M)$ or $U_{u',v'}(M)\subseteq L_{u,v}(M)$,

(c) if $S$ is a non-empty subset of $e_\ell M$ with $0\notin S$, then there is some pair $(u,v)$ 
such that $S$ meets $U_{u,v}(M)$ and does not meet $L_{u,v}(M)$.

Namely, for (c), choose $u$ such that $S$ meets $(D'_u)^+(M)$ but not $(D'_u)^-(M)$. Then the set $S'$ defined by
\[
S' = (D'_u)^-(M) + (S\cap (D'_u)^+(M)) = ((D'_u)^-(M) + S) \cap (D'_u)^+(M)
\] 
is non-empty and doesn't contain 0, so there is some $v$ such that $S'$ meets $D_v^+(M)$ but not $D_v^-(M)$. It follows that $S$ meets $U_{u,v}(M)$ but not $L_{u,v}(M)$.

Taking $K$-bases for the quotients $U_{u,v}(M)/L_{u,v}(M)$, lifting them to elements of $e_\ell M$, and combining them, properties (a)--(c) ensure that we obtain a basis for $e_\ell M$. Thus
\[
\dim_K e_\ell M = \sum_{(u,v)} \dim_K \left( \frac{U_{u,v}(M)}{L_{u,v}(M)} \right).
\]

Let $w$ be an end-admissible word, say $I$-indexed, and let $i\in I$. Let $u = (w_{\le i})^{-1}$ and $v= w_{>i}$ and suppose that $u$ has head $\ell$ and sign $\epsilon$. Then the functors $U_{u,v}$ and $L_{u,v}$ belong to the two-sided $\epsilon$ filtration of the forgetful functor at $\ell$, and by Zassenhaus' Lemma and  Lemma~\ref{lemma:hom-functors-relation-functors} we have isomorphisms
\[
U_{u,v}/L_{u,v} \cong F_{D'_u,D_v} \cong F_{D_w,i} \cong {}_{\pi_i^{-1}} F_{D_w}.
\]
Similarly the functors $U_{v,u}$ and $L_{v,u}$ belong to the two-sided $(-\epsilon)$-filtration of the forgetful functor at $\ell$ and 
\[
U_{v,u}/L_{v,u} \cong F_{D_u,D'_v} \cong F_{D'_w,i} \cong {}_{\pi_i^{-1}} F_{D'_w}.
\]

Considering the quotients $U_{u,v}/L_{u,v}$ which arise in the $\epsilon$- and $(-\epsilon)$-filtrations of the forgetful functors for all vertices $\ell$ in $Q$, we find $|P_w|$ copies of twists of $F_{D_w}$ and $|P_w|$ copies of twists of $F_{D'_w}$. 

Now if $w$ is not finite or periodic, then $P_w$ is infinite by Lemma~\ref{lemma:pairsofwordsproperties}(i), so $F_{D_w}(M)= F_{D'_w}(M) = 0$. 
By Theorem \ref{theorem:orientations-main-result}, if $w$ is finite or periodic and end-admissible then $F_{C_w} \cong F_{D_w} \cong F_{D'_w}$. If $w$ is end-admissible, and if $z$ is an equivalent word, then it is easy to see that there is some $\sigma$ such that $F_{C_z} \cong {}_\sigma F_{C_w}$.

Now as $w$ runs through representatives of the equivalence classes of end-admissible words, the correspondence mentioned above uses up all quotients in the two-sided $\epsilon-$ and $(-\epsilon)$-filtrations. Thus
\[
\dim_K M = \sum |P_w| \, \dim_K F_{C_w}(M)
\]
where the sum is over end-admissible words $w$ which are finite or periodic. 

Next we show that $F_{C_w}=0$ if $w$ is not relation-admissible. Suppose $w$ contains $r^*$ for a relation $r$. Then $D_w$ contains $r$, and so $(D_w)_{>i} = r E$ for some walk $E$ and some $i\in I$. But then $((D_w)_{>i})^\pm(M) = r E^\pm(M) = 0$, so $F_{D_w,i}=0$, so $F_{D_w} = 0$, so $F_{C_w}=0$. Similarly, if $w$ contains $(r^*)^{-1}$, consider $D'_w$.

The result follows, using the formula for $\dim_K M$ above and Lemma~\ref{lemma:pairsofwordsproperties}(ii).

(ii) Suppose $F_{C_w}(\theta)$ is an isomorphism for all $w\in\Omega$.  Then $F_{D_w}(\theta)$ and $F_{D'_w}(\theta)$ are isomorphisms for all end-admissible words $w$. Thus $(U_{u,v}/L_{u,v})(\theta)$ is an isomorphism for all $u,v$. Now $\theta$ is a monomorphism. Namely, if $0\neq m\in e_\ell M$, then taking $S = \{m\}$, by property (c) there is some $(u,v)$ with $m$ inducing a non-zero element of $(U_{u,v}/L_{u,v})(M)$. But then $\theta(m)$ induces a non-zero element of $(U_{u,v}/L_{u,v})(M')$, so $\theta(m)\neq 0$. Now by dimensions $\theta$ is an isomorphism.
\end{proof}



\begin{lemma}
\label{lemma:evoffwonmcwtensorv}
Let $w$ be a string or band and $V$ a finite-dimensional $R_w$-module.
\begin{itemize}
\item[(i)]
The natural map $V\to F_{C_w}(M(C_w)\otimes_{R_w} V)$ is an isomorphism.
\item[(ii)]
If $z$ is a string or band which is not equivalent to $w$, then $F_{C_{z}}(M(C_w)\otimes_{R_w} V) = 0$. 
\end{itemize}
\end{lemma}

\begin{proof}
(i) Let $M = M(C_w)\otimes V$. By Theorem \ref{theorem:orientations-main-result} there is an element $j\in I$ such that either
\begin{itemize}
\item[(a)]
$T_{C_w,j} = T_{D_w,j}$ and $B_{C_w,j} = B_{D_w,j}$, or 
\item[(b)]
$T_{C_w,j} = T_{D'_w,j}$ and $B_{C_w,j} = B_{D'_w,j}$.
\end{itemize}
In particular for a string we take $j=0$ and we are in case (a). We deal with case (a). Case (b) is similar.

We have an $R_w$-module homomorphism $V\to \Hom_R(M(C_w),M) = T_{C_w}(M)$, sending $v\in V$ to $\theta_v$ defined by $\theta_v(m) = m\otimes v$ for $m\in M(C_w)$. This induces a map $V\to F_{C_w}(M)$, and by Lemma~\ref{lemma:hom-functors-relation-functors} it suffices to show that the composition
\[
V\to F_{C_w}(M) \to F_{C_w,j}(M)
\]
is an isomorphism. Now $\dim_K M = |J_w|.\dim_K V$, so Theorem~\ref{theorem:locfiniteandreflectisos} gives $\dim_K V \ge \dim_K F_{C_w,j}(M)$. Thus it suffices to prove that the composition is a monomorphism. Now the image of $\theta_v$ under the map $T_{C_w}(M) \to {}_{\pi_j} T_{C_w,j}(M)$ considered in Lemma~\ref{lemma:hom-functors-relation-functors} is $\theta_v(b_j) = b_j\otimes v$. Thus we need to show that $(b_j \otimes V ) \cap B_{C_w,j}(M) = 0$.

For a string, with $j=0$, we can compute the functor $B_{D_w,j}$ on $M$ and show that $b_j \otimes V \cap B_{D_w,j}(M) = 0$. 
Namely, if $w$ is a string with head $\ell$ and sign $\epsilon$,  then by Lemma \ref{lemma:evaluation-b-i-can-not-be-in-the-minus} any element of  $B_{D_w,0}(M)$ is a sum of terms of the form $b_{i}\otimes v_{i}$ with $i\in J_{w}$ such that  $w(i,-\epsilon)<1_{\ell,-\epsilon}$ or $w(i,\epsilon)<w$. 
Since $w(0,-\epsilon)=1_{\ell,-\epsilon}$ and $w(0,\epsilon)=w$, and since the $b_{i}$ with $i\in J_{w}$ give an $R_{w}$-basis for $M(C_{w})$, this means $b_j \otimes V \cap B_{D_w,j}(M) = 0$.

For a band, the walk $(D_w)_{>j}$ is of the form $E^\infty$ for some finite walk $E$. We think of $E$ as a relation. Then $B_{D_w,j} = E'' \cap (E^{-1})' + E' \cap (E^{-1})''$ which is the same as $E' \cap (E^{-1})''$ by Lemma \ref{lemma:rewriting-relations-symmetric-bands} since $\dim_{K}(M)<\infty$. Moreover $E' = ((D_w)_{>j})^-(M)$ satisfies $E'\cap b_j\otimes V=0$ by Lemma \ref{lemma:evaluation-b-i-can-not-be-in-the-minus}, as above. 

(ii) Follows from (i) by the dimension formula in Theorem~\ref{theorem:locfiniteandreflectisos}(i).
\end{proof}

The following result is similar to \cite[Theorem 9.1]{Cra2018}.

\begin{theorem}
Let $u$ be a string or a band. 
Let $M=\bigoplus M(C_{w})\otimes_{R_{w}}V_{w}$ as $w$ runs through a set $W$ of strings and bands, where each $R_{w}$-module $V_{w}$ is finite-dimensional. 
Then for each $w$ in $W_{u}=\{w\in W\mid w\sim u\}$ there is an automorphism  $\pi_{w}$ of $K$ such that $F_{C_{u}}(M)\cong \bigoplus {}_{\pi_{w}}V_{w}$ as $w$ runs through $W_{u}$. 
\end{theorem}
\begin{proof}
Let $C=C_{u}$. 
For any walk $D$ with $D^{*}\in H(\ell,\epsilon)$, the functors $D^{\pm}$ commute with arbitrary direct sums (see for example \cite[Lemma 6.1]{Cra2018}). 
Hence, and by Lemma  \ref{lemma:hom-functors-relation-functors}, the same is true for the functor $F_{C}\cong F_{C,0}$. 
When $w\sim u$ we have an automorphism $\pi_{w}$ with $F_{C}\cong {}_{\pi_{w}}F_{C_{w}}$. 
For each $w\in W$ let $M_{w}=M(C_{w})\otimes _{R_{w}}V_{w}$. 
By part (ii) of Lemma \ref{lemma:evoffwonmcwtensorv} we have that $F_{C}(M)$ is the direct sum of the $F_{C}(M_{w})$ with $u\sim w$. 
By part (i) we have that $F_{C_{w}}(M_{w})\cong V_{w}$ and hence $F_{C}(M_{w})\cong {}_{\pi_{w}}V_{w}$ for all $w\sim u$. 
\end{proof}

\begin{lemma}
\label{lemma:existsgamma}
Given a string or band $w$ and a finite-dimensional $R$-module $M$, there is an $R_w$-module homomorphism $\gamma:M(C_w)\otimes_{R_w} F_{C_w}(M) \to M$ such that $F_{C_w}(\gamma)$ is an isomorphism.
\end{lemma}

\begin{proof}
By Lemma \ref{lemma:B_C-is-an-R_w-submodule}, $B_{C_w}(M)$ is an $R_w$-submodule of $T_{C_w}(M)$. The splitting property says that  there is an $R_w$-complement $V$ with $T_{C_w}(M) = B_{C_w}(M)\oplus V$. 
For $w$ a string this property holds automatically, since then $R_{w}$ is a semisimple artinian ring. 
For $w$ a band, see Lemmas \ref{splitting-lemma-asymmetric-band} and \ref{lemma:splitting-symmetric-bands}. 

Hence we can identify $V$ as an $R_w$-module with $F_{C_w}(M)$. The inclusion of $V$ in $T_{C_w}(M)$ is an element of $\Hom_{R_w}(V,\Hom_R(M(C_w),M))$, so corresponds to an element $\gamma \in \Hom_R(M(C_w)\otimes_{R_w} V,M)$. 
Now, the composition of the following maps is the identity.
\[
V \to \Hom_R(M(C_w),M(C_w)\otimes_{R_w}V) = T_{C_w}(M(C_w)\otimes_{R_w}V)\xrightarrow{T_{C_w}(\gamma)} T_{C_w}(M) \to F_{C_w}(M) \cong V
\]
So, the map $T_{C_w}(M(C_w)\otimes_{R_w}V)\to F_{C_w}(M)$ is onto, hence so is 
$F_{C_w}(\gamma) : F_{C_w}(M(C_w)\otimes_{R_w}V)\to F_{C_w}(M)$. Now it is an isomorphism by dimensions, using Lemma~\ref{lemma:evoffwonmcwtensorv}.
%
%
%
%
%
\end{proof}

The main theorem now follows from the lemma at the start of section 3 of \cite{Rin1975}, using the indexing set $\Omega$, the categories of finite-dimensional $R$-modules and of finite-dimensional $R_w$-modules, and the functors $M(C_w)\otimes_{R_w}-$ and $F_{C_w}$, and applying Theorem~\ref{theorem:locfiniteandreflectisos} and Lemmas~\ref{lemma:evoffwonmcwtensorv} and~\ref{lemma:existsgamma}.


\section{Examples}
\label{s:examples}
\subsection{Dynkin and extended Dynkin species of classical type}
The indecomposable representations of quivers of type $A_n$ and $\tilde A_n$ are well-known, but using the fact that the path algebras are string algebras, it is sometimes useful to view them in terms of strings and bands. Using semilinear clannish algebras one can do the same thing for Dynkin and extended Dynkin species of classical type. 

For example the $\R$-subalgebra
\[
A = \begin{pmatrix}
\R & 0 & 0 & 0 \\
0 & \R & 0 & 0 \\
\R & \R & \R & 0 \\
\C & \C & \C & \C 
\end{pmatrix}
\]
of $M_4(\C)$ is a hereditary algebra of type 
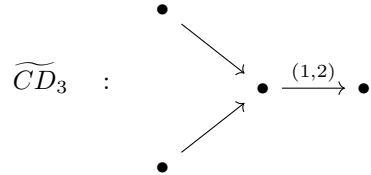
\[
\widetilde{CD}_3 \quad : \quad 
\begin{tikzcd}
	\bullet \\
	& \bullet & \bullet \\
	\bullet
	\arrow["{(1,2)}", from=2-2, to=2-3]
	\arrow[from=1-1, to=2-2]
	\arrow[from=3-1, to=2-2]
\end{tikzcd}
\]
in the sense of~\cite{DlabRin1976}. Let $R$ be the semilinear clannish algebra with $K=\R$, 
$Q$ of shape
\[
\begin{tikzcd}
1 \arrow[out=160,in=200,loop,swap,"t"] \arrow[r,"a"] 
& 2 \arrow[r,"b"]
& 3 \arrow[out=340,in=20,loop,swap,"s"]
\end{tikzcd}
\]
with all automorphisms in $\boldsymbol{\sigma}$ trivial and with $t$ and $s$ special where $q_t(x) = x^2-1$ and $q_s(x) = x^2+1$. There is an isomorphism of $\R$-algebras $R\to A$ given by
\[
e_1 \mapsto 
\begin{pmatrix}
1 & 0 & 0 & 0 \\
0 & 1 & 0 & 0 \\
0 & 0 & 0 & 0 \\
0 & 0 & 0 & 0 
\end{pmatrix},
\quad
e_2 \mapsto 
\begin{pmatrix}
0 & 0 & 0 & 0 \\
0 & 0 & 0 & 0 \\
0 & 0 & 1 & 0 \\
0 & 0 & 0 & 0 
\end{pmatrix},
\quad
e_3 \mapsto 
\begin{pmatrix}
0 & 0 & 0 & 0 \\
0 & 0 & 0 & 0 \\
0 & 0 & 0 & 0 \\
0 & 0 & 0 & 1 
\end{pmatrix},
\]
\[
t \mapsto 
\begin{pmatrix}
1 & 0 & 0 & 0 \\
0 & -1 & 0 & 0 \\
0 & 0 & 0 & 0 \\
0 & 0 & 0 & 0 
\end{pmatrix},
\quad
a \mapsto 
\begin{pmatrix}
0 & 0 & 0 & 0 \\
0 & 0 & 0 & 0 \\
1 & 1 & 0 & 0 \\
0 & 0 & 0 & 0 
\end{pmatrix},
\quad
b \mapsto 
\begin{pmatrix}
0 & 0 & 0 & 0 \\
0 & 0 & 0 & 0 \\
0 & 0 & 0 & 0 \\
0 & 0 & 1 & 0 
\end{pmatrix},
\quad
s \mapsto 
\begin{pmatrix}
0 & 0 & 0 & 0 \\
0 & 0 & 0 & 0 \\
0 & 0 & 0 & 0 \\
0 & 0 & 0 & i 
\end{pmatrix}.
\]

\medskip

Another example: the $\R$-subalgebra 
\[
B = \begin{pmatrix}
\R & 0 & 0 \\
\C & \C & 0 \\
\HH & \HH & \HH 
\end{pmatrix}
\]
of $M_3(\HH)$, where $\HH=\R\oplus\R i\oplus \R j\oplus \R k$ is the algebra of quaternions, is a hereditary algebra of type 
\[
\widetilde{BC}_2 \quad : \quad 
\begin{tikzcd}
	\bullet & \bullet & \bullet
	\arrow["{(1,2)}", from=1-1, to=1-2]
	\arrow["{(1,2)}", from=1-2, to=1-3]
\end{tikzcd}
\]
in the sense of~\cite{DlabRin1976}. The $\R$-subalgebra
\[
B' = \begin{pmatrix}
\R & \R & 0 & 0 \\
\R & \R & 0 & 0 \\
\C & \C & \C & 0 \\
\HH & \HH & \HH & \HH 
\end{pmatrix}
\]
of $M_4(\HH)$ is Morita equivalent to $B$, since 
\[
e = 
\begin{pmatrix}
0 & 0 & 0 & 0 \\
0 & 1 & 0 & 0 \\
0 & 0 & 1 & 0 \\
0 & 0 & 0 & 1 
\end{pmatrix}
\]
is an idempotent in $B'$ with $B' e B' = B'$ and $e B' e \cong B$. Let $R$ be the semilinear clannish algebra with $K=\C$, $Q$ as in the first example, $\sigma_t = \sigma_s$ the conjugation automorphism of $\C$, $\sigma_a=\sigma_b=1$, $s$ and $t$ special, $q_t(x) = x^2-1$ and $q_s(x) = x^2+1$. There is an $\R$-algebra isomorphism $R\to B'$ given by
\[
i \mapsto 
\begin{pmatrix}
0 & -1 & 0 & 0 \\
1 & 0 & 0 & 0 \\
0 & 0 & i & 0 \\
0 & 0 & 0 & i 
\end{pmatrix},
\quad
e_1 \mapsto 
\begin{pmatrix}
1 & 0 & 0 & 0 \\
0 & 1 & 0 & 0 \\
0 & 0 & 0 & 0 \\
0 & 0 & 0 & 0 
\end{pmatrix},
\quad
e_2 \mapsto 
\begin{pmatrix}
0 & 0 & 0 & 0 \\
0 & 0 & 0 & 0 \\
0 & 0 & 1 & 0 \\
0 & 0 & 0 & 0 
\end{pmatrix},
\quad
e_3 \mapsto 
\begin{pmatrix}
0 & 0 & 0 & 0 \\
0 & 0 & 0 & 0 \\
0 & 0 & 0 & 0 \\
0 & 0 & 0 & 1 
\end{pmatrix},
\]
\[
t \mapsto 
\begin{pmatrix}
1 & 0 & 0 & 0 \\
0 & -1 & 0 & 0 \\
0 & 0 & 0 & 0 \\
0 & 0 & 0 & 0 
\end{pmatrix},
\quad
a \mapsto 
\begin{pmatrix}
0 & 0 & 0 & 0 \\
0 & 0 & 0 & 0 \\
1 & i & 0 & 0 \\
0 & 0 & 0 & 0 
\end{pmatrix},
\quad
b \mapsto 
\begin{pmatrix}
0 & 0 & 0 & 0 \\
0 & 0 & 0 & 0 \\
0 & 0 & 0 & 0 \\
0 & 0 & 1 & 0 
\end{pmatrix},
\quad
s \mapsto 
\begin{pmatrix}
0 & 0 & 0 & 0 \\
0 & 0 & 0 & 0 \\
0 & 0 & 0 & 0 \\
0 & 0 & 0 & j 
\end{pmatrix}.
\]

As for types $A_n$ and $\tilde A_n$, the classification of the indecomposables for these hereditary algebras is already known, see \cite{DlabRin1976}, but the examples are instructive, as they provide building blocks for the construction of more complicated algebras.


\subsection{Dedekind-like rings}
If $K$ is a field, then the prototypical string algebra $K[x,y]/\ideal{xy}$ is a `Dedekind-like' ring whose unique singular maximal ideal is `strictly split' in the sense of \cite[Definition 10.1 and Theorem and Definition 11.3]{KliLev2005}.

On the other hand, if $K/F$ is a field extension of degree 2, then the subring $A = F + x K[x]$ of $K[x]$ is a Dedekind-like ring whose unique singular maximal ideal is `unsplit'. Writing $K = F(\omega)$ where $\omega^2 = p+q\omega$ with $p,q\in F$, and setting $y = \omega x$, we have $A\cong F[x,y]/\ideal{y^2 - qxy - px^2}$. 

Suppose now that $K/F$ is a separable extension. Then it is Galois, with group $\{1,\sigma\}$, where $\sigma(\omega) = q-\omega$. Let $R$ be the semilinear clannish algebra given by the field $K$, the quiver with a single vertex and loops $a$ and $t$, with $\sigma_a = \sigma_t = \sigma$, $t$ special with $q_t(x)=x^2-1$ and $a$ ordinary with the zero relation $a^2=0$. Thus $R = F\langle \omega,a,t\rangle/I$ where $I$ is given by the relations $\omega^2-q\omega-p=0$, $a^2 = 0$, $t^2 = 1$, $a \omega = (q-\omega) a$, $t \omega = (q-\omega) t$. The $F$-algebra homomorphism $R\to M_2(A)$ given by
\[
\omega\mapsto \begin{pmatrix} 0 & p \\ 1 & q \end{pmatrix}, \quad
a\mapsto \begin{pmatrix} -\omega x & -\omega^2 x \\ x & \omega x \end{pmatrix}, \quad
t\mapsto \begin{pmatrix} 1 & q \\ 0 & -1 \end{pmatrix}, 
\]
is easily seen to be an isomorphism. Thus by Morita equivalence, the classification of finite-dimensional indecomposable modules for $R$ gives a classification of finite-dimensional indecomposable modules for $A$. 
%
%
Note that Klingler and Levy~\cite{KliLev2001} have also given a classification of indecomposables for unsplit Dedekind-like rings.



%
Noncommutative analogues of Dedekind-like rings have been studied by Drozd \cite{Dpn} and with the name `nodal algebras' by Burban and Drozd \cite{BurDro2004}; see also \cite{Zembyk}. Assuming that the base field is algebraically closed, Burban and Drozd classify modules, and more generally objects in the derived category, by reducing to a Bondarenko matrix problem \cite{Bon1991}. It would be interesting to explore examples for other base fields using our semilinear clannish algebras (or `semilinear clans', as mentioned at the end of the introduction).

\subsection{Modular representations of the alternating group $A_4$}
Let $k$ be a field of characteristic 2, not containing a primitive cube root of unity, let $K=k(\omega)$ where $\omega^2+\omega+1=0$, and let $\sigma\in\Aut_k(K)$ be the automorphism with $\sigma(\omega) = \omega^2$. Since the group algebra $kA_4$ is self-injective, any finite-dimensional indecomposable $kA_4$-module which is not projective is a module for the quotient $A$ of $kA_4$ by its socle, and this is Morita equivalent to the semilinear clannish algebra $R$ given by the field $K$ and quiver
\[
\begin{tikzcd}
1 \arrow[out=160,in=200,loop,swap,"c"]
\arrow[r,bend left,"a"]
& 2 \arrow[out=340,in=20,loop,swap,"s"]
\arrow[l,bend left,"b"]
\end{tikzcd}
\]
with $s$ special, $\sigma_a=1$, $\sigma_b=\sigma_c=\sigma_s=\sigma$, $q_s(x)=x^2-1$, and zero-relations $ab,ac,ba,cb,c^2$. Namely, letting 
\[
A' = \bigg\{ \begin{pmatrix}
y & 0 & 0 & u_1 & u_2 & w \\
0 & x_1 & x_2 & 0 & 0 & v_1 \\
0 & x_3 & x_4 & 0 & 0 & v_2 \\
0 & 0 & 0 & x_1 & x_2 & 0 \\
0 & 0 & 0 & x_3 & x_4 & 0 \\
0 & 0 & 0 & 0 & 0 & \sigma(y)
\end{pmatrix}\in M_6(K) \biggm\vert x_i\in k, \ y,u_i,v_i,w\in K \biggr\}
\]
and $e=\diag(1,1,0,0,1,1)$, then $A'eA' = A'$ and $eA'e \cong A$ by the description of $A$ in \cite[\S2]{DlaRin1989}. Writing $E_{ij}$ for the elementary matrix with a 1 at position $(i,j)$, we have a $k$-algebra isomorphism $R\to A'$ given by 
\[
\omega\mapsto
\begin{pmatrix}
\omega & 0 & 0 & 0 & 0 & 0 \\
0 & 0 & 1 & 0 & 0 & 0 \\
0 & 1 & 1 & 0 & 0 & 0 \\
0 & 0 & 0 & 0 & 1 & 0 \\
0 & 0 & 0 & 1 & 1 & 0 \\
0 & 0 & 0 & 0 & 0 & \omega^2 
\end{pmatrix},
\quad
s\mapsto
\begin{pmatrix}
0 & 0 & 0 & 0 & 0 & 0 \\
0 & 1 & 1 & 0 & 0 & 0 \\
0 & 0 & 1 & 0 & 0 & 0 \\
0 & 0 & 0 & 1 & 1 & 0 \\
0 & 0 & 0 & 0 & 1 & 0 \\
0 & 0 & 0 & 0 & 0 & 0 
\end{pmatrix},
\]
$e_1\mapsto \diag(1,0,0,0,0,1)$, $e_2\mapsto \diag(0,1,1,1,1,0)$,
$a\mapsto E_{26}+\omega^2 E_{36}$, $b\mapsto E_{14}+\omega^2 E_{15}$, $c\mapsto E_{16}$.

\subsection{Algebras arising from surfaces with orbifold points of order 2}
Triangulations of surfaces lead to gentle algebras, which are a special case of string algebras.
In the work of Geuenich and Labardini \cite{GeuenichLabardini1}, orbifold points of order 2 are also allowed, and this leads to new algebras involving field extensions. 

For example if the relevant field extension is $\C/\R$ and the triangulation includes a triangle which contains two orbifold points, as in \cite[Definition 5.2(2)]{GeuenichLabardini1}, the building block for the algebra is given by a species with potential, but it is isomorphic to the clannish $\R$-algebra given by the quiver
\[
\begin{tikzcd}
1 \arrow[out=140,in=180,loop,swap,"s"]
\arrow[rr,"a"]
& & 2 \arrow[out=0,in=40,loop,swap,"t"]
\arrow[dl,"c"]
\\
&
\arrow[ul,"b"]
3
&
\end{tikzcd}
\]
with relations $ab = 0$, $bc = 0$, $ca = 0$ and $s$ and $t$ special with polynomial $x^2+1$. We are grateful to Jan Geuenich and Daniel Labardini-Fragoso for help with this example.

\subsection{Dieudonn\'e modules}
Let $K$ be a perfect field of characteristic $p>0$. The ring $W(K)$ of Witt vectors is a discrete valuation ring with maximal ideal $\ideal{p}$ and residue field $K$. The Frobenius automorphism $\sigma$ of $K$ lifts to an automorphism $\hat\sigma$ of $W(K)$. A \emph{Dieudonn\'e module} is a $W(K)$-module $M$ equipped with a $\hat\sigma$-semilinear map $F:M\to M$ and a ${\hat\sigma}^{-1}$-semilinear map $V:M\to M$ satisfying $FV = VF = p 1_M$, see \cite[V,\S1,3.1]{DemGab1970}. Such modules appear in connection with the classification of finite group schemes. 

Clearly the Dieudonn\'e modules annihilated by $p$ are representations of a semilinear string algebra over $K$ given by the quiver with one vertex and ordinary loops $V$ and $F$, with $\sigma_F = \sigma$, $\sigma_V = \sigma^{-1}$ and  $VF = FV = 0$.


\subsection{Existence of $F$-Crystals}
To study the existence of $F$-crystals, Kottwitz and Rapoport \cite{KotRap2003} are led to consider configurations of vector spaces and mappings which amount to representations of the semilinear string algebra $R$ given by an algebraically closed field $K$ of characteristic $p>0$, the quiver 
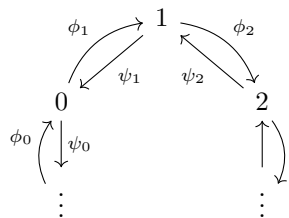
\[
\begin{tikzcd}
& 1 \ar[rd, bend left, "\phi_2"] \ar[ld, "\psi_1"] \\
0 \ar[ur, bend left, "\phi_1"] \ar[d, "\psi_0"] & & 2 \ar[d, bend left] \ar[ul, "\psi_2"]  \\
\vdots \ar[u, bend left, "\phi_0"] & & \vdots \ar[u] 
\end{tikzcd}
\]
(the `double' of a quiver of type $\tilde A_n$) with all automorphisms $\sigma_{\phi_i}$ and $\sigma_{\psi_i}$ being integer powers (positive, negative or zero) of the Frobenius automorphism, and with the zero relations $\psi_i \phi_i=0$ and $\phi_i \psi_i = 0$. Ringel \cite{Rin-lec} used the string and band classification of indecomposable representations of $R$ to give a new proof of Theorem 6.1 of \cite{KotRap2003}, which states that if $V$ is a non-zero representation of $R$ such all the vector spaces $V_i$ have the same dimension, then $V$ has a subrepresentation $W$ such that all $W_i$ have dimension~1.

\bibliographystyle{abbrv}
\bibliography{BTandCB-semilinear-clannish-algebras-arxiv-v3.bib}
\addtocontents{toc}{\protect\setcounter{tocdepth}{0}}
\end{document}